\documentclass[final]{amsart} 
\usepackage[utf8]{inputenc}
\usepackage[english]{babel}
\usepackage{amssymb,amsthm}
\usepackage{mathtools}\mathtoolsset{showonlyrefs}
\usepackage[final]{microtype}
\usepackage[inline]{enumitem}
\usepackage{mathrsfs}
\usepackage[bottom=1in]{geometry}
\usepackage{tikz-cd}
\usepackage{aliascnt}
\usepackage{csquotes}
\usepackage{appendix}
\usepackage[pagebackref,final]{hyperref}
\usepackage[inline]{showlabels}
\AtBeginDocument{
\let\oldref\ref
\renewcommand{\ref}[1]{\mbox{\oldref{#1}}%
}}
\renewcommand*{\backref}[1]{}
\renewcommand*{\backrefalt}[4]{%
\ifnum#1=1 %
\ifnum#3=1 %
Cited on p. %
\else
Cited on p. %
\fi
\else
Cited on pp. %
\fi
#2,\par
}

\newcommand{\comment}[1]{\textcolor{red}{#1}}
\renewcommand{\comment}[1]{}

\tikzcdset{every diagram/.style={sep=large}}

\newcommand{\ZZ}{\mathbb{Z}}
\newcommand{\QQ}{\mathbb{Q}}

\newcommand{\CC}{\mathbb{C}}
\newcommand{\FF}{\mathbb{F}}

\newcommand{\RRR}{\mathcal{R}} 
\newcommand{\CCC}{\mathcal{C}}
\newcommand{\NNN}{\mathcal{N}}

\newcommand{\VVV}{\mathcal{V}}
\newcommand{\PPP}{\mathcal{P}} 

\newcommand{\K}{\mathrm{K}} 
\newcommand{\D}{\mathrm{D}} 
\newcommand{\R}{\mathrm{R}} 
\renewcommand{\L}{\mathrm{L}} 
\renewcommand{\H}{\mathrm{H}} 
\newcommand{\X}{\mathrm{X}} 

\newcommand{\field}{\mathfrak{F}}

\newcommand{\red}{\mathrm{red}}

\newcommand{\Chain}{\mathrm{C}} 
\newcommand{\Weyl}{\mathrm{W}} 
\newcommand{\Nm}{\mathrm{Nm}} 

\newcommand{\mon}{\mathsf{mon}}
\newcommand{\pf}{\mathsf{pf}}
\newcommand{\rpf}{\mathsf{rpf}}
\newcommand{\lpf}{\mathsf{lpf}}
\newcommand{\groth}{\mathsf{groth}} 
\newcommand{\wirth}{\mathsf{wirth}}
\newcommand{\ltm}{\mathsf{ltm}}
\newcommand{\assoc}{\mathsf{assoc}}
\newcommand{\sym}{\mathsf{sym}}

\newcommand{\one}{\mathbf{1}}
\newcommand{\blank}{\raisebox{-2pt}{$-$}}
\newcommand{\normal}{\trianglelefteq} 
\newcommand{\subgroup}{\le}
\newcommand{\curlyle}{\preccurlyeq}

\newcommand{\eg}{\textit{e.g.}}
\newcommand{\ie}{\textit{i.e.}}

\newcommand{\opcit}{\textit{op.\,cit.}\ }

\renewcommand{\setminus}{\smallsetminus}
\renewcommand{\emptyset}{\varnothing}

\newcommand{\Fil}{\mathcal{F}\kern-1.5pt{\mathit{il}} } 

\newcommand{\too}{\longrightarrow}
\newcommand{\To}{\Rightarrow}
\newcommand{\Too}{\Longrightarrow}
\newcommand{\raiso}{\stackrel{\cong}{\too}}
\newcommand{\laiso}{\stackrel{\cong}{\longleftarrow}}
\newcommand{\Raiso}{\stackrel{\cong}{\Too}}

\newcommand{\ol}[1]{\overline{#1}}
\newcommand{\ul}[1]{\underline{#1}}
\newcommand{\alg}[1]{\mathbf{#1}}
\newcommand{\set}[2]{\left\{#1\,\middle|\,#2\right\}}
\newcommand{\wt}[1]{\widetilde{#1}}

\newcommand{\cat}[1]{\mathscr{#1}}
\newcommand{\orbitrep}[2]{\NNN_{#1,#2}}
\newcommand{\lie}[1]{\mathfrak{#1}}

\newcommand{\longtwoheadrightarrow}%
{\longrightarrow\hspace{-1.2em}\rightarrow\hspace{.2em}}
\renewcommand{\twoheadrightarrow}%
{\rightarrow\hspace{-1.2em}\rightarrow\hspace{.2em}}
\newcommand{\longhookrightarrow}%
{\lhook\joinrel\relbar\joinrel\rightarrow}

\DeclareMathOperator{\Ad}{Ad}
\DeclareMathOperator{\Aut}{Aut}
\DeclareMathOperator{\cInd}{c-Ind}

\DeclareMathOperator{\Ext}{Ext}

\DeclareMathOperator{\GL}{GL}

\DeclareMathOperator{\Lie}{Lie}

\DeclareMathOperator{\nInd}{\boldsymbol{i}} 
\DeclareMathOperator{\nPRes}{\boldsymbol{r}} 
\DeclareMathOperator{\pInd}{\mathit{i}} 
\DeclareMathOperator{\cosoc}{cosoc}

\DeclareMathOperator{\JH}{JH}
\DeclareMathOperator{\Inf}{Inf}
\DeclareMathOperator{\val}{val}
\DeclareMathOperator{\Sp}{Sp} 

\DeclareMathOperator{\height}{ht}

\DeclareMathOperator{\ev}{ev}
\DeclareMathOperator{\coev}{coev}

\DeclareMathOperator{\Res}{Res}
\DeclareMathOperator{\Ord}{Ord}

\DeclareMathOperator{\Hom}{Hom}
\DeclareMathOperator{\RHom}{RHom}
\DeclareMathOperator{\iHom}{\ul{Hom}}
\DeclareMathOperator{\id}{id}
\DeclareMathOperator{\Nat}{Nat}
\DeclareMathOperator{\Ind}{Ind}
\DeclareMathOperator{\RInd}{RInd}
\DeclareMathOperator{\ind}{ind}
\DeclareMathOperator{\Supp}{Supp}

\DeclareMathOperator{\pr}{pr}
\DeclareMathOperator{\Rep}{Rep}

\swapnumbers
\newaliascnt{main}{subsubsection}
\newtheorem{prop}[main]{Proposition}
\newtheorem{thm}[main]{Theorem}
\newtheorem{lem}[main]{Lemma}
\newtheorem{cor}[main]{Corollary}

\newtheorem{propintro}{Proposition}
\newtheorem{thmintro}[propintro]{Theorem}

\theoremstyle{definition}
\newtheorem{defn}[main]{Definition}

\newtheorem{notation}[main]{Notation}

\newtheorem{ex}[main]{Example}
\newtheorem{rmk}[main]{Remark}

\newtheorem*{defn*}{Definition}
\newtheorem*{claim*}{Claim}

\theoremstyle{remark}
\newtheorem*{rmk*}{Remark}
\newtheorem*{ex*}{Example}
\newtheorem*{notation*}{Notation}

\numberwithin{equation}{subsection}
\makeatletter
\renewcommand{\@secnumfont}{\bfseries}
\def\section{\@startsection{section}{1}%
  \z@{.7\linespacing\@plus\linespacing}{.5\linespacing}%
  {\Large\normalfont\bfseries\centering}}
\let\c@equation\c@subsubsection

\makeatother

\tikzcdset{proof/.style={sep=normal,font=\tiny,labels={font=\tiny}} } 

\newenvironment{hideproof}{\noindent\textit{Proof.}}{\hfill\qedsymbol}
\newif\ifhideproofs
\ifhideproofs
	\usepackage{environ}
	\NewEnviron{hide}{}
	
\fi

\newcounter{hideeq}[main]
\newenvironment{hideeq}{\refstepcounter{hideeq}\equation}%
{\tag{\themain.\thehideeq}\endequation}

\title[The Geometrical Lemma]{The Geometrical Lemma for Smooth Representations in Natural Characteristic}
\author{Claudius Heyer}
\thanks{The project was funded by the Deutsche
Forschungsgemeinschaft (DFG, German Research Foundation) – Project-ID 427320536
– SFB 1442, as well as under Germany's Excellence Strategy EXC 2044 390685587,
Mathematics Münster: Dynamics–Geometry–Structure.}
\address{Mathematisches Institut, Westf\"alische Wilhelms-Universit\"at
M\"unster, Einsteinstra\ss{}e 62, D-48149 M\"unster, Germany}
\email{cheyer@uni-muenster.de}
\subjclass[2020]{11F85, 18G80, 20G25}

\begin{document}
\begin{abstract} 
The Geometrical Lemma is a classical result in the theory of (complex) smooth
representations of $p$-adic reductive groups, which helps to analyze the
parabolic restriction of a parabolically induced representation by providing a
filtration whose graded pieces are (smaller) parabolic inductions of parabolic
restrictions. In this article, we establish the Geometrical Lemma for the
derived category of smooth mod $p$ representations of a $p$-adic reductive
group. 

As an important application we compute higher extension groups between
parabolically induced representations, which in a slightly different context had
been achieved by Hauseux assuming a conjecture of Emerton
concerning the higher ordinary parts functor. We also compute the (cohomology
functors of the) left adjoint of derived parabolic induction on principal series
and generalized Steinberg representations.
\end{abstract} 
\maketitle
\tableofcontents

\section{Introduction}
\subsection{History and motivation} 
Fix a finite extension $\field/\QQ_p$ and let $G$ be (the group of
$\field$-points of) a connected reductive $\field$-group. In the theory of
complex smooth representations of $G$, the Geometrical Lemma, independently due
to Bernstein--Zelevinsky \cite{Bernstein-Zelevinski.1977} and Casselman
\cite{Casselman.1995}, is one of the main tools in the classification of
irreducible smooth representations in terms of parabolically induced
representations. To state it, we fix a parabolic subgroup $P=
M\ltimes U$ of $G$ with Levi quotient $M$ and unipotent radical
$U$. Denote by $\nInd_{P}^{G} \colon \Rep_{\CC}(M)\to
\Rep_{\CC}(G)$ the normalized parabolic induction functor on the categories of
smooth representations. Its left adjoint is the normalized parabolic restriction
functor $\nPRes^{G}_{P}$. If $Q = L\ltimes N$ is another parabolic
subgroup with Levi quotient $L$ and unipotent radical $N$ such that $P\cap Q$ contains a fixed minimal parabolic subgroup, then the Geometrical
Lemma states that the functor $\nPRes^{G}_{Q} \nInd_{P}^{G}$ admits a filtration
by subfunctors such that the graded pieces are given by functors of the form
\[
\nInd_{g^{-1}Pg\cap L}^{L} \circ
g_*^{-1} \circ \nPRes^{M}_{M\cap
gQg^{-1}},
\]
where $g_*^{-1} \colon \Rep_{\CC}\bigl(M\cap gLg^{-1}\bigr) \raiso
\Rep_{\CC}\bigl(g^{-1}Mg\cap L\bigr)$ is the equivalence of categories induced
by conjugation with $g^{-1}$ and $g$ runs through a certain set of double coset
representatives of $P\backslash G/Q$. All known proofs rely on the use of Haar
measures and the exactness of $\nPRes^{G}_{P}$. They are easily adapted to prove
a Geometrical Lemma for smooth representations over a field $k$ of
characteristic $\neq p$. 

However, if $k$ is a field of characteristic $p$, then the proofs break down
completely. Since there is no modulus character, one cannot talk about
normalization. We denote by $\pInd_{P}^{G} \colon \Rep_k(M)\to \Rep_k(G)$ the
unnormalized parabolic induction functor and $\L^0(U,\blank)$ its
left adjoint. In this context, Haar measures do not exist and
$\L^0(U,\blank)$ is not exact; the Geometrical Lemma takes on a
much simpler but also less satisfactory form: the functor $\L^0(N,\blank)
\pInd_{P}^{G}$ is isomorphic to $\pInd_{P \cap L}^{L} \L^0(M\cap N,\blank)$, see
\cite[Theorem~5.5]{AHV.2019}; in other words, there is a filtration with only
one graded piece. The reason for this pathological behaviour comes down to the
fact that there exists no non-trivial Haar measure on $U$.

As it turns out, there is a full Geometrical Lemma once we pass to the derived
categories.

\subsection{Main results} 
From now on, let $k$ be a field of characteristic $p$. For any $p$-adic Lie
group $H$ we denote by $\D(H)$ the unbounded derived category of the category
$\Rep_k(H)$ of smooth $k$-linear representations of $H$. With the notation from above the parabolic induction extends to a derived functor $\R\!\pInd_P^G\colon \D(M)\to \D(G)$. 
By the main result of \cite{Heyer.2022}, there exists a left adjoint
$\L(U,\blank)$. It should be noted that we need to assume that $\field$ has characteristic zero. If $\field$ has positive characteristic, then $G$ is not a $p$-adic Lie group and the methods of \opcit do not apply; in particular the existence of $\L(U,\blank)$ remains unknown.

To state the
Geometrical Lemma, we fix a set $\orbitrep{P}{Q}$ of double coset
representatives of $P\backslash G/Q$ which normalizes a maximal $\field$-split
torus (of $G$) that is contained in $P\cap Q$. For the notion of filtration on a
triangulated functor, we refer to Definition~\ref{defn:filtration}. If $H$ is a $p$-adic Lie group, we denote by $\dim H$ its dimension as a $p$-adic manifold in the sense of \cite{Schneider.2011}.

\begin{thmintro}[Corollary~\ref{cor:geometrical}] 
\label{thmintro:geometrical}
The functor $\L(N,\blank)\circ \R\!\pInd_P^G \colon \D(M)\to \D(L)$ admits a
filtration of length $\lvert\orbitrep{P}{Q}\rvert$ with graded pieces of the
form
\[
\R\!\pInd_{n^{-1}Pn\cap L}^L \circ (\omega_n\otimes_k\blank) \circ n_*^{-1}
\circ \L(M\cap nNn^{-1},\blank),
\]
for $n\in \orbitrep{P}{Q}$, where $\omega_n \in \D(n^{-1}Mn\cap L)$ is a
character in cohomological degree $-\dim(n^{-1}\ol{U}n\cap N)$ and
$\ol{U}$ is the unipotent radical of the parabolic opposite $P$.
\end{thmintro} 

To fix ideas, we note that, if one chooses $\orbitrep{P}{Q}$ carefully and if
$G$ is $\field$-split, then $\omega_n$ is concentrated in degree
$-[\field:\QQ_p]\ell(w)$, where $\ell(w)$ denotes the length of the image $w$ of
$n$ in the (finite) Weyl group. Thus, $\omega_n$ contributes a cohomological
shift which is not detected on the abelian categories; this gives a conceptual
explanation of why there is no proper Geometrical Lemma for (underived) smooth
mod $p$ representations. The twist by the character $\omega_n$ is not surprising
as it occurs also in the classical context although it is hidden in the
normalization of the parabolic induction and restriction functors.

The proof of Theorem~\ref{thmintro:geometrical} follows the general strategy
employed by Bernstein--Zelevinsky and Casselman in that the problem is reduced
to checking certain compatibilities
between compact induction and (derived) coinvariants. Since we do not have Haar
measures at our disposal (which are the primary tool in the classical setting),
it is not possible to write down the required isomorphisms by hand. In comparison, our proof is more conceptual. As a non-formal input the proof uses that the derived inflation functor
$\R\Inf^M_P\colon \D(M)\to \D(P)$ is fully faithful, which follows from the fact
that $U$ is a unipotent group, see \cite[Example~3.4.24]{Heyer.2022}. Hence, the
same proof will apply in other contexts as well.

A general method to determine explicitly the characters $\omega_n$, for $n\in
\orbitrep{P}{Q}$, is presented in Propositions~\ref{prop:duality-character}
and~\ref{prop:explicit-character}. This is applied in Lemma~\ref{lem:delta_w}
to deduce a concrete description of the characters $\omega_n$.

As an application, we will compute several $\Ext$-groups between parabolically
induced representations. These are virtually identical with the main results of
\cite{Hauseux.2016, Hauseux.2018}; the important differences are that Hauseux
computes higher extensions in the category of locally admissible representations
and relies for the strongest of these results on an open conjecture of Emerton
\cite[Conjecture~3.7.2]{EmertonII}, which to our knowledge has only been resolved
for $\GL_2$. In contrast, the $\Ext$-groups in this paper are computed in the
category of all smooth representations and do not rely on conjectural
statements. 

To state the results concerning $\Ext$-groups, we introduce some
notation. We fix a maximal $\field$-split torus contained in a minimal parabolic
subgroup $B$. These choices come with a (relative) root system together with a
set of simple roots. Fix standard parabolic subgroups $P=M\ltimes U$ and $Q =
L\ltimes N$. We choose a distinguished set $\orbitrep{P}{Q}$ of
double coset representatives of $P\backslash G/Q$, see
\S\ref{sss:distinguished-reps} for more details. For each $n\in \orbitrep{P}{Q}$
we consider the smooth character $\delta_n$ of $n^{-1}Mn\cap L$ given by
$\omega_n = \delta_n[\dim(n^{-1}\ol Un\cap N)]$.

\begin{thmintro}[Theorems~\ref{thm:PS} and~\ref{thm:Ext}]~ 
\begin{enumerate}[label=(\alph*)]
\item Assume $Q=B$. Let $\chi\colon M\to k^\times$ be a smooth character, let
$r\in \ZZ_{\ge0}$, and denote $Z(L)$ the center of $L$.
\begin{enumerate}[label=(\roman*)]
\item Let $\chi'\colon L\to k^\times$ be a smooth character. If
\[
\Ext_G^r\bigl(\pInd_P^G\chi, \pInd_B^G\chi'\bigr) \neq 0,
\]
then there exists $n\in \orbitrep{P}{B}$ such that $\dim(n^{-1}\ol Un\cap N) \le
r$ and $\chi'\cong \delta_n\otimes_k n_*^{-1}\chi$ after restriction to $Z(L)$.
\item Assume $\delta_n\otimes_k n_*^{-1}\chi\not\cong \chi$ after restriction to
$Z(L)$, for all $n\in\orbitrep{P}{B}$. For each $n\in\orbitrep{P}{B}$ with
$\dim(n^{-1}\ol Un\cap N)\le r$ one has $k$-linear isomorphisms
\[
\Ext_G^r\bigl(\pInd_P^G\chi, \pInd_B^G(\delta_n\otimes_kn_*^{-1}\chi)\bigr) \cong
\Ext_L^{r-\dim(n^{-1}\ol Un\cap N)}(\one,\one) \cong \H^{r-\dim(n^{-1}\ol Un\cap
N)}(L,k),
\]
where $\H^*(L,k)$ denotes continuous group cohomology.
\end{enumerate}

\item Let $V \in \Rep_k(M)$ and $W\in \Rep_k(L)$.
\begin{enumerate}[label=(\roman*)]
\item Assume $P\nsubseteq Q$ and $P\nsupseteq Q$, that $V$ is left cuspidal and
that $W$ is right cuspidal (see Definition~\ref{defn:left_cuspidal}). Then
\[
\Ext_G^1\bigl(\pInd_P^GV, \pInd_Q^GW\bigr) = 0.
\]
\item Assume $P=Q$. For each $0\le i <[\field:\QQ_p]$ the functor $\pInd_P^G$
induces a $k$-linear isomorphism 
\[
\Ext^{i}_M(V,W) \raiso \Ext_G^i\bigl(\pInd_P^GV, \pInd_P^GW\bigr).
\]
If moreover $V$ is left cuspidal or $W$ is right cuspidal, and
$V$ and $W$ admit distinct central characters, then
\[
\Ext_G^{[\field:\QQ_p]} \bigl(\pInd_P^GV, \pInd_P^GW\bigr) \cong
\bigoplus_{\alpha\in \Delta_M^{\perp,1}} \Hom_M\bigl(\delta_{n_\alpha} \otimes_k
n_{\alpha*}^{-1}V, W\bigr),
\]
where $\Delta_M^{\perp,1}$ denotes the set of simple roots $\alpha$ of $G$
which are orthogonal to all simple roots of $M$ and such that the associated
root space has dimension $[\field:\QQ_p]$ as a $p$-adic manifold. Here,
$n_\alpha\in \orbitrep{P}{P}$ denotes the lift of the simple reflection
corresponding to $\alpha$.

\item Assume $P\supsetneqq Q$ and that $V$ is left cuspidal. For all $0\le i\le
[\field:\QQ_p]$, the functor $\pInd_P^G$ induces a $k$-linear isomorphism
\[
\Ext_M^i\bigl(V, \pInd_{M\cap Q}^MW\bigr) \raiso \Ext_G^i\bigl(\pInd_P^GV, \pInd_Q^GW\bigr).
\]

\item Dually, assume $P\subsetneqq Q$ and that $W$ is right cuspidal. For all
$0\le i\le [\field:\QQ_p]$, the functor $\pInd_Q^G$ induces a $k$-linear isomorphism
\[
\Ext_L^i\bigl(\pInd_{P\cap L}^LV, W\bigr) \raiso \Ext_G^i\bigl(\pInd_P^GV, \pInd_Q^GW\bigr).
\]
\end{enumerate}
\end{enumerate}
\end{thmintro} 

We further compute the representations $\L^{-j}(N, \Sp_P^G)$ for all $j\ge0$,
where $\Sp_P^G$ denotes the generalized Steinberg representation attached to
$P$, \ie, the unique irreducible quotient of $\pInd_P^G(\one)$; see
Theorem~\ref{thm:Generalized_Steinberg} and
Corollary~\ref{cor:Generalized_Steinberg}. To our knowledge, this is the first
computation of this kind. This raises the question whether one can compute
$\L^{-j}(N,V)$ for all irreducible smooth representations $V$. However, for
supersingular $V$ the answer seems to be out of reach with the current methods
available. A naive hope would be that $\L(N,V) = 0$ for all supersingular $V$,
but \cite[Theorem~10.37]{Hu-Wang.2022} shows that already for $G=\GL_2(\field)$,
where $\field\supsetneq \QQ_p$ is unramified, there exist a supersingular
representation $V$ and a principal series $\pInd_P^GW$ such that
$\Ext_G^1(V,\pInd_P^GW) \neq 0$, and this implies $\L^{-1}(U,V) \neq 0$. In view of
this it is unclear what one should expect.
\subsection{Organization of the paper} 
The Geometrical Lemma is the content of Corollary~\ref{cor:geometrical}. The preparatory lemmas concerning compact induction and derived coinvariants are proved in~\S\ref{ss:induction_coinvariants}; these in turn rely on the abstract results about functors in monoidal categories which are presented in appendix~\ref{ss:abstract}. The various characters that implicitly appear in the Geometrical Lemma are completely determined in~\S\ref{ss:duality}.

Regarding applications, we compute $\L^{-j}(N,V)$ whenever $V$ is a principal series representation (Example~\ref{ex:LU(PS)}) or a generalized Steinberg representation (Theorem~\ref{thm:Generalized_Steinberg} and Corollary~\ref{cor:Generalized_Steinberg}).  Finally, we use the Geometrical Lemma to compute many Ext-groups between parabolically induced representations in Theorems~\ref{thm:PS} and~\ref{thm:Ext}.

\subsection{Acknowledgments} 
I wish to thank Lucas Mann for several interesting discussions. I have further
profited from discussions with Peter Schneider, Eugen Hellmann, Damien Junger,
Kieu Hieu Nguyen, Konstantin Ardakov, Claus Sorensen, and Karol Kozio\l. I thank an anonymous referee for providing detailed feedback and making several helpful suggestions.


\section{Preliminaries}
\subsection{Notation and conventions} 
We fix a finite extension $\field/\QQ_p$ and a coefficient field $k$ of characteristic $p$.
The mod $p$ cyclotomic character of $\QQ_p^\times$ is given by the composition $\varepsilon\colon \QQ_p^\times \to \ZZ_p^\times \to \FF_p^\times \subseteq k^\times$, where the first map is given by $x\mapsto xp^{-\val_p(x)}$ for the usual $p$-adic valuation $\val_p$. 
We put $\varepsilon_{\field} \coloneqq \varepsilon\circ \Nm_{\field/\QQ_p}$, where $\Nm_{\field/\QQ_p}$ is the norm of $\field/\QQ_p$.

If $\alg H$ is an algebraic group defined over $\field$, we denote its
group of $\field$-points by the corresponding lightface letter, that is,
$H = \alg H(\field)$.

We fix a field $k$ of characteristic $p>0$.

For a $p$-adic Lie group $G$, we denote by $\dim(G)$ its dimension as a $p$-adic manifold. We denote by $\Rep_k(G)$ the Grothendieck abelian category of smooth $k$-linear $G$-representations. 

The (unbounded) derived category of $\Rep_k(G)$ is denoted by $\D(G)$; it is a \emph{tensor triangulated category}, that is, a triangulated category which is symmetric monoidal and such that the functors $\blank\otimes_kX$ are triangulated for any $X$ in $\D(G)$. The tensor unit is $\one = k[0]$. We denote the right adjoint of $\blank\otimes_kX$ by $\iHom(X,\blank) \colon \D(G)\to \D(G)$.
The \emph{smooth dual} of $X$ is denoted by $X^\vee \coloneqq \iHom(X,\one)$. 

The category $\D(G)$ is naturally enriched over $\D(k)\coloneqq \D(\{1\})$, the unbounded derived category of the category of $k$-vector spaces; we denote by $\RHom_G(X,Y) \in \D(k)$ the derived Hom-complex, for each $X,Y\in \D(G)$; it defines a triangulated functor in each variable.

The derived category $\D(G)$ comes with a natural t-structure. For any integer $n\in\ZZ$ we denote by $\D^{\le n}(G)$ (resp.\ $\D^{\ge n}(G)$) the full subcategory of objects $X$ in $\D(G)$ satisfying $\H^i(X) = 0$ for all $i>n$ (resp.\ $i<n$). Denote by $\H^i\colon \D(G)\to \Rep_k(G)$ the $i$-th cohomology functor. Note that $\RHom_G(X,Y) \in \D^{\ge0}(k)$ provided $X\in \D^{\le0}(G)$ and $Y\in \D^{\ge0}(G)$.

If $F, G\colon \cat C \to \cat D$ are two functors, we denote by $\Nat(F,G)$ the class of natural transformations from $F$ to $G$.

\subsection{Compact induction and derived coinvariants} 
\label{ss:induction_coinvariants}
Let $k$ be a field of characteristic $p$. 

Many arguments in this section rely crucially on the mate correspondence, which is recalled in \S\ref{sss:mates}.
\subsubsection{} 
\label{sss:compact_induction}
Let $G$ be a $p$-adic Lie group and $H\subgroup G$ a closed subgroup. 
Given a smooth $H$-re\-pre\-sen\-ta\-tion $V$, we denote by $\cInd_H^GV \in
\Rep_k(G)$
the space of all locally constant functions $f\colon G\to V$ which satisfy
$f(hg) = hf(g)$ for all $h\in H$, $g\in G$, and have compact support in
$H\backslash G$; note that $f$ is fixed by an open
subgroup of $G$ under the right translation action. The functor $V\mapsto
\cInd_H^GV$ is exact and hence extends to a triangulated functor on the
(unbounded) derived categories:
\[
\cInd_H^G\colon \D(H) \too \D(G).
\]
As $\cInd_H^G$ clearly commutes with direct sums, Brown representability shows
that it admits a right adjoint $\RRR^G_H$,
cf.~\cite[Corollary~2.3.10(a)]{Heyer.2022}. The functors $\cInd_H^G$ and
$\RRR^G_H$ are transitive, \ie, if $H\subgroup H'\subgroup G$ is a closed
intermediate group, then $\cInd_H^G \cong \cInd_{H'}^G\cInd_H^{H'}$ and
$\RRR^G_H \cong \RRR^{H'}_H \RRR^G_{H'}$.
If $K\subgroup G$ is an open subgroup, we observe that $\RRR^G_K \cong \Res^G_K$ by Frobenius reciprocity, where $\Res^G_K$ is the restriction functor; in this case we prefer to write $\ind_K^G$ instead of $\cInd_K^G$.

\subsubsection{}\label{sss:smooth_induction} 
With $H\subgroup G$ as above, given a smooth $H$-representation $V$, the group
$G$ acts by right translation on the space of all functions $f\colon G\to V$
which satisfy $f(hg) = hf(g)$ for all $h\in H$, $g\in G$; we denote by $\Ind_H^G
V \in \Rep_k(G)$ the subspace of functions which are fixed by an open subgroup
of $G$. The functor $V\mapsto \Ind_H^GV$ is left exact; if $H\backslash G$ is
compact, then $\cInd_H^G \Raiso \Ind_H^G$ is even exact. Taking the right
derived functor, we obtain a triangulated functor
\[
\RInd_H^G\colon \D(H) \too \D(G),
\]
which is right adjoint to restriction $\Res^G_H\colon \D(G)\to \D(H)$ by Frobenius reciprocity. By a slight abuse of notation we write $\Ind_H^G$ for $\RInd_H^G$ in case
$H\backslash G$ is compact.

The restriction functor
$\Res^G_H\colon \D(G)\to\D(H)$ satisfies the following compatibility with
compact induction.

\begin{lem}[Projection formula] 
\label{lem:projection_formula}
Let $H\subgroup G$ be a closed subgroup. There exists an isomorphism
\[
\cInd_H^G\bigl(X\otimes_k\Res^G_HY\bigr)
\raiso 
\cInd_H^G(X) \otimes_k Y 
\qquad \text{in $\D(G)$}
\]
which is natural in $X \in \D(H)$ and $Y\in\D(G)$. Moreover, the natural map
\[
\iHom\bigl(\cInd_H^GX, Y\bigr) \raiso \RInd_H^G\iHom\bigl(X, \RRR^G_HY\bigr)
\]
is an isomorphism in $\D(G)$, for all $X\in \D(H)$ and $Y\in \D(G)$.
\end{lem} 
\begin{proof} 
We describe the inverse map.
For any $V\in \Rep_k(H)$ and $W\in \Rep_k(G)$, the morphism $\cInd_H^G(V)\otimes_kW
\to \cInd_H^G(V\otimes_k\Res^G_HW)$ given by $f\otimes w\mapsto
[g\mapsto f(g)\otimes gw]$ is a natural isomorphism,
cf.~\cite[Lemma~2.5]{AHV.2019}. Since $\cInd_H^G$, $\Res^G_H$, and
$\blank\otimes_k\blank$ are exact functors, this isomorphism readily extends to
the derived categories. 

The isomorphism $\cInd_H^G (X\otimes_k\blank) \Res^G_H \Raiso
\cInd_H^GX\otimes_k\blank$ yields, by passing to the right adjoints, an
isomorphism $\iHom(\cInd_H^GX,\blank) \Raiso \RInd_H^G \iHom(X,\blank)
\RRR^G_H$, which proves the last assertion.
\end{proof} 

\subsubsection{} 
\label{sss:coinvariants}
Given a closed normal subgroup $N\normal G$, the inflation 
$\Inf^{G/N}_G\colon \D(G/N)\to \D(G)$ along the projection $G\to G/N$ admits a
right adjoint $\R\H^0(N,\blank)$ as well as a 
left adjoint, see~\cite[Theorem~3.2.3]{Heyer.2022}, which we call the functor of
\emph{derived coinvariants} and denote
\[
\L_N\colon \D(G) \too \D(G/N).
\]
Note that for any subgroup $N\subgroup H\subgroup G$ there is an isomorphism $\Res^{G/N}_{H/N} \L_N \cong \L_N \Res^G_H$ by \cite[Proposition~3.2.19]{Heyer.2022} so that the notation should not lead to confusion.

If $G$ is compact and torsion-free, then $\R\H^0(N,\blank)$ admits a
right adjoint denoted $F^{G/N}_G$, \cite[Lemma~3.1.3]{Heyer.2022}. In
this case, we call $\omega_G\coloneqq F^{G/N}_{G}(\one) \in \D(G)$ the
\emph{dualizing complex}; we remark that $\omega_G \cong k[\dim N]$,
\cite[Proposition~3.1.10]{Heyer.2022}.

\begin{lem}\label{lem:Res_F_commute} 
Let $G$ be compact and torsion-free, and let $N\subgroup H\subgroup G$ be closed
subgroups such that $N\normal G$. Then one has a commutative diagram
\[
\begin{tikzcd}
\D(G/N) \ar[r,"\Res^{G/N}_{H/N}"] \ar[d,"F^{G/N}_G"']
&
\D(H/N) \ar[d,"F^{H/N}_H"]
\\
\D(G) \ar[r,"\Res^G_H"']
&
\D(H).
\end{tikzcd}
\]
\end{lem} 
\begin{proof} 
We apply Lemma~\ref{lem:abstract-a} to the isomorphism $\alpha\colon \Res^G_H
\Inf^{G/N}_G \Raiso \Inf^{H/N}_{H} \Res^{G/N}_{H/N}$. By
\S\ref{sss:coinvariants} and the projection formula for $\R\H^0(N,\blank)$,
\cite[Lemma~3.1.2]{Heyer.2022}, the assumptions \ref{Assumption1} and
\ref{Assumption2} of \S\ref{sss:setup1} are satisfied.
Thus, we obtain a commutative diagram
\[
\begin{tikzcd}
\Res^G_H \omega^{G/N}_G \otimes_k \Res^G_H \Inf^{G/N}_G(X)
\ar[r,"\cong"] \ar[d,"\beta_{\one} \otimes\alpha"']
&
\Res^G_H F^{G/N}_G(X) \ar[d,"\beta"]
\\
\omega^{H/N}_H\otimes_k \Inf^{H/N}_H \Res^{G/N}_{H/N}(X)
\ar[r,"\cong"']
&
F^{H/N}_H \Res^{G/N}_{H/N}(X),
\end{tikzcd}
\]
where $\beta \coloneqq r(r(\alpha^{-1})^{-1})$ in the notation
of~\S\ref{sss:setup1}. The top and bottom horizontal maps are
isomorphisms by \cite[Corollary~3.1.7]{Heyer.2022}. The assertion is that
$\beta$ is an isomorphism. But note that $\beta$ is non-zero (as the right mate
of a non-zero map) and hence $\beta_{\one}\neq 0$ by the commutativity of the
diagram. But since $\omega^{G/N}_G$ and $\omega^{H/N}_H$ are characters,
$\beta_{\one}$ is necessarily an isomorphism.
Since
also $\alpha$ is invertible, we deduce that the left vertical map in the
diagram is an isomorphism. Therefore, $\beta$ is an isomorphism.
\end{proof} 

\begin{lem}\label{lem:LN_Inf_commute} 
Let $H, N \normal G$ be closed normal subgroups such that $HN$ is closed. Assume
$\L_{H\cap N}(\one) \cong \one$.\footnote{Recall that by \cite[Proposition~3.2.19]{Heyer.2022} this condition is independent of the group containing $H\cap N$ as a normal subgroup.} Then one has a commutative diagram
\[
\begin{tikzcd}
\D(G/H) \ar[r,"\Inf^{G/H}_G"] \ar[d,"\L_{HN/H}"'] & \D(G) \ar[d,"\L_N"]
\\
\D(G/HN) \ar[r,"\Inf^{G/HN}_{G/N}"'] & \D(G/N).
\end{tikzcd}
\]
\end{lem} 
\begin{proof} 
The natural isomorphism $\Inf^{G/H}_G \Inf^{G/HN}_{G/H} \Raiso
\Inf^{G/N}_G \Inf^{G/HN}_{G/N}$ induces, by passing to the left mates, a natural
transformation 
$
\L_N \Inf^{G/H}_G \To \Inf^{G/HN}_{G/N} \L_{HN/H}
$
which we claim is an isomorphism of functors $\D(G/H)\to \D(G/N)$. This can be
checked after applying the conservative functor $\Res^{G/N}_1$. By the
compatibility of restriction with derived coinvariants,
\cite[Proposition~3.2.19]{Heyer.2022}, and inflation we reduce to the case
$G=N$. Hence, we have to show that the natural map
\begin{equation}\label{eq:LN_Inf_commute}
\L_G \Inf^{G/H}_G \Too \L_{G/H}
\end{equation}
which arises as the left mate of $\varphi\colon \Inf^{G/H}_G \Inf^1_{G/H}
\Raiso \Inf^1_G$ is an isomorphism. Let us denote by $\psi\colon
\L_G\Raiso \L_{G/H}\L_H$ the isomorphism obtained from $\varphi$
by passing to the left adjoints. 
By \cite[Corollary~3.4.23]{Heyer.2022}, the
hypothesis $\L_H(\one)\cong \one$ means that the counit $\varepsilon\colon
\L_H\Inf^{G/H}_G \Raiso \id_{\D(G/H)}$ is an isomorphism.
Consider the following commutative diagram
\[
\begin{tikzcd}
&[-5.3em]
\Nat\bigl(L_{G/H}, \L_{G/H}\bigr)
\ar[r,leftrightarrow] \ar[d,"(\L_{G/H}\varepsilon)^*"']
&
\Nat\bigl(\Inf^1_{G/H}, \Inf^1_{G/H}\bigr) 
\ar[d,"\Inf^{G/H}_G"] 
&
[-7.5em]\ni\id_{\Inf^1_{G/H}}
\\
&
\Nat\bigl(\L_{G/H}\L_H\Inf^{G/H}_G, \L_{G/H}\bigr)
\ar[r,leftrightarrow] \ar[d,"(\psi\Inf^{G/H}_G)^*"']
&
\Nat\bigl(\Inf^{G/H}_G\Inf^1_{G/H}, \Inf^{G/H}_G\Inf^1_{G/H}\bigr)
\ar[d,"\varphi_*"]
\\
\eqref{eq:LN_Inf_commute}\in 
&
\Nat\bigl(\L_G\Inf^{G/H}_G, \L_{G/H}\bigr)
\ar[r,leftrightarrow]
&
\Nat\bigl(\Inf^{G/H}_G \Inf^1_{G/H}, \Inf^1_G\bigr),
\end{tikzcd}
\]
where the horizontal maps are given by passing to the right/left mates. Now, the
map \eqref{eq:LN_Inf_commute} is the image of $\id_{\Inf^1_{G/H}}$ under the
lower-right circuit. By the commutativity of the diagram we deduce that
\eqref{eq:LN_Inf_commute} coincides with the composition
\[
\L_G\Inf^{G/H}_G \xRightarrow{\psi\Inf^{G/H}_G} \L_{G/H}\L_H\Inf^{G/H}_G
\xRightarrow{\L_{G/H}\varepsilon} \L_{G/H}
\] 
of two isomorphisms and is thus itself an isomorphism.
\end{proof} 

\begin{lem}\label{lem:LN_cInd_commute} 
Let $N\subgroup H\subgroup G$ be closed subgroups such that $N\normal G$ is
normal. One has commutative diagrams
\[
\begin{tikzcd}
\D(H) \ar[r,"\cInd_H^G"] \ar[d,"\L_N"'] & \D(G) \ar[d,"\L_N"]
&
\D(G/N) \ar[r,"\RRR^{G/N}_{H/N}"] \ar[d,"\Inf^{G/N}_G"']
& \D(H/N) \ar[d,"\Inf^{H/N}_H"]
\\
\D(H/N) \ar[r,"\cInd_{H/N}^{G/N}"'] & \D(G/N)
& \D(G) \ar[r,"\RRR^G_H"'] & \D(H).
\end{tikzcd}
\]
\end{lem} 
\begin{proof} 
Denoting $\pr\colon G\to G/N$ the projection map, we consider the natural
isomorphism 
\begin{equation}\label{eq:cInd_Inf_commute}
\cInd_H^G \Inf^{H/N}_H \Raiso \Inf^{G/N}_G \cInd_{H/N}^{G/N}
\end{equation}
of functors $\D(H/N)\to \D(G)$, whose inverse is given by $f\mapsto f\circ\pr$ on
the level of underived categories. Passing to the left and right mates,
respectively, we obtain natural transformations
\begin{align}\label{eq:LN_cInd_commute_left}
\L_N\cInd_H^G &\Too \cInd_{H/N}^{G/N}\L_N\\
\label{eq:LN_cInd_commute_right}
\Inf^{H/N}_H\RRR^{G/N}_{H/N} &\Too \RRR^G_H\Inf^{G/N}_G.
\end{align}
Note that \eqref{eq:LN_cInd_commute_right} is obtained from
\eqref{eq:LN_cInd_commute_left} by passing to the right adjoints, and hence one
is an isomorphism if and only if the other is. Thus, it
suffices to prove that \eqref{eq:LN_cInd_commute_right} is an isomorphism.

Let $K\subgroup G$ be any open subgroup; we write $K_H = K\cap H$ and $K_N =
K\cap N$. Consider the following commutative diagram:
\[
\begin{tikzcd}[column sep=normal]
\Res^H_{K_H} \Inf^{H/N}_H \RRR^{G/N}_{H/N}
\ar[d,Rightarrow, "\Res^H_{K_H}\eqref{eq:LN_cInd_commute_right}"'] 
\ar[r,Rightarrow, "\cong"]
&
\Inf^{K_H/K_N}_{K_H}\Res^{H/N}_{K_H/K_N} \RRR^{G/N}_{H/N}
\ar[r,Rightarrow, "\cong"]
&
\Inf^{K_H/K_N}_{K_H} \RRR^{K/K_N}_{K_H/K_N} \Res^{G/N}_{K/K_N}
\ar[d,Rightarrow, "\eqref{eq:LN_cInd_commute_right} \Res^{G/N}_{K/K_N}"]
\\
\Res^H_{K_H} \RRR^G_H \Inf^{G/N}_G 
\ar[r,Rightarrow, "\cong"']
&
\RRR^K_{K_H} \Res^G_K \Inf^{G/N}_G 
\ar[r,Rightarrow, "\cong"']
&
\RRR^K_{K_H} \Inf^{K/K_N}_K \Res^{G/N}_{K/K_N}.
\end{tikzcd}
\]
Here, the upper right and lower left horizontal maps are isomorphisms, because
compact induction, hence also its right adjoint, is transitive,
\cite[I.5.3]{Vigneras.1996}. As $\Res^H_{K_H}$ is conservative,
\eqref{eq:LN_cInd_commute_right} is an isomorphism if and only if the left
vertical map is an isomorphism, if and only if the right vertical map is an
isomorphism.

Thus, replacing $G$ by $K$, we may assume from the beginning that $G$ is compact
and torsion-free. In this setting, the map \eqref{eq:LN_cInd_commute_left} reads
\begin{equation}\label{eq:LN_cInd_commute_left-2}
\L_N \Ind_H^G \Too \Ind_{H/N}^{G/N} \L_N.
\end{equation}
We finish by proving that \eqref{eq:LN_cInd_commute_left-2} is an isomorphism.
In the proof of Lemma~\ref{lem:Res_F_commute} we verified that the isomorphism
$\alpha\colon \Res^G_H \Inf^{G/N}_G \Raiso \Inf^{H/N}_{H}
\Res^{G/N}_{H/N}$ satisfies the assumptions \ref{Assumption1} and
\ref{Assumption2} of \S\ref{sss:setup1} in the notation of which the
map \eqref{eq:LN_cInd_commute_left-2} is just $\gamma\coloneqq
l(r(\alpha)^{-1})$. We apply Lemma~\ref{lem:abstract-b} in the context
\[
\begin{tikzcd}[sep=large,font=\small]
\D(G/N)\ar[d,"\ol{f}^* = \Inf^{G/N}_G"'] \ar[r,"\ol{g}^* = \Res^{G/N}_{H/N}"] 
& \D(H/N) \ar[d,"f^* = \Inf^{H/N}_H"] \\
\D(G) \ar[r,"g^* = \Res^G_H"'] & \D(H)
\end{tikzcd}
\]
and with $a = \one$; we thus obtain a commutative diagram
\[
\begin{tikzcd}[sep=normal,font=\small]
\R\H^0(N,\omega^{G/N}_G\otimes_k \blank)\Ind_H^G
\ar[r,Rightarrow,"\cong"] \ar[d,Rightarrow,"\cong"']
&[-1pt]
\R\H^0(N,\blank) \Ind_H^G \bigl(\Res^G_H\omega^{G/N}_G\otimes_k\blank\bigr)
\ar[r,Rightarrow,"\cong"]
&[-1pt]
\Ind_{H/N}^{G/N} \R\H^0(N,\omega^{H/N}_H\otimes_k\blank)
\ar[d,Rightarrow,"\cong"]
\\
\L_N\Ind_H^G
\ar[rr,Rightarrow,"\gamma"']
&&
\Ind_{H/N}^{G/N}\L_N
\end{tikzcd}
\]
commutes. The left and right vertical maps are isomorphisms by
\cite[Corollary~3.1.8]{Heyer.2022}. Hence, $\gamma$ is an isomorphism.
\end{proof} 

\begin{rmk}\label{rmk:RN_Ind} 
Let $N\normal G$ and $H\subgroup G$ be closed  subgroups such that $G = HN$.
Then one has a natural isomorphism
\[
\R\H^0(N,\blank) \RInd_H^G \Raiso \R\H^0(H\cap N,\blank)
\]
of functors $\D(H)\to \D(G/N)$ arising from the
isomorphism $\Inf^{H/H\cap N}_{H} \Raiso \Res^G_H\Inf^{G/N}_G$ by passing to the
right adjoints.

We will need the following dual statement.
\end{rmk} 

\begin{lem}\label{lem:omega_LN_cInd} 
Let $N\normal G$ and $H\subgroup G$ be closed subgroups such that $G = HN$.
Assume there exists a subnormal series 
$
H\cap N = N_r \normal N_{r-1} \normal \dotsb \normal N_1 \normal N_0 = N
$
by closed subgroups such that $\L_{N_i}(\one) = \one$ for all $i=1,\dotsc,r$.
Then
\begin{equation}\label{eq:omega}
\omega_{H\backslash G} \coloneqq \L_N \cInd_H^G(\one)\quad \in \D(G/N)
\end{equation}
is a character concentrated in degree $-\dim H\backslash G$, and the following
diagram commutes:
\[
\begin{tikzcd}
\D(H) \ar[d,"\L_{H\cap N}"'] \ar[r,"\cInd_H^G"] & \D(G) \ar[d,"\L_N"]
\\
\D(H/H\cap N) \ar[r,"\omega_{H\backslash G}\otimes \blank"'] & \D(G/N).
\end{tikzcd}
\]
\end{lem} 

\begin{rmk*}
We will check in Corollary~\ref{cor:explicit-character} that $\Inf^{G/N}_H\omega_{H\backslash G}$ is independent of the choice of $N$.
\end{rmk*}

\begin{proof} 
We will implicitly make the identification $H/H\cap N\cong G/N$ throughout the
proof. Let $Y\in \D(k)$ be arbitrary and put $X \coloneqq \Inf^1_{H}Y$. Consider
the natural isomorphism
\[
\L_N \cInd_H^G (X\otimes_k \blank) \Res^G_H\Inf^{G/N}_G
\xRightarrow[\cong]{\L_N\pf} \L_N (\cInd_H^G X\otimes_k \blank)
\Inf^{G/N}_G \xRightarrow[\cong]{\lpf} \L_N\cInd_H^GX\otimes_k \blank,
\]
where $\pf$ is the isomorphism in Lemma~\ref{lem:projection_formula} and $\lpf$
is the map $\lpf_f$ from \S\ref{sss:setup2} for the functor $f^* =
\Inf^{G/N}_G$; note that $\lpf$ is an isomorphism by
\cite[Corollary~3.3.5]{Heyer.2022}. Using $\Res^G_H\Inf^{G/N}_G = \Inf^{G/N}_H$
and passing to the left mate in the displayed isomorphism above, we obtain a natural map
\[
\rho_{N,H} \colon \L_N\cInd_H^G(X\otimes X') \Too \L_N\cInd_H^GX \otimes_k
\L_{H\cap N}X'
\]
for any $X'\in \D(H)$. We claim that $\rho_{N,H}$ is an isomorphism, which
then for $X = \one$ witnesses the commutativity of the asserted diagram.
\medskip

We reduce to the case $N=G$ as follows: apply the discussion of
\S\ref{sss:setup2} to the diagram
\[
\begin{tikzcd}[column sep=tiny]
\D(G/N)
\ar[rrr,"\Res^{G/N}_1"] \ar[dr,"\Inf^{G/N}_G"] 
\ar[dd,bend right,"\Inf^{G/N}_H"']
& &[5em] & 
\D(k)
\ar[dl,"\Inf^1_N"] 
\ar[dd,bend left,"\Inf^1_{H\cap N}"]
\\
& 
\D(G)
\ar[r,"\Res^G_N"] \ar[dl,shift right,"\Res^G_H"']
& 
\D(N)
\ar[dr,shift left,"\Res^N_{H\cap N}"]
\\
\D(H)
\ar[rrr,"\Res^H_{H\cap N}"'] \ar[ur,shift right,dashed, "\cInd_H^G"'] 
& & & 
\D(H\cap N).
\ar[ul,shift left, dashed, "\cInd_{H\cap N}^N"]
\end{tikzcd}
\]
Note that, since $G=HN$, restriction of functions induces a natural
isomorphism whose inverse
\[
\cInd_{H\cap N}^N\Res^{H}_{H\cap N} \Raiso \Res^G_N \cInd_H^G
\]
plays the role of $\phi$ in \S\ref{sss:setup2}.
The condition \ref{Assumption3} is satisfied by
\cite[Theorem~3.2.3]{Heyer.2022}. The isomorphisms in \ref{Assumption4} are
supplied by Lemma~\ref{lem:projection_formula} for
which \ref{Assumption4} is easily verified. Now, Lemma~\ref{lem:abstract-2}
provides a commutative diagram
\[
\begin{tikzcd}[column sep=3em]
\L_N\cInd_{H\cap N}^N\bigl(\Res^H_{H\cap N}X\otimes_k \Res^H_{H\cap N}X'\bigr)
\ar[r,"\rho_{N,H\cap N}"] \ar[d,"\cong"']
&
\L_N\cInd_{H\cap N}^N\Res^H_{H\cap N}X\otimes_k \L_{H\cap N}\Res^H_{H\cap N}X'
\ar[d,"\cong"]
\\
\Res^{G/N}_1 \L_N\cInd_H^G(X\otimes_kX')
\ar[r,"\Res^{G/N}_1\rho_{N,H}"']
&
\Res^{G/N}_1\bigl(\L_N\cInd_H^GX\otimes_k \L_{H\cap N}X'\bigr).
\end{tikzcd}
\]
The maps $l(\sigma)$ and $l(\rho)$ in Lemma~\ref{lem:abstract-2} correspond to
the natural transformations
\[
l(\sigma)\colon \L_N\Res^G_N \Raiso \Res^{G/N}_1\L_N
\quad\text{and}\quad
l(\rho)\colon \L_{H\cap N}\Res^{H}_{H\cap N} \Raiso \Res^{G/N}_1 \L_{H\cap N},
\]
which are isomorphisms by \cite[Proposition~3.2.19]{Heyer.2022}. Therefore, the
vertical maps are isomorphisms.
Since $\Res^{G/N}_1$ is conservative, it suffices to show that the top map is an
isomorphism. In other words, we may assume from the beginning that $N=G$ and
have to show that the map
\begin{equation}\label{eq:rho_N,H}
\rho_{N,H}\colon \L_N\cInd_H^N(X\otimes_kX') \too \L_N\cInd_H^NX\otimes_k
\L_{H\cap N}X'
\end{equation}
is a natural isomorphism. We induct on the length $r$ of the subnormal series.
Let $r=1$ so that $H\normal N$.\footnote{Note that, if $r=0$, then $N=H=G$ and
the assertion of the lemma is tautological. It is true, although not trivial,
that $\rho_{N,N}$ is an isomorphism and will be left as an easy exercise.} 
We apply the discussion in \S\ref{sss:setup3} to the diagram
\[
\begin{tikzcd}
\D(k)
\ar[r,"\Inf^1_{N/H}"] \ar[dr,bend right,equals]
&
\D(N/H)
\ar[d,shift right,"\Res^{N/H}_1"'] \ar[r,"\Inf^{N/H}_{N}"]
&
\D(N)
\ar[d,shift left,"\Res^{N}_{H}"]
\\
&
\D(k)
\ar[u,shift right,dashed,"\cInd_1^{N/H}"'] \ar[r,"\Inf^1_{H}"']
&
\D(H).
\ar[u,shift left,dashed,"\cInd_H^N"]
\end{tikzcd}
\]
The map $l(\phi)$ in Lemma~\ref{lem:abstract-3} is now given by
$\L_H\cInd_H^N\Raiso \cInd_{1}^{N/H}\L_H$, see~\eqref{eq:LN_cInd_commute_left}.
The conditions \ref{Assumption3}--\ref{Assumption5} are clearly satisfied. By
Lemma~\ref{lem:abstract-3}, and since $\L_N\cong \L_{N/H}\L_H$, we obtain a
commutative diagram
\[
\begin{tikzcd}
\L_N\cInd_H^N\bigl(\Inf^1_{H}Y\otimes_k X'\bigr)
\ar[d,"\cong"'] \ar[r,"\rho_{N,H}"]
&
\L_N\cInd_H^N\Inf^1_HY\otimes_k \L_HX'
\ar[d,"\cong"]
\\
\L_{N/H} \cInd_1^{N/H}\L_H\bigl(\Inf^1_{H}Y \otimes_k X'\bigr)
\ar[d,"\cong"']
&
\L_N\Inf^{N/H}_N\cInd_1^{N/H}Y\otimes_k \L_HX'
\ar[d,"\varepsilon\otimes\id","\cong"']
\\
\L_{N/H}\cInd_{1}^{N/H}(Y\otimes_k \L_HX')
\ar[r,"\rho_{N/H,1}"',"\cong"]
&
\L_{N/H}\cInd_{1}^{N/H}Y\otimes \L_HX'.
\end{tikzcd}
\]
The lower left vertical map is an isomorphism by
\cite[Corollary~3.3.5]{Heyer.2022}, and the lower right vertical map is an
isomorphism by \cite[Corollary~3.4.23]{Heyer.2022} and the hypothesis
$\L_H(\one) \cong \one$. It is obvious from the construction that $\rho_{N/H,1}$
is an isomorphism. Hence, $\rho_{N,H}$ is an isomorphism, which settles the case
$r=1$.

Let now $r\ge2$. We apply the discussion in \S\ref{sss:setup4} to the diagram
\[
\begin{tikzcd}
& & 
\D(N)
\ar[dl,shift right,"\Res^N_{N_1}"']
\ar[dd,shift left, bend left,"\Res^{N}_{H}"]
\\
\D(k)
\ar[urr,bend left,"\Inf^1_N"]
\ar[r,"\Inf^1_{N_1}"]
\ar[drr,bend right,"\Inf^1_{H}"']
& 
\D(N_1)
\ar[dr,shift right,"\Res^{N_1}_H"']
\ar[ur,dashed,shift right,"\cInd_{N_1}^N"']
\\
& & 
\D(H)
\ar[ul,dashed,shift right, "\cInd_{H}^{N_1}"']
\ar[uu,dashed,shift left, bend right,"\cInd_{H}^{N}"]
\end{tikzcd}
\]
Note that there is a natural isomorphism $\cInd_H^N \Raiso \cInd_{N_1}^N
\cInd_{H}^{N_1}$, \cite[I.5.3]{Vigneras.1996}. Again, the conditions
\ref{Assumption6}--\ref{Assumption8} are easily verified. Let $Y'\in \D(k)$ be
arbitrary and put $Z\coloneqq \Inf^1_{N_1}Y'$. By
Lemma~\ref{lem:abstract-4} we obtain a commutative diagram
\[
\begin{tikzcd}
\L_N\cInd_{H}^{N}\bigl(\Res^{N_1}_{H}Z\otimes_k X\otimes_k X'\bigr)
\ar[d,"\rho_{N,H}"']
\ar[r,"\cong"]
&[2em]
\L_N\cInd_{N_1}^{N}\bigl(Z\otimes_k
\cInd_{H}^{N_1}(X\otimes X')\bigr)
\ar[d,"\rho_{N,N_1}","\cong"']
\\
\L_N\cInd_{H}^{N}\bigl(\Res^{N_1}_{H}Z\otimes_k X\bigr) \otimes_k \L_HX'
\ar[d,"\cong"']
&
\L_N\cInd_{N_1}^{N} Z\otimes_k\L_{N_1}
\cInd_{H}^{N_1}\bigl(X\otimes X'\bigr)
\ar[d,"\id\otimes\rho_{N_1,H}","\cong"']
\\
\L_N\cInd_{N_1}^{N}\bigl(Z\otimes_k
\cInd_{H}^{N_1}X\bigr) \otimes \L_HX'
\ar[r,"\rho_{N,N_1}\otimes \id"',"\cong"]
&
\L_N\cInd_{N_1}^{N} Z \otimes_k \L_{N_1}\cInd_{H}^{N_1}X \otimes_k \L_HX'.
\end{tikzcd}
\]
Here, the top horizontal and lower left vertical map are isomorphisms by
Lemma~\ref{lem:projection_formula}. The maps $\rho_{N,N_1}$ and $\rho_{N_1,H}$
are isomorphisms by the induction hypothesis. It follows that the top left
vertical map is an isomorphism. In particular, for $Y'=\one$, this proves that
\eqref{eq:rho_N,H} is an isomorphism. This finishes the induction step.
\medskip

It remains to prove that $\omega_{H\backslash G} = \L_N\cInd_H^G\one$ is a
character concentrated in degree $-\dim H\backslash G$. Note that
$\Res^{G/N}_1\omega_{H\backslash G} \cong \omega_{H\cap N\backslash N}$, so that
we may assume $N=G$ from the beginning. We first treat the case $H=\{1\}$ so
that the claim becomes $\L_N\cInd_1^N(\one) \cong k[\dim N]$. Fix a
torsion-free open compact subgroup $K_N \subgroup N$; note that $\dim K_N= \dim
N$. We compute
\begin{align*}
\L_N\cInd_1^{N}(\one) &\cong \L_N \ind_{K_N}^N \Ind_1^{K_N}(\one)\\
&\cong \L_{K_N} \Ind_1^{K_N}(\one) & & \text{(by
\cite[Proposition~3.2.11]{Heyer.2022})}\\
&\cong \R\H^0\bigl(K_N, \omega_{K_N}\otimes_k\Ind_1^{K_N}(\one)\bigr) & &
\text{(by \cite[Proposition~3.1.8]{Heyer.2022})}\\
&\cong \R\H^0\bigl(K_N, \Ind_1^{K_N}k[\dim K_N]\bigr) & & \text{(by
\cite[Proposition~3.1.10]{Heyer.2022})}\\
&\cong k[\dim K_N],
\end{align*}
where the last isomorphism is Shapiro's lemma (or Remark~\ref{rmk:RN_Ind}). 
This settles the case $H=\{1\}$.

If $r=1$, we compute 
\[
\L_N\cInd_H^N(\one)\cong
\L_N\Inf^{N/H}_{N}\cInd_{1}^{N/H}(\one)\cong \L_{N/H}\cInd_{1}^{N/H}(\one) \cong
k[\dim N/H],
\]
where the first isomorphism is \eqref{eq:cInd_Inf_commute}, and the second
follows from the assumption $\L_H(\one)\cong \one$. Now, for $r\ge2$, we have
canonical isomorphisms
\begin{align*}
\omega_{H\backslash N} &= \L_N\cInd_H^N(\one)\\
&\cong \L_{N/N_1} \L_{N_1} \cInd_{N_1}^N \cInd_{H}^{N_1}(\one)\\
&\cong \L_{N/N_1} \cInd_1^{N/N_1} \L_{N_1}\cInd_{H}^{N_1}(\one) & & \text{(by
Lemma~\ref{lem:LN_cInd_commute})}\\
&\cong \L_{N/N_1}\cInd_1^{N/N_1}(\one) \otimes_k
\L_{N_1}\cInd_{H}^{N_1}(\one) & & \text{(induced by $\rho_{N/N_1,1}$)}\\
&= \omega_{N/N_1} \otimes_k \omega_{H\backslash N_1}.
\end{align*}
Since $\dim(N/N_1) + \dim(H\backslash N_1) = \dim(H\backslash N)$, it follows
that $\omega_{H\backslash N} \cong k[\dim H\backslash N]$. This finishes the
proof.
\end{proof} 
\subsection{The duality character} 
\label{ss:duality}

\subsubsection{}\label{sss:cyclotomic} 
We denote by $\varepsilon\colon \QQ_p^\times \to k^\times$ the mod $p$ cyclotomic character given by sending $x$ to the image of $xp^{-\val_p(x)}$ in $k$. If $\field/\QQ_p$ is a finite field extension, we put $\varepsilon_{\field} \coloneqq \varepsilon \circ \Nm_{\field/\QQ_p}$, where $\Nm_{\field/\QQ_p}$ denotes the norm map. 

\subsubsection{}\label{sss:d_G} 
The purpose of this section is to give an explicit description of the
character \eqref{eq:omega}. 
Let $G$ be a $p$-adic Lie group of dimension $d$ and denote by $\lie{g}$ its
$\QQ_p$-Lie algebra (cf.~\cite[p.~100]{Schneider.2011}). The determinant of the
adjoint action of $G$ on $\lie{g}$ provides a character $\det\lie{g}
\colon G\to \QQ_p^\times$. 
Put
\[
\lie{d}_G\coloneqq \varepsilon \circ \det(\lie g) \colon G \too k^\times.
\]
Note that, if $G$ is even a Lie group over a finite extension $\field/\QQ_p$ (in the sense of \cite{Schneider.2011}), then $\lie g$ is an $\field$-vector space and $\lie{d}_G = \varepsilon_{\field}\circ \det_{\field}(\lie g)$ by \cite[III, \S9, no.\,4, Proposition~6]{Bourbaki.2007}, where $\det_{\field}(\lie g)$ denotes the determinant of the $\field$-linear action of $G$ on $\lie g$.

\begin{defn}\label{defn:duality-character} 
Let $\Delta\colon G\to G\times G$, $g\mapsto (g,g)$ be the diagonal and write
$G_1 \coloneqq G\times\{1\}$ and $G_2\coloneqq \{1\}\times G$ as subgroups of
$G\times G$. The \emph{dualizing character} of $G$ is defined as
\[
\omega_G \coloneqq \L_{G_2} \cInd_{\Delta(G)}^{G\times G}(\one) \in \D(G),
\]
where we identify $(G\times G)/G_2 \cong G$ via the first projection. 

We leave it as an easy exercise for the reader to check that the automorphism on $G\times G$ given by $(g,h)\mapsto (h,g)$ induces an isomorphism $\omega_G \cong \L_{G_1}\cInd_{\Delta(G)}^{G\times G}(\one)$.
\end{defn} 

\begin{prop}\label{prop:duality-character} 
One has $\omega_G \cong \lie{d}_G[d]$ and $\omega_G^\vee \cong\RRR^{G\times
G}_{\Delta(G)}(\one)$.
\end{prop} 
\begin{proof} 
We compute
\begin{align*}
\iHom(\omega_G,\one) &\cong \R\H^0\bigl(G_2, \iHom(\cInd_{\Delta(G)}^{G\times G}(\one), \one)\bigr) & &
\text{(\cite[Corollary~3.3.7]{Heyer.2022})}\\
&\cong \R\H^0\bigl(G_2, \RInd_{\Delta(G)}^{G\times
G}\iHom\bigl(\one,\RRR^{G\times G}_{\Delta(G)}(\one)\bigr)\bigr) & &
\text{(Lemma~\ref{lem:projection_formula})}\\
&\cong \RRR^{G\times G}_{\Delta(G)}(\one),
\end{align*}
where the last isomorphism uses Remark~\ref{rmk:RN_Ind} and that
canonically $\iHom(\one,X) \cong X$ for each $X\in \D(G)$.

We prove $\omega_G\cong \lie{d}_G[d]$ by showing $\omega_G^{\vee} \cong
\chi_G[-d]$, where $\chi_G$ is Kohlhaase's duality character. The claim then
follows from \cite[Theorem~5.1]{Kohlhaase.2017}. Observe that restriction of
functions induces a $k$-linear isomorphism $\cInd_{\Delta(G)}^{G\times G}(\one)
\Raiso \cInd_1^{G_2}(\one)$. Under this identification the $G\times G$ action on
$\cInd_1^{G_2}(\one)$ is given by $\bigl((g_1,g_2)f\bigr)(g) = f(g_1^{-1}gg_2)$.
In this way, $G$ acts on the $1$-dimensional $k$-vector space
$\Ext^d_G(\cInd_1^{G_2}(k), k)$ through the opposite $G_1$-action on
$\cInd_1^{G_2}(k)$, which is denoted $\chi_G$ in
\cite[Definition~3.12]{Kohlhaase.2017}. We finish by computing
\begin{align*}
\Res^G_1\omega_G^\vee 
&\cong \RHom_{\D(k)}\bigl(\Res^G_1\L_{G_2}\cInd_{\Delta(G)}^{G\times G}(\one),
\one\bigr)\\
&\cong \RHom_{\D(k)}\bigl(\L_{G_2}\Res^{G\times G}_{G_2}
\cInd_{\Delta(G)}^{G\times G}(\one), \one\bigr) & &
\text{(\cite[Proposition~3.2.19]{Heyer.2022})}\\
&\cong \RHom_{\D(G)}\bigl(\Res^{G\times G}_{G_2}\cInd_{\Delta(G)}^{G\times
G}(\one), \one\bigr) & & \text{(adjunction and $\Inf^1_G(\one) = \one$)}\\
&\cong \chi_G[-d]. & & \qedhere
\end{align*}
\end{proof} 

\begin{rmk*} 
The idea for defining the dualizing character as $\omega_G$ comes from the
theory of six functor formalisms. The computation in \cite[Proof of
Theorem~11.6]{Condensed} (which goes back to an idea of Verdier \cite[Proof of
Theorem~3]{Verdier.1969}) shows that $\RRR^{G\times G}_{\Delta(G)}(\one)$
has to be the inverse dualizing object if derived smooth mod $p$
representations were part of a six functor formalism.
\end{rmk*} 

\begin{lem}\label{lem:RRR(1)-invertible} 
Let $H\subgroup G$ be a closed subgroup. Then $\RRR^G_H(\one)$ is a character
concentrated in degree $\dim(G/H)$.
\end{lem} 
\begin{proof} 
Let $K\subgroup G$ be a torsion-free compact open subgroup. By our observations
in \S\ref{sss:compact_induction}, we compute
$\Res^H_{H\cap K}\RRR^G_H(\one) \cong \RRR^K_{H\cap K}\Res^G_K(\one)$.
Therefore, we may assume from the beginning that $G$ and $H$ are compact and
torsion-free. The isomorphism $\Res^G_H\Inf^1_G \Raiso \Inf^1_H$ yields an
isomorphism $\RRR^G_H F^1_G \Raiso F^1_H$ by passing to the right adjoints
twice. Since $F^1_G(\one) = \one[\dim G]$ and $F^1_H(\one) = \one[\dim H]$
(see~\S\ref{sss:coinvariants}), we deduce $\RRR^G_H(\one[\dim G]) \cong
\one[\dim H]$. The assertion follows.
\end{proof} 

\begin{lem}\label{lem:RRR-dualizable} 
Let $H\subgroup G$ be a closed subgroup and let $X,Y\in \D(G)$. The natural map
\[
\RRR^G_H(X)\otimes_k \Res^G_H(Y) \too \RRR^G_H(X\otimes_k Y)
\]
is an isomorphism provided $Y$ is dualizable.
\end{lem} 
\begin{proof} 
Recall that $Y$ is called \emph{dualizable} if $Y^\vee\otimes_k\blank$ is right
adjoint to $Y\otimes_k\blank$. 
Observe that we always have the evaluation map $\ev\colon Y\otimes Y^\vee \to
\one$.
By \cite[\href{https://kerodon.net/tag/02E3}{Tag~02E3}]{kerodon}, being
dualizable is equivalent to the existence of a coevaluation map
$\coev\colon \one\to Y^\vee\otimes Y$ such that the compositions
\begin{gather}
Y \xrightarrow{\id_Y\otimes \coev} Y\otimes_k Y^\vee\otimes_k Y
\xrightarrow{\ev\otimes\id_Y} Y\\
\intertext{and}
Y^\vee \xrightarrow{\coev\otimes\id_{Y^\vee}} Y^\vee\otimes_k Y\otimes_kY^\vee
\xrightarrow{Y^\vee\otimes\ev} Y^\vee
\end{gather}
are the identity morphisms on $Y$ and $Y^\vee$, respectively. As $\D(G)$ is
symmetric monoidal, it follows that $Y^\vee\otimes_k\blank$ is also left adjoint
to $Y\otimes_k\blank$. 
Since $\Res^G_H$ is monoidal, we deduce that $\Res^G_HY$ is
dualizable and that $\Res^G_H(Y^\vee)\raiso (\Res^G_HY)^\vee$. Recall from
Lemma~\ref{lem:projection_formula} the natural isomorphism $\cInd_H^G
(\blank\otimes_k \Res^G_HY) \Raiso (\blank\otimes_kY)\cInd_H^G$ of functors
$\D(H)\to \D(G)$. Passing to the left resp.\ right mates yields
natural transformations
\begin{align}\label{eq:RRR-dualizable-1}
(\blank\otimes_k Y^\vee) \cInd_H^G &\Too \cInd_H^G (\blank\otimes_k
\Res^G_HY^\vee),\\
\label{eq:RRR-dualizable-2}
(\blank\otimes_k\Res^G_HY) \RRR^G_H &\Too \RRR^G_H (\blank\otimes_kY).
\end{align}
Note that \eqref{eq:RRR-dualizable-2} arises from \eqref{eq:RRR-dualizable-1} by
passing to the right adjoints. It thus suffices to show that
\eqref{eq:RRR-dualizable-1} is an isomorphism. 
We contemplate the commutative diagram
\[
\begin{tikzcd}
\cInd_H^G(X)\otimes_k Y^\vee
\ar[r,equals] \ar[d,"\coev"'] &
\cInd_H^G(X)\otimes_k Y^\vee
\ar[d,"\coev"] 
\\
\cInd_H^G(X\otimes_k \Res^G_HY^\vee \otimes_k \Res^G_HY) \otimes_k Y^\vee
\ar[d,"\cong"] \ar[r,"\cong"] &
\cInd_H^G(X)\otimes_k Y^\vee\otimes_kY\otimes_kY^\vee
\ar[d,equals]
\\
\cInd_H^G(X\otimes_k\Res^G_HY^\vee)\otimes_k Y\otimes_k Y^\vee
\ar[d,"\ev"'] \ar[r,"\cong"] &
\cInd_H^G(X)\otimes_k Y^\vee\otimes_kY\otimes_kY^\vee
\ar[d,"\ev"]
\\ 
\cInd_H^G(X\otimes_k\Res^G_HY^\vee)
\ar[r,"\cong"'] &
\cInd_H^G(X)\otimes_k Y^\vee,
\end{tikzcd}
\]
where the isomorphisms are given by the projection formula. The composite of the
left vertical arrows is \eqref{eq:RRR-dualizable-1}, and the composite of the
right vertical arrows is the identity. We deduce that
\eqref{eq:RRR-dualizable-1}, and hence also \eqref{eq:RRR-dualizable-2}, is an
isomorphism.
\end{proof} 

\begin{prop}\label{prop:explicit-character} 
Let $H\subgroup G$ be a closed subgroup. There is an isomorphism
\[
\bigl(\RRR^G_H(\one)\bigr)^\vee \cong \omega_H^\vee \otimes_k \Res^G_H\omega_G.
\]
In particular, if $H$ is open, one has $\Res_H^G\omega_G \cong \omega_H$.
\end{prop} 
\begin{proof} 
We first compute
\begin{align*} 
\Res^{H\times H}_{\Delta(H)} \RRR^{G\times H}_{H\times H}(\one) &= 
\Res^{H\times H}_{\Delta(H)} \RRR^{G\times H}_{H\times H}\bigl(\Inf^G_{G\times
H}(\one)\bigr)\\
&\cong \Res^{H\times H}_{\Delta(H)} \Inf^H_{H\times H}\RRR^G_H(\one) & &
\text{(Lemma~\ref{lem:LN_cInd_commute})}\\
&\cong \RRR^G_H(\one),
\end{align*}
where the inflations are taken along the first projection maps. 
A similar computation shows $\Res^{G\times H}_{\Delta(H)} \RRR^{G\times
G}_{G\times H}(\one) \cong \RRR^G_H(\one)$. Hence, we compute
\begin{align*}
\RRR^G_H(\one) \otimes_k \Res^G_H\RRR^{G\times G}_{\Delta(G)}(\one) 
&\cong
\RRR^G_H\RRR^{G\times G}_{\Delta(G)}(\one) \hspace{10em}
\text{(Lemma~\ref{lem:RRR-dualizable})}\\
&\cong \RRR^{H\times H}_{\Delta(H)} \RRR^{G\times H}_{H\times H} \RRR^{G\times
G}_{G\times H}(\one)\\
&\cong \RRR^{H\times H}_{\Delta(H)}(\one)\otimes_k \Res^{H\times H}_{\Delta(H)}
\RRR^{G\times H}_{H\times H}(\one)\otimes_k \Res^{G\times H}_{\Delta(H)}
\RRR^{G\times G}_{G\times H}(\one)\\
&\cong \RRR^{H\times H}_{\Delta(H)}(\one) \otimes_k \RRR^G_H(\one) \otimes_k
\RRR^G_H(\one).
\end{align*}
Since $\RRR^G_H(\one)$ is invertible by Lemma~\ref{lem:RRR(1)-invertible}, we
can cancel it on both sides. By Proposition~\ref{prop:duality-character} we have
$\Res^G_H\omega_G^\vee \cong \omega_H^\vee \otimes_k \RRR^G_H(\one)$, which is
equivalent to the first assertion. The last assertion follows from the fact that
$\RRR^G_H = \Res^G_H$ if $H$ is open.
\end{proof} 

\begin{cor}\label{cor:explicit-character}
Suppose the hypotheses of Lemma~\ref{lem:omega_LN_cInd} are satisfied: let $N\normal G$ and $H\subgroup G$ be closed subgroups such that $G = HN$ and assume there exists a subnormal series $H\cap N = N_r \normal N_{r-1} \normal \dotsb \normal N_0 = N$ by closed subgroups such that $\L_{N_i}(\one) = \one$ for all $i = 1,\dotsc,r$. Then
\[
\Inf^{G/N}_H\omega_{H\backslash G} = \omega_H\otimes \Res^G_H\omega_G^\vee.
\]
In particular, $\Inf^{G/N}_H\omega_{H\backslash G}$ does not depend on $N$.
\end{cor}
\begin{proof}
Again, we make the identification $H/H\cap N \cong G/N$. As in the proof of Proposition~\ref{prop:duality-character} we compute
\begin{align*}
\omega_{H\backslash G}^\vee 
&= \bigl(\L_N \cInd_H^G(\one)\bigr)^\vee \\
&\cong \R\H^0\bigl(N,(\cInd_H^G(\one))^\vee\bigr) & & \text{(\cite[Corollary~3.3.7]{Heyer.2022})} \\
&\cong \R\H^0\bigl(N, \RInd_H^G \RRR^G_H(\one)\bigr) & & \text{(Lemma~\ref{lem:projection_formula})} \\
&\cong \R\H^0\bigl(H\cap N, \RRR^G_H(\one)\bigr) & & \text{(Remark~\ref{rmk:RN_Ind})} \\
&= \R\H^0\bigl(H\cap N, \RRR^G_H(\one)^{\vee\vee}\bigr)\\
&\cong \bigl(\L_{H\cap N} \RRR^G_H(\one)^\vee\bigr)^\vee & & \text{(\cite[Corollary~3.3.7]{Heyer.2022})}.
\end{align*}
Hence, passing to the inverses yields $\omega_{H\backslash G} \cong \L_{H\cap N} \RRR^G_H(\one)^\vee$. Note that both $\Inf^{G/N}_H\omega_{H\backslash G}$ and $\RRR^G_H(\one)^\vee$ are characters of $H$ concentrated in degree $-\dim(G/H)$ by Lemmas~\ref{lem:omega_LN_cInd} and~\ref{lem:RRR(1)-invertible}, respectively. Since $\L^0_{H\cap N}$ is the usual coinvariants functor, we deduce that $H\cap N$ acts trivially on $\RRR^G_H(\one)^\vee$. In other words, there exists $\omega \in \D(G/N)$ such that $\RRR^G_H(\one)^\vee = \Inf^{G/N}_H(\omega)$. The projection formula for $\L_{H\cap N}$ (\cite[Corollary~3.3.5]{Heyer.2022}) and the hypothesis $\L_{H\cap N}(\one) = \one$ imply
\[
\omega_{H\backslash G} = \L_{H\cap N}\Inf^{G/N}_H(\omega) \cong \L_{H\cap N}(\one) \otimes_k \omega = \omega.
\]
Now apply Proposition~\ref{prop:explicit-character} to finish the proof.
\end{proof}

\section{The Geometrical Lemma}\label{s:geometrical}
\subsection{Setup} 
We fix a finite extension $\field/\QQ_p$ and a field $k$ of characteristic $p$.

\subsubsection{}\label{sss:setup_reductive_group} 
Let $\alg G$ be a connected reductive $\field$-group and $\alg T$ a maximal
$\field$-split torus of $\alg G$. Let $\alg P = \alg M\ltimes \alg U$ and $\alg
Q = \alg L\ltimes \alg N$ be semistandard parabolic $\field$-subgroups of $\alg
G$, that is, the Levi factors $\alg M$ and $\alg L$ both contain 
$\alg Z_{\alg G}(\alg T)$. Denote by $\ol{\alg P} = \alg M\ltimes \ol{\alg U}$ the
parabolic opposite $\alg P$. We note the following:
\begin{enumerate}[label=--]
\item The semistandard condition is not a
restriction, as any two parabolic $\field$-subgroups contain a common
minimal Levi, \cite[4.18. Corollaire]{Borel-Tits.1965}. 
\item For any $n\in \alg N_{\alg G}(\alg T)(\field)$ also $n\alg P n^{-1}$ is
semistandard.
\item One has a decomposition $\alg P\cap \alg Q = (\alg M\cap \alg L)\ltimes
\bigl((\alg U\cap \alg L)(\alg M\cap \alg N)(\alg U\cap\alg N)\bigr)$,
cf.~\cite[Proposition~2.8.2 and Theorem~2.8.7]{Carter.1985}.
\end{enumerate}

Recall our convention that for an algebraic $\field$-group $\alg H$ the
corresponding lightface letter $H\coloneqq \alg H(\field)$ denotes its
associated group of $\field$-points. Then $H$ is a $p$-adic Lie group, and
we denote by $\Rep_k(H)$ the abelian category of smooth $k$-linear representations
of $H$. We write $\Chain(H)$ (resp.\ $\K(H)$, resp.\ $\D(H)$) for the
category of unbounded complexes (resp.\ the homotopy category of unbounded
complexes, resp.\ the unbounded derived category) of $\Rep_k(H)$.

\subsubsection{}\label{sss:parabolic_induction} 
We write $\pInd_P^G \coloneqq \Ind_P^G\circ\Inf^M_P \colon \D(M)\to \D(G)$ for the
functor of \emph{parabolic induction}. It has a left adjoint $\L(U,\blank)
\coloneqq \L_U\circ \Res^G_P$ and a right adjoint $\R(U,\blank) \coloneqq
\R\H^0(U,\blank)\circ \RRR^G_P$. For any integer $i$ we denote by $\L^i(U,\blank),
\R^i(U,\blank)\colon \Rep_k(G)\to \Rep_k(M)$ the corresponding cohomology
functors given by composing $\L(U,\blank)$, resp.\ $\R(U,\blank)$, with $\H^i$.
By the proof of \cite[Theorem~4.1.1]{Heyer.2022} we have natural isomorphisms
\[
\RHom_M\bigl(\L(U,X),Y\bigr) \cong \RHom_G\bigl(X,\pInd_P^GY\bigr)
\quad\text{and}\quad
\RHom_M\bigl(Y,\R(U,X)\bigr) \cong \RHom_G\bigl(\pInd_P^GY, X\bigr)
\]
for all $X\in \D(G)$ and $Y\in \D(M)$.
\medskip

In this section we will study the composite functor
\[
\L(N,\blank)\circ \pInd_P^G \colon \D(M) \too \D(L)
\]
by constructing certain filtrations in \S\ref{ss:filtration} whose graded
pieces will be described explicitly in \S\ref{ss:geometrical}

\subsection{Filtrations}\label{ss:filtration} 
\subsubsection{}\label{sss:PxQ-action} 
The group $P\times Q$ acts continuously on $G$ via $(x,y)\cdot g \coloneqq
xgy^{-1}$. By the Bruhat decomposition, the coset space $P\backslash G/Q$ admits
a finite representing system $\orbitrep{P}{Q} \subseteq
\alg N_{\alg G}(\alg T)(\field)$. It follows from \cite[1.5.
Proposition]{Bernstein-Zelevinski.1976} that the orbits $PnQ$ are locally closed
in $G$. We define a partial order on $\orbitrep{P}{Q}$ via
\[
n \le n' \stackrel{\text{def}}{\iff} Pn'Q \subseteq \ol{PnQ},
\]
where the overline means topological closure.
The following lemma is well-known.

\begin{lem}\label{lem:union_of_orbits_open} 
For each $n\in \orbitrep{P}{Q}$ the subset $X_n\coloneqq
\bigcup_{n'\le n} Pn'Q$ is the smallest open $P\times Q$-invariant subset in $G$
containing $n$. In particular, $PnQ$ is open if $n$ is
minimal and is closed if $n$ is maximal.
\end{lem} 
\begin{proof} 
Note that $G\setminus X_n = \bigcup_{n'\not\le n}Pn'Q$ is
closed, because $n'\not\le n$ and $n'\le n''$ implies $n''\not\le n$. If $Y$ is
any open $P\times Q$-invariant subset containing $n$ and if $n'\le n$, then
$Y\cap \ol{Pn'Q}$ contains $n$. From the definition of topological closure and
the $P\times Q$-invariance of $Y$ it follows that $Pn'Q \subseteq Y$. Hence
$X_n\subseteq Y$. The last assertions are immediate. 
\end{proof} 

\begin{rmk*} 
\begin{enumerate}[label=(\alph*)]
\item Every open $P\times Q$-invariant subset of $G$ is a union of $X_n$'s.
\item The poset $\orbitrep{P}{Q}$ has a smallest and a largest element but is in
general not totally ordered.
\end{enumerate}
\end{rmk*} 

\subsubsection{}\label{sss:compact_induction2} 
We extend the notation introduced in \S\ref{sss:compact_induction}.
Let $Z \subseteq G$ be a $P\times Q$-invariant subset. For any $V\in \Rep_k(P)$
we denote by $\cInd_P^ZV$ the $k$-vector space of locally constant functions
$f\colon Z\to V$ which satisfy $f(xz) = xf(z)$, for all $x\in P$, $z\in Z$, and
have compact support in $P\backslash Z$. The group $Q$ acts smoothly by right
translation on $\cInd_P^ZV$. Observe that $\cInd_P^ZV \subseteq \cInd_P^GV$
provided $Z$ is open in $G$, that is, every element of $\cInd_P^ZV$ is fixed by
an open compact subgroup of $G$.
The functor $\cInd_P^Z\colon \Rep_k(P)\to
\Rep_k(Q)$ is exact and hence extends to a triangulated functor on the derived
categories:
\[
\cInd_P^Z\colon \D(P)\too \D(Q).
\]

The next lemma is well-known, cf.~\cite[Lemma~6.1.1]{Casselman.1995}.

\begin{lem}\label{lem:cInd_exact_sequence} 
Let $Z \subseteq Z' \subseteq G$ be $P\times Q$-invariant subsets such that $Z$
is open in $Z'$. For every smooth $P$-representation $V$ one has an exact
sequence
\[
0 \too \cInd_P^Z V \too \cInd_P^{Z'} V \too \cInd_{P}^{Z'\setminus Z}V \too 0
\] 
of smooth $Q$-representations.
Here, the first map is extension by zero and the second map is
restriction of functions.
\end{lem} 
\begin{proof} 
Choose a continuous section of the projection map $G\to P\backslash G$,
cf.~\cite[Lemma~2.3]{AHV.2019}; for any $P\times Q$-invariant subset
$Y\subseteq G$ it restricts to a continuous section $\sigma\colon P\backslash
Y\to Y$. For any $k$-vector space $W$ we denote by $\CCC^\infty_c(P\backslash Y,
W)$ the space of locally constant functions $P\backslash Y\to W$ with compact
support. We have $k$-linear isomorphisms $\cInd_P^YV \raiso
\CCC^\infty_c(P\backslash Y,V) \laiso \CCC^\infty_c(P\backslash Y,k)\otimes_kV$,
where the first isomorphism is given by $f\mapsto f\circ \sigma$ and the second
by $f\otimes v\mapsto [Py\mapsto f(Py)v]$. Put $Z''\coloneqq Z'\setminus Z$.
Under all these identifications the sequence in the lemma arises by applying
$\blank\otimes_kV$ to the sequence
\begin{equation}\label{eq:extend_loc_const_functions}
0\too \CCC^\infty_c(P\backslash Z,k) \too \CCC^\infty_c(P\backslash Z',k) \too
\CCC^\infty_c(P\backslash Z'',k) \too 0.
\end{equation}
It therefore suffices to show that \eqref{eq:extend_loc_const_functions} is
exact. Injectivity of the first map and exactness in the middle are clear.
Let $f\colon P\backslash Z''\to k$ be a locally constant function with compact
support. 
Write $\Supp(f) = \bigsqcup_{i=1}^r \Omega''_i$ for some compact open subsets $\Omega''_i \subseteq P\backslash Z''$ such that $f$ is constant on each $\Omega''_i$. 
Choose compact open subsets $\Omega'_i \subseteq P\backslash Z'$ such that $\Omega'_i \cap P\backslash Z'' = \Omega''_i$ for all $i$ (recall that $P\backslash Z'$ admits a basis consisting of compact open subsets). Replacing each $\Omega'_i$ by $\Omega'_i \setminus \bigcup_{j\neq i}\Omega'_j$ if necessary, we may further assume that the $\Omega'_i$ are pairwise disjoint. Extending each $f_{|\Omega''_i}$ constantly to $\Omega'_i$ then yields an extension of $f$ in $\CCC_c^\infty(P\backslash Z', k)$.
\end{proof} 

\begin{rmk}\label{rmk:Z_cup_Z'_disjoint} 
If $Z'\setminus Z$ is open in $Z'$, then the sequence in
Lemma~\ref{lem:cInd_exact_sequence} splits, so that we obtain a natural
isomorphism $\cInd_P^{Z'} \cong \cInd_P^{Z} \oplus \cInd_P^{Z'\setminus Z}$ of functors $\D(P) \to \D(Q)$.
\end{rmk} 

\begin{defn}\label{defn:filtration} 
Let $F\colon \cat C\to \cat D$ be a triangulated functor.
A \emph{(ascending) filtration} of length $n$ on $F$ is a sequence of natural
transformations
\begin{equation}\label{eq:filtration}
F_i \xRightarrow{\lambda_i} F_{i+1} \xRightarrow{\mu_{i+1}} F_{i+1,i}
\xRightarrow{\nu_i} F_i[1] 
\end{equation}
of triangulated functors $\cat C\to \cat D$, for $i=0,1,\dotsc,n-1$, such
that:
\begin{enumerate}[label=--]
\item $F_0 = 0$ and $F_n = F$;
\item evaluating \eqref{eq:filtration} at any object of $\cat C$ gives a
distinguished triangle in $\cat D$.
\end{enumerate}
The functor $F_{i,i-1}$ is called the \emph{$i$-th graded piece} of the
filtration; note that $\mu_1\colon F_1\Raiso F_{1,0}$ is an isomorphism.
\comment{TODO: Find a better definition of filtration.}
\end{defn} 

\begin{lem}\label{lem:filtration} 
Let $Z\subseteq Z' \subseteq G$ be $P\times Q$-invariant subsets such that $Z$
is open in $Z'$. There
exists a natural transformation $\nu\colon \cInd_P^{Z'\setminus Z}\To
\cInd_P^Z[1]$ of triangulated functors $\D(P)\to \D(Q)$ such that for all $X \in
\D(P)$ one has a distinguished triangle 
\[
\cInd_P^ZX \too \cInd_P^{Z'}X \too
\cInd_P^{Z'\setminus Z}X \stackrel{\nu}{\too}
\cInd_P^ZX[1].
\]
\end{lem} 
\begin{proof} 
Denote by $\lambda\colon \cInd_P^Z\To \cInd_P^{Z'}$ and $\mu\colon \cInd_P^{Z'}\To
\cInd_P^{Z'\setminus Z}$ the obvious maps.
In order to construct $\nu$ let $F\colon \Chain(P)\to \Chain(Q)$ be the cone of
$\lambda$. More precisely, given a complex $(X,d)$ in $\Chain(P)$, define 
$F(X)^n\coloneqq \cInd_P^ZX^{n+1}\oplus \cInd_P^{Z'}X^n$ with differential
$F(X)^n\to F(X)^{n+1}$ given by the matrix $\begin{psmallmatrix}
-\cInd_P^Zd^{n+1} & 0\\ \lambda & \cInd_P^{Z'}d^n\end{psmallmatrix}$. It is clear
that $F\colon \Chain(P)\to \Chain(Q)$ is indeed a functor which comes with
natural
transformations $\nu_1\colon F\To \cInd_P^Z[1]$ and $\nu_2\colon F\To
\cInd_P^{Z'\setminus Z}$ given by $(\id,0)$ and $(0,\mu)$, respectively. In
combination with Lemma~\ref{lem:cInd_exact_sequence}, the usual long exact
sequence argument and the five lemma show that $\nu_2(X)$ is a
quasi-isomorphism for all $X\in \Chain(P)$. Since
$\cInd_P^Z$ and $\cInd_P^{Z'}$ are exact and mapping cones are functorial in
$\Chain(Q)$, a similar argument shows that $F$ preserves quasi-isomorphisms.
Hence $F$ descends to a functor $\D(P)\to \D(Q)$, and then $\nu\coloneqq \nu_1
\circ \nu_2^{-1}$ is the desired natural transformation.
\end{proof} 

\subsubsection{}\label{sss:Phi_Z} 
For any $P\times Q$-invariant subset $Z \subseteq G$ we define
\[
\Phi_Z \coloneqq \L(N,\blank) \circ \cInd_P^Z\circ \Inf^M_P \colon \D(M)\to
\D(L).
\]
Observe that, if $Z = \bigsqcup_{i} Pn_iQ$ is a union of open orbits in $Z$,
then $\Phi_Z \cong \bigoplus_{i} \Phi_{Pn_iQ}$.

\begin{prop}\label{prop:filtration} 
Let $\emptyset = Z_0 \subseteq Z_1\subseteq Z_2\subseteq \dotsb \subseteq Z_r$
be a chain of open $P\times Q$-invariant subsets of $Z_r$. There exists a
filtration
\[
0 \Too \Phi_{Z_1} \Too \Phi_{Z_2} \Too \dotsb \Too \Phi_{Z_r}
\]
of $\Phi_{Z_r}$, whose $i$-th graded piece is $\Phi_{Z_i\setminus Z_{i-1}}$.
\end{prop} 
\begin{proof} 
Use Lemma~\ref{lem:filtration}.
\end{proof} 

\begin{defn}\label{defn:height} 
Let $\emptyset = Z_0\subseteq Z_1 \subseteq Z_2\subseteq \dotsb \subseteq Z_r =
G$ be the unique chain of open $P\times Q$-invariant subsets of $G$ for which
the $P\times Q$-orbits in $Z_i\setminus Z_{i-1}$ are the open orbits in
$G\setminus Z_{i-1}$ of maximal dimension. 
We call $\height(\orbitrep{P}{Q}) \coloneqq r$ the \emph{height} of $\orbitrep{P}{Q}$. 
The \emph{height} $\height(n)$ of $n\in \orbitrep{P}{Q}$ is defined as
the smallest integer $i$ with $n\in Z_{i}$. 
\end{defn} 

\subsection{The Geometrical Lemma}\label{ss:geometrical} 
We keep the notations of the previous subsections.

\begin{prop}\label{prop:geometrical} 
The following diagram commutes:
\[
\begin{tikzcd}
\D(M\cap Q) \ar[d,"\L_{M\cap N}"'] \ar[r,"\Inf^{M\cap Q}_{P\cap Q}"]
&
\D(P\cap Q) \ar[r,"\cInd_{P\cap Q}^Q"]
&
\D(Q) \ar[d,"\L_N"]
\\
\D(M\cap L) \ar[r,"\omega\otimes\blank"']
&
\D(M\cap L) \ar[r,"\pInd_{P\cap L}^L"']
&
\D(L),
\end{tikzcd}
\]
where $\omega \coloneqq \Res^{P\cap L}_{M\cap L}\L_N
\cInd_{P\cap Q}^{(P\cap L)N}(\one)$ is a character concentrated in
degree $-\dim(\ol U\cap N)$.
\end{prop} 
\begin{proof} 
We first verify the commutativity of the diagram
\[ 
\begin{tikzcd}
\D(M\cap Q)
\ar[d,"\L_{M\cap N}"'] \ar[r,"\Inf^{M\cap Q}_{P\cap Q}"]
&
\D(P\cap Q)
\ar[d,"\L_{P\cap N}"'] \ar[r,"\cInd_{P\cap Q}^{(P\cap L)N}"]
&[2em]
\D\bigl((P\cap L)N\bigr)
\ar[d,"\L_N"] \ar[r,"\cInd_{(P\cap L)N}^{Q}"]
&[2em]
\D(Q)
\ar[d,"\L_N"]
\\
\D(M\cap L)
\ar[r,"\Inf^{M\cap L}_{P\cap L}"']
&
\D(P\cap L)
\ar[r,"\omega'\otimes\blank"']
&
\D(P\cap L)
\ar[r,"\cInd_{P\cap L}^L"']
&
\D(L),
\end{tikzcd}
\] 
where $\omega'\coloneqq \omega_{(P\cap Q)\backslash(P\cap L)N} \in \D(P\cap L)$
is defined as in \eqref{eq:omega}. The left diagram commutes by
Lemma~\ref{lem:LN_Inf_commute} applied to $(G,H,N) = (P\cap Q, U\cap Q, P\cap
N)$; note that $\L_{U\cap N}(\one)\cong \one$ by
\cite[Example~3.4.24]{Heyer.2022}.
The diagram on the right commutes by Lemma~\ref{lem:LN_cInd_commute}.
Since $\alg N$ is a unipotent algebraic group, the hypotheses of
Lemma~\ref{lem:omega_LN_cInd} are easily seen to be satisfied. Hence, also the
middle diagram commutes and $\omega'$ is a character of $P\cap L$ concentrated in
degree $-\dim(\ol U\cap N)$; here, the
computation of the degree uses the decomposition $(P\cap L)N = (P\cap Q)\times
(\ol U\cap N)$ as $p$-adic manifolds. As the only character of $U\cap L$ is the
trivial one, we have
$\omega' = \Inf^{M\cap L}_{P\cap L}\omega$. Hence, since
$(\omega'\otimes_k\blank)\Inf^{M\cap L}_{P\cap L} = \Inf^{M\cap L}_{P\cap L}
(\omega\otimes_k\blank)$ and compact induction is transitive, this proves the
assertion.
\end{proof} 

\begin{notation} 
Given a closed subgroup $H\subgroup G$ and $g\in G$, we write $g(H) \coloneqq
gHg^{-1}$ and denote by $g_*\colon \D(H)\raiso \D(g(H))$ the ``inflation'' along
the conjugation map $g(H)\raiso H$.
\end{notation} 

\begin{lem}\label{lem:cInd_isomorphism} 
There is a natural isomorphism 
\[
\cInd^{PnQ}_P \Raiso \cInd^Q_{n^{-1}(P)\cap Q}
\Res^{n^{-1}(P)}_{n^{-1}(P)\cap Q} n_*^{-1}
\]
of functors $\D(P)\to \D(Q)$.
\end{lem} 
\begin{proof} 
Note that, by \cite[1.6. Corollary]{Bernstein-Zelevinski.1976}\footnote{The
condition ``countable at infinity'' is automatic for $\field$-points of linear
algebraic groups.}, the inclusion $Q\hookrightarrow n^{-1}(P)Q$ and the homeomorphism $n^{-1}(P)Q\xrightarrow{n\cdot} PnQ$ induce homeomorphisms $n^{-1}(P)\cap Q\backslash Q \raiso n^{-1}(P)\backslash n^{-1}(P)Q \raiso P\backslash PnQ$. It follows that for any $V \in \Rep_k(P)$ the natural map
\begin{align*}
\cInd_P^{PnQ}V &\too \cInd_{n^{-1}(P)\cap Q}^Q \Res^{n^{-1}(P)}_{n^{-1}(P)\cap
Q} n^{-1}_*V,\\
f &\longmapsto [q\mapsto f(nq)]
\end{align*}
is an isomorphism in $\Rep_k(Q)$. As all functors involved are
exact, the assertion follows.
\end{proof} 

\subsubsection{}\label{sss:omega_n} 
For any $n\in \orbitrep{P}{Q}$ we consider the functors $\D(M)\to \D(L)$ given by
\begin{align*}
\Phi_{PnQ} &= \L_N \circ \cInd^{PnQ}_P \circ\Inf^M_P,
\\
\Psi_n &\coloneqq \pInd_{n^{-1}(P)\cap L}^L \circ (\omega_n\otimes_k\blank) \circ
n^{-1}_* \circ \L(M\cap n(N), \blank),
\end{align*}
where $\omega_n\coloneqq \Res^{n^{-1}(P)\cap L}_{n^{-1}(M)\cap
L}\L_N\cInd_{n^{-1}(P)\cap Q}^{(n^{-1}(P)\cap
L)N}(\one)$ is by Lemma~\ref{lem:omega_LN_cInd} a character of $n^{-1}(M)\cap L$
concentrated in degree $-\dim(n^{-1}(\ol U)\cap N)$.

\begin{thm}\label{thm:geometrical} 
For every $n\in \orbitrep{P}{Q}$ there is a natural isomorphism $\Phi_{PnQ}
\Raiso \Psi_n$. 
\end{thm} 
\begin{proof} 
Applying Lemma~\ref{lem:cInd_isomorphism} and Proposition~\ref{prop:geometrical}
to $(P,Q) = (n^{-1}(P),Q)$, we compute
\begin{align*}
\Phi_{PnQ} &\Raiso \L_N\cInd_{n^{-1}(P)\cap Q}^{Q}
\Res^{n^{-1}(P)}_{n^{-1}(P)\cap Q} n^{-1}_* \Inf^M_P\\
&= \L_N \cInd_{n^{-1}(P)\cap Q}^Q\Inf^{n^{-1}(M)\cap Q}_{n^{-1}(P)\cap Q}
n^{-1}_* \Res^M_{M\cap n(Q)}\\
&\Raiso i^L_{n^{-1}(P)\cap L} (\omega_n\otimes_k\blank) \L_{n^{-1}(M)\cap N}
n^{-1}_* \Res^M_{M\cap n(Q)}\\
&\cong \Psi_n. \qedhere
\end{align*}
\end{proof} 

\comment{TODO: Does this theorem have any implications for
\cite[Conjecture~3.3.4]{Hauseux.2018}?}

\begin{cor}[Geometrical Lemma]\label{cor:geometrical} 
Let $\emptyset = Z_0 \subseteq Z_1\subseteq Z_2 \subseteq \dotsb \subseteq Z_r =
G$ be a chain of open $P\times Q$-invariant subsets of $G$ such that
$Z_i\setminus Z_{i-1} = Pn_iQ$ for all $1\le i\le r$, where $n_i\in
\orbitrep{P}{Q}$. It induces a filtration on $\Phi_G = \L(N,\blank) \pInd_P^G$ whose
$i$-th graded piece is $\Psi_{n_i}$.
\end{cor} 
\begin{proof} 
Combine Proposition~\ref{prop:filtration} and Theorem~\ref{thm:geometrical}.
\end{proof} 

\begin{cor}\label{cor:geometrical-2} 
The functor $\L(N,\blank)\circ \pInd_P^G$ admits a filtration
of length $\height(\orbitrep{P}{Q})$ whose $i$-th graded piece is
$\bigoplus_{n, \height(n)=i} \Psi_n$.
\end{cor} 
\begin{proof} 
Combine Proposition~\ref{prop:filtration}, \S\ref{sss:Phi_Z}, and
Theorem~\ref{thm:geometrical}.
\end{proof} 

\begin{rmk} 
There is a dual version of Corollary~\ref{cor:geometrical} which states that
there exists a descending filtration on $\R(U,\blank) \pInd_Q^G$ with graded pieces
(in some order) $\pInd_{M\cap n(Q)}^M (\Omega_n \otimes_k\blank)
n_*\R\bigl(n^{-1}(U)\cap L,\blank\bigr)$, where $\Omega_n \coloneqq \Res^{P\cap
n(Q)}_{M\cap n(L)} \RRR^{(P\cap n(L))n(N)}_{P\cap n(Q)}(\one)$ is a character
concentrated in degree $\dim(\ol U\cap n(N))$ and $n\in \orbitrep{P}{Q}$. However,
for the proof it seems we need to employ the theory of $\infty$-categories.
Granted this more general framework, we could deduce this version by passing
everywhere in the filtration to the right adjoints. 
\comment{TODO: explain in more detail in an appendix?}
\end{rmk} 

\section{Applications}\label{s:applications}
\subsection{Setup and notation}\label{ss:appl_setup} 
\subsubsection{} 
We fix a finite field extension $\field/\QQ_p$ and a coefficient field $k$ of characteristic $p$.

\subsubsection{}\label{sss:setup_reductive} 
Let $\alg{G}$ be a connected reductive $\field$-group. 
Fix a maximal $\field$-split torus $\alg{T}\subgroup \alg{G}$.
We choose a set of simple roots $\Delta_G$ inside the (relative) reduced root
system $\Sigma\coloneqq \Phi(\alg{G}, \alg{T})^{\mathrm{red}}$ together with
the associated set $\Sigma^+$ of positive roots, which corresponds to some
minimal parabolic subgroup $\alg{B}$ containing $\alg{Z}\coloneqq
\alg{Z}_{\alg{G}}(\alg{T})$.
Put
\[
\Delta_{G}^1\coloneqq \set{\alpha\in \Delta_G}{\dim_{\field} \alg{U}_\alpha=1},
\]
where $\alg{U}_\alpha$ denotes the root $\field$-group normalized by $\alg{Z}$ and  associated with $\alpha\in \Sigma$, see \cite[5.2]{Borel-Tits.1965} for the definition. 
We denote by $\Weyl$ the finite Weyl group associated with $(\alg{G},\alg{T})$.
We fix once and for all a set $\NNN \subseteq
\alg{N}_{\alg{G}}(\alg{T})(\field)$ of
representatives of $\Weyl$ and denote by $n_w$ the element of $\NNN$ corresponding
to $w\in \Weyl$. We denote by $n_\alpha \in \NNN$ the element lifting the simple
reflection $s_\alpha$ attached to $\alpha\in \Delta_G$. 
For each $w\in \Weyl$ we put
\[
d_w \coloneqq d_{n_w} \coloneqq \sum_{\alpha\in \Sigma^+\cap
w^{-1}(-\Sigma^+)} \dim_{\field}\alg{U}_{\alpha}.
\]
We also abbreviate $d_{\alpha} \coloneqq d_{s_\alpha}$ for $\alpha\in\Sigma$.

To each standard parabolic subgroup $\alg{P} = \alg{M}\ltimes \alg{U}$
corresponds a subset $\Delta_M \subseteq \Delta_G$. Set
\[
\Delta^{\perp}_{M} \coloneqq \set{\alpha\in\Delta_G}{\text{$\langle \alpha,
\beta^\vee\rangle = 0$ for all $\beta\in \Delta_M$}}
\quad\text{and}\quad
\Delta_{M}^{\perp,1} \coloneqq \Delta_{M}^{\perp}\cap \Delta_G^1.
\]
We denote by $\Sigma_M = \Phi(\alg{M},\alg{T})^{\red} \subseteq \Sigma$ the root system
of $\alg{M}$; the positive roots corresponding to $\Delta_M$ are $\Sigma_M^+ =
\Sigma_M\cap \Sigma^+$. We denote by $\Weyl_M \subgroup \Weyl$ the finite Weyl
group associated with $(\alg{M},\alg{T})$.
\subsubsection{}\label{sss:distinguished-reps} 
Fix standard parabolic subgroups $\alg{P}=\alg{M}\ltimes\alg{U}$ and
$\alg{Q}=\alg{L}\ltimes\alg{N}$ of $\alg{G}$. 
We will choose a distinguished set $\orbitrep{P}{Q}$ of double coset
representatives of $P\backslash G/Q$ as follows: Put
\[
D_M\coloneqq \set{w\in \Weyl}{w(\Delta_M) \subseteq \Sigma^+}
\]
and define $D_L$ similarly. For the various properties of the set
\[
D_{M,L}\coloneqq (D_M)^{-1}\cap D_L = \set{w\in \Weyl}{\text{$w\in D_L$ and
$w^{-1}\in D_M$}}
\]
we refer to \cite[\S2.7]{Carter.1985}. Let us only mention that $D_{M,L}$ is a
set of double coset representatives of $\Weyl_M\backslash \Weyl/\Weyl_L$ and
that each $w\in D_{M,L}$ is the unique element of minimal length in $\Weyl_M
w\Weyl_L$. Finally, let $\orbitrep{P}{Q} \subseteq \NNN$ be the subset
corresponding to $D_{M,L}$.
\begin{lem}\label{lem:dw=dim} 
For each $w\in D_{M,L}$ one has
\[
[\field:\QQ_p]\cdot d_w = \dim\bigl(n_w^{-1}(\ol{U}) \cap N\bigr).
\]
\end{lem} 
\begin{proof} 
Note that $\dim(\cdot) = [\field:\QQ_p]\cdot \dim_{\field}(\cdot)$. 
Put $\Sigma_{N} =
\Sigma^+\setminus\Sigma^+_L$ and $\Sigma_{\ol{U}} =
-(\Sigma^+\setminus\Sigma^+_M)$.
Then
\[
n_w^{-1}(\ol{\alg{U}})\cap \alg{N}= \prod_{\alpha\in \Sigma_N\cap
w^{-1}(\Sigma_{\ol U})}
\alg{U}_\alpha.
\]
The claim thus reduces to $\Sigma_N\cap w^{-1}(\Sigma_{\ol{U}}) =
\Sigma^+\cap w^{-1}(-\Sigma^+)$. By definition of $D_{M,L}$ we have
$w^{-1}(\Sigma^+_M) \subseteq \Sigma^+$ and $\Sigma^+_L \subseteq
w^{-1}(\Sigma^+)$.
Hence,
\begin{align*}
\Sigma_N \cap w^{-1}(\Sigma_{\ol{U}}) &= \Sigma^+\setminus \Sigma^+_L \cap
-w^{-1}(\Sigma^+ \setminus \Sigma_M^+)\\
&= \Sigma^+\setminus w^{-1}(-\Sigma_M^+) \cap w^{-1}(-\Sigma^+)\setminus
\Sigma_L^+ = \Sigma^+ \cap w^{-1}(-\Sigma^+). \qedhere
\end{align*}
\end{proof} 
\begin{notation}\label{nota:delta_w} 
For each $w\in D_{M,L}$, let $\delta_w \in \Rep_k(n_w^{-1}(M)\cap L)$ be the
character defined by
\[
\omega_w\coloneqq \omega_{n_w} = \delta_{w} \bigl[[\field:\QQ_p]d_{w}\bigr],
\]
see \S\ref{sss:omega_n} and Lemma~\ref{lem:dw=dim}. 
If $w$ is the simple reflection corresponding to $\alpha\in \Delta_G$, we also
write $\delta_\alpha$ for $\delta_{w}$ and $\omega_\alpha$ for $\omega_w$.
\end{notation} 

\begin{lem}\label{lem:delta_w}
Let $w\in D_{M,L}$ and let $\delta'_w \colon n_w^{-1}(M) \cap L \to \field^\times$ be the determinant of the $\field$-linear adjoint action on $\Lie(n_w^{-1}(\ol{U}) \cap N)$. One has that
\[
\delta_w = \varepsilon_{\field} \circ \delta'_w.
\]
\end{lem}
\begin{proof}
By Corollary~\ref{cor:explicit-character} we have
\[
\omega_{w} = \Res^{n_w^{-1}(P)\cap Q}_{n_w^{-1}(M)\cap L}\omega_{n_w^{-1}(P) \cap Q}^{\vee} \otimes \Res^{(n_w^{-1}(P)\cap L)N}_{n_w^{-1}(M)\cap L}\omega_{(n_w^{-1}(P)\cap L)N}
\]
and hence Proposition~\ref{prop:duality-character} implies $\delta_w = \lie{d}_{n_w^{-1}(P)\cap Q}^{-1}\cdot \lie{d}_{(n_w^{-1}(P)\cap L)N}$ as characters of $n_w^{-1}(M)\cap L$. The assertion now follows from \S\ref{sss:d_G} and the observation that we have an $n_w^{-1}(M)\cap L$-equivariant decomposition $\Lie\bigl((n_w^{-1}(P)\cap L)N\bigr) = \Lie\bigl(n_w^{-1}(P)\cap Q\bigr) \oplus \Lie\bigl(n_w^{-1}(\ol{U})\cap N\bigr)$. 
\end{proof}


\subsection{Computation of \texorpdfstring{$\Ext$}{Ext}-groups} 
Recall the setup in \S\ref{sss:setup_reductive}
\begin{ex}\label{ex:LU(PS)} 
Let $\chi\colon M\to k^\times$ be a smooth character. We then have
\begin{equation}\label{eq:PS}
\L^{-j}(N,\pInd_P^G\chi) = \bigoplus_{\substack{w\in D_{M,L}\\ [\field:\QQ_p]d_w =
j}} \pInd_{n_w^{-1}(P)\cap L}^L
\bigl(\delta_w\otimes_k n_{w*}^{-1}\chi\bigr).
\end{equation}
Indeed, fix $j\ge0$. Consider the open $P\times Q$-invariant subsets $Z\subseteq
Z'\subseteq G$ defined by $Z \coloneqq
\bigsqcup_{\substack{w\in D_{M,L}\\ [\field:\QQ_p]d_w>j}}Pn_wQ$ and
$Z'\coloneqq \bigsqcup_{\substack{w\in D_{M,L}\\ [\field:\QQ_p]d_w\ge j}}Pn_wQ$.
By Proposition~\ref{prop:filtration} we obtain a filtration
\[
\Phi_Z(\chi) \too \Phi_{Z'}(\chi) \too \Phi_G(\chi) = \L(N,\pInd_P^G\chi).
\]
Applying $\H^{-j}(\blank)$ to the triangle $\Phi_{Z'}(\chi) \to \Phi_{G}(\chi)
\to \Phi_{G\setminus Z'}(\chi)\xrightarrow{+}$ yields an exact sequence
\begin{equation}\label{eq:PS-1}
\H^{-j-1}\bigl(\Phi_{G\setminus Z'}(\chi)\bigr) \too
\H^{-j}\bigl(\Phi_{Z'}(\chi)\bigr) \too \H^{-j}\bigl(\Phi_G(\chi)\bigr) \too
\H^{-j}\bigl(\Phi_{G\setminus Z'}(\chi)\bigr).
\end{equation}
By Proposition~\ref{prop:filtration} and Theorem~\ref{thm:geometrical} there
exists a filtration on $\Phi_{G\setminus Z'}(\chi)$ with graded pieces of the
form $\pInd_{n_w(P)\cap L}^L \bigl(\delta_w\bigl[[\field:\QQ_p]d_w\bigr]\otimes_k
n_{w*}^{-1}\chi\bigr) \in \D^{\ge 1-j}(L)$; here we have used that, since $M\cap n_w(N)$ acts trivially on $\chi$, we have $\L(M\cap n_w(N),\chi) \cong \chi$ viewed as a character of $M\cap n_w(L)$. Since $\D^{\ge 1-j}(L)$ is closed under extensions, we deduce $\Phi_{G\setminus Z'}(\chi) \in \D^{\ge 1-j}(L)$. 
Hence, we have $\H^{-j-1}(\Phi_{G\setminus Z'}(\chi)) = \H^{-j}(\Phi_{G\setminus
Z'}) = 0$, and then \eqref{eq:PS-1} shows $\H^{-j}(\Phi_{Z'}(\chi)) \raiso
\H^{-j}(\Phi_G(\chi))$. A similar argument applied to the triangle
$\Phi_{Z}(\chi) \to \Phi_{Z'}(\chi) \to \Phi_{Z'\setminus Z}(\chi)
\xrightarrow{+}$ implies
\[
0 = \H^{-j}\bigl(\Phi_Z(\chi)\bigr) \too \H^{-j}\bigl(\Phi_{Z'}(\chi)\bigr)
\xrightarrow{\cong} \H^{-j}(\Phi_{Z'\setminus Z}(\chi)\bigr) \too
\H^{-j+1}\bigl(\Phi_Z(\chi)\bigr) = 0.
\]
The discussion shows $\L^{-j}(N,\pInd_P^G\chi) = \H^{-j}(\Phi_G(\chi))\cong
\H^{-j}(\Phi_{Z'\setminus Z}(\chi))$.
Moreover, by Remark~\ref{rmk:Z_cup_Z'_disjoint} and
Theorem~\ref{thm:geometrical}, we have
\begin{align*}
\H^{-j}\bigl(\Phi_{Z'\setminus Z}(\chi)\bigr) &\cong 
\bigoplus_{\substack{w\in D_{M,L}\\{[\field:\QQ_p]}d_w = j}}
\H^{-j}\bigl(\Phi_{Pn_wQ}(\chi)\bigr)\\ 
&\cong \bigoplus_{\substack{w\in D_{M,L}\\{[\field:\QQ_p]}d_w=j}} 
\H^{-j}\bigl(\Psi_{n_w}(\chi)\bigr) \\
&\cong \bigoplus_{\substack{w\in D_{M,L}\\ {[\field:\QQ_p]}d_w=j}}
\pInd_{n_w^{-1}(P)\cap L}^L(\delta_w\otimes_k n_{w*}^{-1}\chi),
\end{align*}
where the last isomorphism again uses the fact that $\L(M\cap n_w(N),\chi)\cong \chi$ viewed as a character of $M\cap n_w(L)$. This proves~\eqref{eq:PS}.
\end{ex} 
\begin{lem}\label{lem:Ext=0} 
Let $V,W \in \Rep_k(G)$. Assume there exists a central element $z\in G$ such
that the action of $z$ on $V$ and $W$ is given by multiplication with distinct
scalars. Then
\[
\Ext_G^i(V,W) = 0, \qquad \text{for all $i\ge0$.}
\]
\end{lem} 
\begin{proof} 
This is well-known, see \cite[I, \S4]{Borel-Wallach.2000}.
\end{proof} 

\begin{rmk}\label{rmk:Hauseux_conjecture} 
In \cite[Conjecture~3.17]{Hauseux.2019} Hauseux states a conjecture regarding
the computation of $\Ext^1$-groups of representations which are
parabolically induced from a supersingular representation. This conjecture is
conditionally resolved in \cite[Corollary~5.2.9]{Hauseux.2018}; parts of the
argument rely, however, on an open conjecture of Emerton,
\cite[Conjecture~3.7.2]{EmertonII}, which states that the higher ordinary parts
functor $H^\bullet\Ord_P$ is the right derived functor of $\Ord_P$ on the
category of locally admissible representations. Granting this conjecture,
Hauseux is able to prove even more general forms of his conjecture, see
particularly \cite[Remarks~5.2.6 and~5.2.8]{Hauseux.2018}. Further, Hauseux
computes higher Ext-groups for principal series representations in
\cite[Th\'eor\`eme~5.3.1]{Hauseux.2016}, which again relies on the conjecture of
Emerton.

In Theorems~\ref{thm:PS} and~\ref{thm:Ext} below I give unconditional proofs of
the computation, resp.\ the generalized conjecture, of Hauseux in a slightly
different context: where Hauseux is computing
(higher) extensions in the category of locally admissible representations, our
computations naturally work in the category of all smooth representations. 
\end{rmk} 

\subsubsection{} 
Observe (\eg, by \cite[V.2.6 Lemme]{Renard.2010} applied to $G = Z$) that restriction of characters induces an inclusion $\X^*(Z) \subseteq
\X^*(T)$ with finite index, which is $\Weyl$-equivariant for the natural actions
of $\Weyl$ on $\X^*(Z)$ and $\X^*(T)$, respectively. (Note that $\Weyl$ acts on $X^*(Z)$ since $\alg{N}_{\alg{G}}(\alg{T})(\field)$ normalizes $Z$ and $Z$ acts trivially on $X^*(Z)$.) For each $\alpha\in\Sigma$, the
composite $\alg{T}\subseteq\alg{Z}\xrightarrow{\Ad} \Aut(\Lie\alg{U}_\alpha) \to
\Aut(\bigwedge^{d_\alpha} \Lie\alg{U}_\alpha) \cong \mathbb
G_{\mathrm{m},\field}$ coincides with $d_\alpha\alpha$. 
Hence, we have $d_\alpha\alpha \in \X^*(Z)$ under the above inclusion. 

For any $w\in \Weyl$, we put $\alpha_w\coloneqq \sum_{\beta\in\Sigma^+\cap
w^{-1}(-\Sigma^+)}d_{\beta}\beta \in \X^*(Z)$, which we also view as a character $Z^\times\to \field^\times$. By Lemma~\ref{lem:delta_w} the character $\delta_w$ from
Notation~\ref{nota:delta_w} is given by $\delta_w = \varepsilon_\field\circ
\alpha_w$, where $\varepsilon_{\field}$ is the composition of the norm map $\Nm_{\field/\QQ_p}$ and the mod $p$ cyclotomic character $\varepsilon\colon \QQ_p^\times \to k^\times$.

Inspired by the definition in \cite[Definition~5.2.1]{Buzzard-Gee.2010}, we call
a character $\theta\in \X^*(Z)$ a \emph{twisting element} if $\langle \theta,
\alpha^\vee\rangle = d_\alpha$ for all $\alpha\in \Delta_G$. If $\rho\coloneqq
\frac{1}{2} \sum_{\alpha\in\Sigma^+} d_\alpha\alpha$ lies in $\X^*(Z)$ (\eg, if
$\alg{G}$ is $\field$-split, semisimple, and simply connected), then we
show below that $\rho$ is a twisting element.
\begin{lem}~ 
\label{lem:cocycle}
\begin{enumerate}[label=(\alph*)]
\item\label{lem:cocycle-a} One has $\alpha_w = \rho-w^{-1}(\rho)$ in
$\X^*(Z)\otimes_{\ZZ}\QQ$, for all $w\in \Weyl$.
\item\label{lem:cocycle-b} The function $\Weyl \to \X^*(Z)$, $w\mapsto \alpha_w$
is injective and satisfies the cocycle condition $\alpha_{vw} = \alpha_w +
w^{-1}(\alpha_v)$ for all $v,w\in \Weyl$.
\item\label{lem:cocycle-c} The assignment $\chi\star w \coloneqq \delta_w
\otimes_kw_*^{-1}\chi$ defines a right action of $\Weyl$ on the group of smooth
characters of $Z$.
\item\label{lem:cocycle-d} One has $\langle \rho, \alpha^\vee\rangle = 0$ in $X^*(Z)\otimes_{\ZZ}\QQ$, for all $\alpha \in \Delta_G$. In particular, if $\rho\in\X^*(Z)$, then $\rho$ is a twisting
element.
\item\label{lem:cocycle-e} Assume there exists a twisting element $\theta$. Then
$\chi\star w = (\varepsilon_\field\circ\theta)
\otimes_kw^{-1}_*\bigl(\chi\otimes_k \varepsilon_\field^{-1}\circ\theta\bigr)$
for all $w\in \Weyl$ and all smooth characters $\chi\colon Z\to k^\times$.
\end{enumerate}
\end{lem} 
\begin{rmk*} 
The action $(\chi,w)\mapsto \chi\star w$ is reminiscent of the ``dot action''
in the representation theory of semisimple Lie algebras.
\end{rmk*} 
\begin{proof}[Proof of Lemma~\ref{lem:cocycle}] 
Note that $d_\alpha = d_{w(\alpha)}$ for all $\alpha\in\Sigma$, $w\in \Weyl$.
For \ref{lem:cocycle-a} we compute
\begin{align*}
2w^{-1}(\rho) &= \sum_{\alpha\in \Sigma^+}d_\alpha\cdot w^{-1}(\alpha)
= \sum_{\alpha\in \Sigma^+\cap w(\Sigma^+)} d_{w^{-1}(\alpha)} w^{-1}(\alpha) +
\sum_{\alpha\in \Sigma^+\cap w(-\Sigma^+)} d_{w^{-1}(\alpha)} w^{-1}(\alpha)\\
&= \sum_{\alpha\in w^{-1}(\Sigma^+)\cap \Sigma^+} d_\alpha\alpha
-\sum_{\alpha\in w^{-1}(-\Sigma^+)\cap \Sigma^+} d_\alpha\alpha
= 2\rho -2\alpha_w.
\end{align*}
The cocycle condition in \ref{lem:cocycle-b} can be verified in
$\X^*(Z)\otimes_{\ZZ}\QQ$, in which case it is obvious from \ref{lem:cocycle-a}.
If $v,w\in \Weyl$ are such that $\alpha_v =\alpha_w$, then $\alpha_v =
\alpha_{vw^{-1}w} = \alpha_w + w^{-1}(\alpha_{vw^{-1}})$, and hence
$\alpha_{vw^{-1}} = 0$; but this necessitates $v=w$. Part \ref{lem:cocycle-c}
follows from \ref{lem:cocycle-b} noting that also $w\mapsto \delta_w =
\varepsilon_{\field}\circ\alpha_w$ is a cocycle. Let us prove
\ref{lem:cocycle-d}. For each $\alpha\in\Delta_G$ we compute, using
\ref{lem:cocycle-a},
\[
\langle \rho, \alpha^\vee\rangle = \langle s_\alpha(\rho),
s_\alpha(\alpha^\vee)\rangle = \langle \rho-d_\alpha\alpha, -\alpha^\vee\rangle
= -\langle \rho, \alpha^\vee\rangle +d_\alpha\langle \alpha,\alpha^\vee\rangle.
\]
In combination with $\langle \alpha,\alpha^\vee\rangle = 2$, this shows $\langle
\rho, \alpha^\vee\rangle = d_\alpha$. Finally, we prove \ref{lem:cocycle-e}.
Note that $\langle \theta-\rho,\alpha^\vee\rangle = 0$ for all
$\alpha\in\Delta_G$, by \ref{lem:cocycle-d}. Hence, $\Weyl$ fixes $\theta-\rho$ (\eg, by \cite[21.2 Theorem]{Borel.1991}),
and therefore $\theta -w^{-1}(\theta) = \rho-w^{-1}(\rho) = \alpha_w$ in
$\X^*(Z)$ (the computation is carried out in $\X^*(Z)\otimes_{\ZZ}\QQ$). We
deduce $\delta_w = (\varepsilon_\field\circ\theta)\otimes_k
w^{-1}_*\bigl(\varepsilon_\field^{-1}\circ\theta\bigr)$, from which the
assertion follows.
\end{proof} 

\begin{thm}\label{thm:PS} 
Let $\alg{P}=\alg{M}\ltimes\alg{U}$ a standard parabolic subgroup of $\alg{G}$
and write $\alg{B} = \alg{Z}\ltimes\alg{N}$. Denote by $C(Z)$ the center of $Z$.
Let $\chi\colon M\to k^\times$ be a smooth character and let $r\in\ZZ_{\ge0}$.
\begin{enumerate}[label=(\alph*)]
\item\label{thm:PS-a} Let $\chi'\colon Z\to k^\times$ be another smooth
character. If
\[
\Ext_G^r\bigl(\pInd_P^G\chi, \pInd_B^G\chi'\bigr) \neq 0,
\]
then there exists $w\in D_{M,Z}$ such that $\chi' \cong \delta_w\otimes_k
n_{w*}^{-1}\chi$ as characters of $C(Z)$ and with $[\field:\QQ_p]d_w \le r$. 
\item\label{thm:PS-b} Assume $\delta_w\otimes_k n_{w*}^{-1}\chi\not\cong \chi$
as characters of $C(Z)$, for all $w\in D_{M,Z}$. For
each $w\in D_{M,Z}$ with $[\field:\QQ_p]d_w \le r$ one has $k$-linear
isomorphisms
\[
\Ext^r_{G}\bigl(\pInd_P^G \chi, \pInd_B^G (\delta_w\otimes_k n_{w*}^{-1}\chi)\bigr)
\cong \Ext_Z^{r-[\field:\QQ_p]d_w}(\one,\one) \cong
\H^{r-[\field:\QQ_p]d_w}(Z,k),
\]
where $\H^*(Z,k)$ denotes the continuous group cohomology of $Z$ with
coefficients in $k$.
\end{enumerate}
\end{thm} 
\begin{proof} 
By
\cite[Corollary~4.1.3]{Heyer.2022}, there is a spectral sequence
\[
E_2^{i,j} = \Ext_Z^i\bigl(\L^{-j}(N,\pInd_P^G\chi), \chi'\bigr) \implies
\Ext_G^{i+j}\bigl(\pInd_P^G\chi, \pInd_B^G\chi'\bigr).
\]
Moreover, \eqref{eq:PS} reads $\L^{-j}(N,\pInd_P^G\chi)\cong
\bigoplus_{\substack{w\in D_{M,Z}\\ [\field:\QQ_p]d_w=j}} \delta_w\otimes_k
n_{w*}^{-1}\chi$. 
We prove \ref{thm:PS-a} by showing the contrapositive. If
after restriction to $C(Z)$ we
have $\chi'\not\cong\delta_w\otimes_k n_{w*}^{-1}\chi$ for all $w\in D_{M,Z}$
with $[\field:\QQ_p]d_w\le r$, then
Lemma~\ref{lem:Ext=0} shows that $E_2^{i,j} = 0$ for all $j\le r$ and all $i$.
But this implies $\Ext_G^r(\pInd_P^G\chi, \pInd_B^G\chi') = 0$.
We now prove \ref{thm:PS-b}. Take $w\in D_{M,Z}$ with $[\field:\QQ_p] d_w\le r$
and put $\chi' \coloneqq \delta_w\otimes_k n_{w*}^{-1}\chi$.
By the assumption and Lemma~\ref{lem:cocycle}.\ref{lem:cocycle-c} we have
$\chi'\not\cong \delta_v\otimes_kn_{v*}^{-1}\chi$
for all $v\neq w$ (as characters of $C(Z)$). Together with Lemma~\ref{lem:Ext=0}
we deduce $E_2^{i,j} = 0$ whenever $j\neq [\field:\QQ_p]d_w$, and 
\[
E_2^{i,[\field:\QQ_p]d_w} \cong \Ext_Z^{i}\bigl(\chi',\chi'\bigr) \cong
\Ext_Z^i(\one,\one) \cong \H^i(Z,k),
\]
where the third isomorphism follows from
\cite[Theorem~1.1]{Fust.2022}. Hence, the spectral sequence collapses on the
second page and gives an isomorphism $\Ext_Z^{r-[\field:\QQ_p]d_w}\bigl(\chi',
\chi'\bigr) \cong \Ext_G^r\bigl(\pInd_P^G\chi,\pInd_B^G \chi'\bigr)$.
\end{proof} 

\comment{TODO: Does the theorem generalize to a computation of Ext-groups
between generalized Steinberg representations?}

\begin{defn}\label{defn:left_cuspidal} 
A smooth representation $V\in \Rep_k(G)$ is called \emph{left cuspidal} 
(resp.\ \emph{right cuspidal}) if for all proper parabolic subgroups $\alg{P} =
\alg{M}\ltimes \alg{U}$ of $\alg G$ it holds that $\L^0(U,V) = 0$ 
(resp.\ $\R^0(U,V)=0$).
\end{defn} 

\begin{rmk} 
By \cite[Corollary~6.9]{AHV.2019}, an irreducible admissible $G$-representation
is left and right cuspidal if and only if it is supercuspidal.
\end{rmk} 

\begin{thm}\label{thm:Ext} 
Let $\alg{P}, \alg{Q}$ and $\orbitrep{P}{Q}$ be as in
\S\ref{sss:distinguished-reps}. Let $V\in \Rep_k(M)$ and $W\in
\Rep_k(L)$.
\begin{enumerate}[label=(\alph*)]
\item\label{thm:Ext-a} Assume $\alg{P}\nsubseteq\alg{Q}$ and
$\alg{P}\nsupseteq\alg{Q}$, that $V$ is left cuspidal and that $W$ is right
cuspidal. Then
\[
\Ext_G^1\bigl(\pInd_P^GV, \pInd_Q^GW\bigr) = 0.
\]

\item\label{thm:Ext-b} Assume $\alg{P}=\alg{Q}$. For each $0\le i <
[\field:\QQ_p]$, the functor $\pInd_P^G$ induces an isomorphism
\[
\Ext_M^i(V,W)\raiso \Ext_G^i\bigl(\pInd_P^GV, \pInd_P^GW\bigr),
\]
and there is an exact sequence
\begin{gather}
0 \too \Ext_M^{[\field:\QQ_p]}(V,W) \too \Ext_G^{[\field:\QQ_p]} \bigl(\pInd_P^GV,
\pInd_P^GW\bigr) \too X \too \Ext_M^{[\field:\QQ_p]+1}(V,W),\\
\intertext{where}
X\coloneqq \bigoplus_{\alpha\in \Delta_G^1\setminus\Delta_M}
\Hom_{n_\alpha^{-1}(M)\cap M} \bigl(\delta_{\alpha} \otimes_k n_{\alpha *}^{-1}
\L^0(M\cap n_\alpha(U), V), \R^0(n_\alpha^{-1}(U)\cap M, W)\bigr).
\end{gather}

\item\label{thm:Ext-c} Assume $\alg{P}=\alg{Q}$, that $V$ is left cuspidal or
$W$ is right cuspidal, and that $V$ and $W$ admit distinct central characters.
Then
\[
\Ext_G^{[\field:\QQ_p]}\bigl(\pInd_P^GV,\pInd_P^GW\bigr) \cong \bigoplus_{\alpha\in
\Delta_M^{\perp,1}} \Hom_M\bigl(\delta_{\alpha}\otimes_k n_{\alpha *}^{-1}V,
W\bigr).
\]

\item\label{thm:Ext-d} Assume $\alg{P}\supsetneqq \alg{Q}$ and that $V$ is left
cuspidal. For all $0\le i\le [\field:\QQ_p]$, the functor $\pInd_P^G$ induces an
isomorphism
\[
\Ext_M^i\bigl(V, \pInd_{M\cap Q}^MW\bigr) \raiso \Ext_G^i\bigl(\pInd_P^GV, \pInd_Q^GW\bigr).
\]

\item\label{thm:Ext-e} Assume $\alg{P}\subsetneqq \alg{Q}$ and that $W$ is right
cuspidal. For all $0\le i\le [\field:\QQ_p]$, the functor $\pInd_Q^G$ induces an
isomorphism
\[
\Ext_L^i\bigl(\pInd_{P\cap L}^LV, W\bigr) \raiso \Ext_G^i\bigl(\pInd_P^GV, \pInd_Q^GW\bigr).
\]
\end{enumerate}
\end{thm} 
\begin{proof} 
First note that the condition that $w\in D_{M,L}$ satisfies $d_w\le 1$ is equivalent to $w = s_\alpha$ for some $\alpha\in \Delta_G^1$.
Thus the $P\times Q$-invariant open subsets $\bigl(G\setminus \bigcup_{\substack{w
\in D_{M,L}\\ d_w\le 1}}Pn_wQ\bigr) \subseteq \bigl(G\setminus PQ\bigr)
\subseteq G$ induce by Proposition~\ref{prop:filtration} and
Theorem~\ref{thm:geometrical} two distinguished triangles
\begin{align}\label{eq:Ext-1}
Y \too \L\bigl(N,\pInd_P^GV\bigr) &\too Y' \xrightarrow{+}\\
\intertext{and} \label{eq:Ext-2}
\bigoplus_{\substack{\alpha\in\Delta_G^1\\ n_\alpha\in \orbitrep{P}{Q} }}
\pInd_{n_\alpha^{-1}(P)\cap L}^L \omega_\alpha \otimes_k n_{\alpha *}^{-1}
\L\bigl(M\cap n_\alpha(N),V\bigr) &\too Y' \too \pInd_{P\cap L}^L \L(M\cap N,V)
\xrightarrow{+}
\end{align}
in $\D(L)$ such that $\H^i(Y) = 0$ for each $i>-2[\field:\QQ_p]$. Applying
$\RHom_L(\blank,W)$ to \eqref{eq:Ext-1} and using
$\RHom_G\bigl(\pInd_P^GV, \pInd_Q^GW\bigr) \cong \RHom_L\bigl(\L(N,\pInd_P^GV), W\bigr)$, we
obtain a distinguished triangle
\[
\RHom_L(Y',W) \too \RHom_G\bigl(\pInd_P^QV, \pInd_Q^GW\bigr) \too \RHom_L(Y,W)
\xrightarrow{+},
\]
where $\H^i\RHom_L(Y,W) = 0$ for all $i<2[\field:\QQ_p]$. Thus, for all
$i<2[\field:\QQ_p]$ we have an isomorphism
\begin{equation}\label{eq:Ext-3}
\H^i\RHom_L(Y',W)\raiso \Ext_G^i\bigl(\pInd_P^GV, \pInd_Q^GW\bigr).
\end{equation}
We now turn to the proofs of the statements.
\bigskip

We prove \ref{thm:Ext-a}. From $\alg{P}\nsubseteq\alg{Q}$ we deduce that
$\alg{M}\cap\alg{Q}$ is a proper parabolic of $\alg{M}$, and similarly from
$\alg{P}\nsupseteq\alg{Q}$ it follows that $\alg{P}\cap\alg{L}$ is a proper
parabolic of $\alg{L}$. As $V$ is left cuspidal and $W$ is right cuspidal, we
obtain
\[
\L^0(M\cap N,V) = 0 = \R^0(U\cap L,W).
\]
Consequently, it follows that
\[
\RHom_L\bigl(\pInd_{P\cap L}^L \L(M\cap N,V), W\bigr) \cong \RHom_{M\cap L}\bigl(
\L(M\cap N,V), \R(U\cap L,W)\bigr)
\]
lies in $\D^{\ge2}(k)$.
\begin{claim*} 
Let $\alpha\in \Delta_G$. Then $\alg{M}\cap n_\alpha(\alg{N}) =\{1\}
=n_\alpha^{-1}(\alg{U})\cap \alg{L}$ implies $\alg{P} = \alg{Q}$.
\begin{proof}[Proof of the claim]
The condition $\alg{M}\cap n_\alpha(\alg{N}) = \{1\}$ implies $\Sigma_M\cap
s_\alpha(\Sigma^+\setminus \Sigma_L) = \emptyset$, where $s_\alpha$ denotes
the simple reflection attached to $\alpha$. Since $-\Sigma_M =
\Sigma_M$, we deduce $\Sigma_M\subseteq s_\alpha(\Sigma_L)$. Similarly, the
condition $n_\alpha^{-1}(\alg{U}) \cap \alg{L} = \{1\}$ implies
$\Sigma_L \subseteq s_\alpha^{-1}(\Sigma_M)$. Hence, we have $\Sigma_M =
s_\alpha(\Sigma_L)$. If $s_\alpha \in D_{M,L}$, then $s_\alpha(\Sigma^+_L)
\subseteq \Sigma^+\cap \Sigma_M = \Sigma^+_M$ and in fact $s_\alpha(\Sigma_L^+)=
\Sigma_M^+$. But then $s_\alpha(\Delta_L)= \Delta_M$ by the uniqueness of root
bases. Since for $\beta\in \Delta_L$ the root $s_\alpha(\beta) = \beta - \langle
\beta,\alpha^\vee\rangle \alpha$ is simple only if $\langle
\beta,\alpha^\vee\rangle =0$, it follows that $\alpha\in \Delta_L^{\perp}$ and
hence that $\Delta_L = s_\alpha(\Delta_L) = \Delta_M$. This entails $\Sigma_L =
\Sigma_M$. If $s_\alpha\notin D_{M,L}$, then $s_\alpha\in \Weyl_M$ or
$s_\alpha\in \Weyl_L$ in which case it is clear that $\Sigma_M =\Sigma_L$. In
any case, we obtain $\alg{M} = \alg{L}$ and hence also $\alg{P} = \alg{Q}$.
\end{proof}
\end{claim*} 

The (contrapositive of the) claim and the assumptions $\alg{P}\nsubseteq\alg{Q}$
and $\alg{P}\nsupseteq\alg{Q}$ imply that $\alg{M}\cap n_\alpha(\alg{Q})$ is a
proper parabolic subgroup of $\alg{M}$ or $n_\alpha^{-1}(\alg{P})\cap \alg{L}$
is a proper parabolic subgroup of $\alg{L}$. As $V$ is left cuspidal and $W$ is
right cuspidal, it follows that $\L^0(M\cap n_\alpha(N),V) = 0$ or
$\R^0(n_\alpha^{-1}(U)\cap L,W) = 0$. Since also $\omega_\alpha =
\delta_\alpha\bigl[[\field:\QQ_p]d_{\alpha}\bigr]$, we deduce that
\[
\begin{split}
\RHom_L\bigl(\pInd_{n_\alpha^{-1}(P)\cap L}^L &\omega_\alpha\otimes_k
n_{\alpha *}^{-1} \L\bigl(M\cap n_\alpha(N),V\bigr), W\bigr)\\
&\cong \RHom_{n_\alpha^{-1}(M)\cap L}\bigl( \omega_\alpha\otimes_k n_{\alpha
*}^{-1} \L\bigl(M\cap n_\alpha(N),V\bigr), \R\bigl(n_\alpha^{-1}(U)\cap L,
W\bigr)\bigr)
\end{split}
\]
lies in $\D^{\ge[\field:\QQ_p]d_{\alpha}+1}(k)$. Hence, applying $\H^1
\RHom_L(\blank,W)$ to \eqref{eq:Ext-2} and using \eqref{eq:Ext-3}, we deduce
\[
\Ext_G^1\bigl(\pInd_P^GV,\pInd_Q^GW\bigr) \cong \H^1\RHom_L(Y',W) \cong
\H^1\RHom_L\bigl(\pInd_{P\cap L}^L \L(M\cap N,V), W\bigr) = 0.
\]

We now prove \ref{thm:Ext-b}, so assume $\alg{P} = \alg{Q}$. For $\alpha\in
\Delta_G^1$ we have $\omega_\alpha = \delta_\alpha\bigl[[\field:\QQ_p]\bigr]$,
and hence the complex
\begin{equation}\label{eq:Ext-4}
\bigoplus_{\substack{\alpha\in \Delta_G^1\\
n_\alpha\in\orbitrep{P}{P}} }\RHom_M\bigl(\pInd_{n_\alpha^{-1}(P)\cap M}^M
\omega_\alpha \otimes_k n_{\alpha *}^{-1}\L\bigl(M\cap n_\alpha(U), V\bigr),
W\bigr)
\end{equation}
lies in $\D^{\ge [\field:\QQ_p]}(k)$. Since
$\H^{[\field:\QQ_p]}\bigl(\eqref{eq:Ext-4}\bigr)
\cong X$, the assertion follows by applying $\H^i\RHom_M(\blank,W)$ to
\eqref{eq:Ext-2} and using \eqref{eq:Ext-3}.
\bigskip

For \ref{thm:Ext-c} we observe that $\Ext_M^i(V,W) = 0$ for all $i\ge0$ by
Lemma~\ref{lem:Ext=0} and the assumption that $V$ and $W$ admit distinct central
characters. Note that $\alg{M} \cap n_\alpha(\alg{P})= \alg{M}$ if and only if
$\alg{M} = n_\alpha(\alg{M})$ if and only if $\alpha\in (\Delta_M
\cup\Delta_M^{\perp})$. As
$V$ is left cuspidal or $W$ is right cuspidal, we deduce $\L^0\bigl(M\cap
n_\alpha(U),V\bigr) = 0$ or $\R^0\bigl(n_\alpha^{-1}(U)\cap M, W\bigr) = 0$, for
each $\alpha\in \Delta_G\setminus (\Delta_M\cup \Delta_M^{\perp})$. Now, the
assertion follows from \ref{thm:Ext-b}.
\bigskip

The proofs of \ref{thm:Ext-d} and \ref{thm:Ext-e} are symmetric, so we only
prove \ref{thm:Ext-d}. Since $V$ is left cuspidal, we have $\L^0\bigl(M\cap
n_\alpha(U),V\bigr) = 0$ for all $\alpha\in \Delta_G \setminus
(\Delta_M \cup \Delta_M^{\perp})$. For all $\alpha\in\Delta_M^{\perp}$ we have
\[
\Hom_M\bigl(\delta_\alpha\otimes_k n_{\alpha *}^{-1}V, \pInd_{M\cap Q}^MW\bigr) = 0,
\]
because with $V$ also $\delta_{\alpha}\otimes_k n_{\alpha *}^{-1}V$ is left
cuspidal and $\alg{M}\cap\alg{Q}$ is a proper parabolic
subgroup of $\alg{M}$. Hence, the assertion follows from \ref{thm:Ext-b}.
\end{proof} 
\subsection{Generalized Steinberg representations} 
We will need the following well-known result, which we will prove for the
convenience of the reader.

\begin{lem}\label{lem:submodule_lattice} 
Let $R$ be a unital associative ring and $M$ an $R$-module of finite length
whose constituents occur with multiplicity one. Denote by $\JH(M)$ the set of
Jordan--H\"older factors of $M$.
\begin{enumerate}[label=(\alph*)]
\item\label{lem:submodule_lattice-a} For each $V\in \JH(M)$ there exists a
unique minimal submodule $M_V \subgroup M$ which has $V$ as a quotient. The
cosocle of $M_V$ is $V$.

\item\label{lem:submodule_lattice-b} One has $\JH(N_1\cap N_2) = \JH(N_1) \cap
\JH(N_2)$ and $\JH(N_1+N_2) = \JH(N_1)\cup \JH(N_2)$, for all submodules
$N_1,N_2 \subgroup M$.
\end{enumerate}
\end{lem} 
\begin{proof} 
We first show \ref{lem:submodule_lattice-a}. Let $V\in \JH(M)$ and take any
submodule $N\subgroup M$ minimal with the property that $V$ is a quotient of
$N$. If we had $\cosoc(N)\neq V$, then the preimage of $V$ under the surjection
$N\twoheadrightarrow \cosoc(N)$ is strictly contained in $N$ and has $V$ as a
quotient, contradicting the minimality assumption on $N$.

Let now $N,N'\subgroup M$ be two minimal submodules having $V$ as a quotient.
Consider the short exact sequence $0 \to N' \too N+N' \to N/(N\cap N') \to 0$.
From the multiplicity one assumption, we deduce $V\notin \JH(N/(N\cap N'))$.
Since $V\in \JH(N)$, it follows that $V\in \JH(N\cap N')$. Hence, there exists
$N'' \subgroup N\cap N'$ having $V$ as a quotient. By the minimality of $N$ and
$N'$ we conclude $N = N'' = N'$. This proves the uniqueness claim.
\medskip

For \ref{lem:submodule_lattice-b} it is generally true that $\JH(N_1\cap N_2)
\subseteq \JH(N_1)\cap \JH(N_2)$ and $\JH(N_1+N_2) =\JH(N_1) \cup \JH(N_2)$. The
remaining assertion needs the multiplicity one condition and follows from
\ref{lem:submodule_lattice-a}, since $V\in \JH(N_1) \cap \JH(N_2)$ implies $M_V
\subseteq N_1\cap N_2$ and hence $V\in \JH(N_1\cap N_2)$.
\end{proof} 

\subsubsection{} 
Recall the setup in \S\ref{sss:setup_reductive}. For any $I\subseteq \Delta_G$
we denote by $\alg{P}_I = \alg{M}_I\ltimes \alg{U}_I$ the corresponding standard
parabolic subgroup of $\alg{G}$. Note that $\alg{B} = \alg{P}_\emptyset$. 
\medskip

In the following, we will abbreviate $\pInd_{P_I}^G$ for $\pInd_{P_I}^G(\one)$. By
\cite[Theorem~D]{EGK.2014} (for split $\alg{G}$ with classical root system) and
\cite[Corollary~3.2]{Ly.2015} (for general $\alg{G}$) the Jordan--H\"older
constituents of $\pInd_{P_I}^G$ are the \emph{generalized Steinberg
representations}
\[
\Sp_{P_J}^G \coloneqq \pInd_{P_J}^G \Big/ \sum_{J\subsetneq J'\subseteq \Delta_G}
\pInd_{P_{J'}}^G,
\]
where $I\subseteq J\subseteq \Delta_G$; they are pairwise non-isomorphic and
occur with multiplicity one.

\begin{lem}\label{lem:distributive_lattice} 
Let $I,J,J_1,\dotsc,J_r \subseteq \Delta_G$ be arbitrary subsets. Inside
$\pInd_{B}^G$ we have:
\begin{enumerate}[label=(\alph*)]
\item\label{lem:distributive_lattice-a} $\pInd_{P_I}^G\cap \pInd_{P_J}^G = \pInd_{P_{I\cup
J}}^G$;

\item\label{lem:distributive_lattice-b} $\pInd_{P_I}^G \cap \sum_{i=1}^r
\pInd_{P_{J_i}}^G = \sum_{i=1}^r\bigl(\pInd_{P_I}^G \cap \pInd_{P_{J_i}}^G\bigr)$.
\end{enumerate}
\end{lem} 
\begin{proof} 
\ref{lem:distributive_lattice-a} is clear as $P_{I\cup J}$ is generated as a
group by $P_I$ and $P_J$.

We now prove \ref{lem:distributive_lattice-b}. Since the constituents of $\pInd_B^G$
occur with multiplicity one, we may prove the equality by showing that the sets
of Jordan--H\"older factors are the same. Using
Lemma~\ref{lem:submodule_lattice} we compute
\begin{align*}
\JH\Bigl(\pInd_{P_I}^G \cap \sum_{i=1}^r \pInd_{P_{J_i}}^G\Bigr) &=
\JH(\pInd_{P_I}^G)\cap \bigcup_{i=1}^r\JH(\pInd_{P_{J_i}}^G)\\ 
&= \bigcup_{i=1}^r \bigl(\JH(\pInd_{P_I}^G) \cap \JH(\pInd_{P_{J_i}}^G)\bigr)
= \JH\Bigl(\sum_{i=1}^r(\pInd_{P_I}^G\cap \pInd_{P_{J_i}}^G)\Bigr). \qedhere
\end{align*}
\end{proof} 

\subsubsection{}\label{sss:coefficient_system} 
Choose a total ordering $\curlyle$ on $\Delta_G$. Fix subsets $I_0,I_1 \subseteq
\Delta_G$ and $I,K\subseteq I_1$, and consider the power set
$\PPP(I_1\setminus I_0)$ as a partially ordered set with respect to reverse
inclusion. Define a functor
\begin{align*}
\VVV^{I_0,I_1}_{I,K} \colon \PPP(I_1\setminus I_0) &\too \Rep_k(M_{K}),\\
J &\longmapsto \pInd_{P_{(I\cup J)\cap K}}^{M_{K}},
\end{align*}
where $\VVV^{I_0,I_1}_{I,K}(J\subseteq J')$ is the inclusion $\pInd_{P_{(I\cup
J')\cap K}}^{M_{K}} \hookrightarrow \pInd_{P_{(I\cup J)\cap K}}^{M_{K}}$, and
where for any $I'\subseteq \Delta_G$ we write $\pInd_{P_{I'\cap K}}^{M_K}$
instead of $\pInd_{P_{I'\cap K}\cap M_K}^{M_K}$ for the sake of readability. 
We define a complex
$C_{I_1,K}^\bullet(I,I_0)$ as follows: for each integer $n\le 0$ we put
\[
C_{I_1,K}^{n}(I,I_0) \coloneqq \bigoplus_{\substack{J\subseteq I_1\setminus I_0\\
\lvert J\rvert = -n}} \VVV^{I_0,I_1}_{I,K}(J);
\]
by convention, $C_{I_1, K}^n(I,I_0) = 0$ for $n>0$. The differential
$d^n\colon C_{I_1,K}^n(I,I_0)\to C_{I_1,K}^{n+1}(I,I_0)$ is defined on the $J$-th summand as
\[
(d^n)_J = \sum_{j_0\in J} \varepsilon_{I_0}(J\setminus\{j_0\}, j_0)\cdot
\VVV^{I_0,I_1}_{I,K}\bigl(J\setminus \{j_0\} \subseteq J\bigr),
\]
where $\varepsilon_{I_0}(J\setminus \{j_0\}, j_0) \coloneqq (-1)^{\lvert \{j\in
I_0\sqcup J\mid j\prec j_0\}\rvert}$. It is an easy exercise to show that
$C_{I_1,K}^\bullet(I,I_0)$ is indeed a complex.
\begin{ex} 
The complex $C_{\Delta_G,\Delta_G}^\bullet(I,I)$ can be depicted as
\begin{equation}\label{eq:Sp-resolution}
\begin{tikzcd}[column sep=small]
0 \ar[r] & \one \ar[r] & 
\displaystyle\bigoplus_{\substack{I\subseteq J\subseteq \Delta_G\\ \lvert
\Delta_G\setminus J\rvert = 1}} \pInd_{P_J}^G \ar[r] & 
\dotsb \ar[r] &
\displaystyle\bigoplus_{\substack{j\in \Delta_G\setminus I}} \pInd_{P_{I\cup\{j\}} }^G
\ar[r] &
\pInd_{P_I}^G \ar[r] & 0.
\end{tikzcd}
\end{equation}
Observe that $\H^0\bigl(C_{\Delta_G,\Delta_G}^\bullet(I,I)\bigr) =
\Sp_{P_I}^G$.
\end{ex} 
\begin{prop}\label{prop:coefficient_system} 
The complex $C_{I_1,K}^\bullet(I,I_0)$ is a resolution of
$\H^0\bigl(C_{I_1,K}^\bullet(I,I_0)\bigr)$.
\end{prop} 
\begin{proof} 
In view of Lemma~\ref{lem:distributive_lattice}, this follows from the
general \cite[Corollary~12.4.11]{KS}. 
\end{proof} 
\begin{notation} 
For $w\in \Weyl$, we put
\[
I(w)\coloneqq \set{\alpha\in\Delta_G}{\ell(s_\alpha w) = \ell(w) + 1},
\]
where $\ell(\cdot)$ denotes the length function on $\Weyl$. 
Observe that $I(w) \subseteq \Delta_G$ is the maximal subset such that
$w^{-1}\in D_{I(w)}$.
\end{notation} 

\begin{thm}\label{thm:Generalized_Steinberg} 
Let $I,K\subseteq \Delta_G$ and $j\in\ZZ$. Then
\[
\L^{-j}\bigl(U_K, \Sp^G_{P_I}\bigr) = \bigoplus_{\substack{w\in D_{I,K},\\
[\field:\QQ_p] d_w = j,\\ I(w)\setminus I \subseteq w(K)}}
\pInd_{P_{w^{-1}(I(w))\cap K}}^{M_K} \bigl(\delta_w \otimes_k n_{w *}^{-1}
\Sp_{P_{I\cap w(K)}}^{M_{I(w)\cap w(K)}}\bigr).
\]
\end{thm} 
\begin{proof} 
We compute $\L^{-j}(U_K,\Sp_{P_I}^G)$ via a spectral sequence coming from the
resolution \eqref{eq:Sp-resolution}. For any complex $X^\bullet \in
\Chain(G)$ and any integer $r$, we denote by $\sigma^{\ge r}X$ the brutal
truncation of $X^\bullet$ \cite[Definition~11.3.11]{KS}; these fit into a
distinguished triangle
\[
\sigma^{\ge r+1}X \too \sigma^{\ge r}X \too X^r[-r] \xrightarrow{+}
\]
in $\D(G)$. As, \eg, described in \cite[Lemma~2.3.20]{Heyer.2022} (applied to
$X(r) \coloneqq \sigma^{\ge r}X$), we obtain for any triangulated functor
$F\colon \D(G)\to \D(M_K)$ a spectral sequence
\[
E_1^{r,s} = \H^s(F)(X^r) \implies \H^{r+s}(FX),
\]
which converges provided $X^\bullet$ is bounded. We apply this to the bounded
complex $X^\bullet = C^\bullet_{\Delta_G,\Delta_G}(I,I)$ and the functor $F =
\L(U_K,\blank)$ and
obtain a third quadrant spectral sequence
\[
E_1^{r,s} = \bigoplus_{\substack{I\subseteq J\subseteq \Delta_G\\ \lvert
J\setminus I\rvert = -r}} \L^s\bigl(U_K,\pInd_{P_J}^G\bigr) \implies
\L^{r+s}\bigl(U_K, \Sp_{P_I}^G\bigr).
\]
We will show that this spectral sequence collapses on the second page. Fix any
$s$. The differential $d^r\colon E_1^{r,s} \to E_1^{r+1,s}$ is induced by the
differential of $C_{\Delta_G,\Delta_G}^\bullet(I,I)$. 
In order to analyse $E_1^{\bullet,s}$, we fix $J\subseteq J'\subseteq
\Delta_G$ for the moment. 
The inclusion $\pInd_{P_{J'}}^G \hookrightarrow \pInd_{P_J}^G$ coincides with
$\pInd_{P_{J'}}^G(\eta)$, where $\eta\colon \one \to \pInd_{P_J\cap M_{J'}}^{M_{J'}}$ is
the unit map. From the Geometrical Lemma (Corollary~\ref{cor:geometrical}) we
deduce
\[
\L^s\bigl(U_K, \pInd_{P_J}^G\bigr) =
\bigoplus_{\substack{w\in D_{J',K}\\ [\field:\QQ_p]d_w \le -s}}
\pInd_{w^{-1}(P_{J'})\cap M_K}^{M_K} \Bigl( \delta_w \otimes n_{w *}^{-1} 
\L^{s +[\field:\QQ_p]d_w} \bigl(M_{J'} \cap n_w(U_K), \pInd_{P_J\cap
M_{J'}}^{M_{J'}}\bigr) \Bigr).
\]
Hence, the map $\L^s(U_K, \pInd_{P_{J'}}^G) \to \L^s(U_K, \pInd_{P_J}^G)$ is induced on
the $w$-th summand (where $w$ satisfies $[\field:\QQ_p] d_w = -s$) by applying
the functor
$\pInd_{w^{-1}(P_{J'}) \cap M_K}^{M_K} \bigl(\delta_w\otimes n_{w*}^{-1}
\L^0\bigl(M_{J'}\cap n_w(U_K), \blank\bigr)\bigr)$ to the unit $\eta\colon \one
\to \pInd_{P_J\cap M_{J'}}^{M_{J'}}$. By Kilmoyer's Theorem
\cite[Theorem~2.7.4]{Carter.1985}, we have $M_{J'} \cap n_w(M_K) = M_{J'\cap
w(K)}$. Hence, $P_J\cap M_{J'} \cap w(M_K)$ is the standard parabolic subgroup
of $M_{J'\cap w(K)}$ attached to $J\cap w(K)$. Observe that the diagram
\[
\begin{tikzcd}
\L^0\bigl(M_{J'}\cap n_w(U_K), \one\bigr)
\ar[d,"\cong"'] \ar[r] &
\L^0\bigl(M_{J'}\cap n_w(U_K), \pInd_{P_J\cap M_{J'}}^{M_{J'}}\bigr)
\ar[d,"\cong"]\\
\one \ar[r] & \pInd_{P_{J\cap w(K)}}^{M_{J'\cap w(K)}},
\end{tikzcd}
\]
commutes, where the bottom map is the unit map. 

To summarize, the map $\L^s(U_K, \pInd_{P_{J'}}^G) \to \L^s(U_K, \pInd_{P_J}^G)$ is on
the $w$-th summand given by
\[
\pInd_{w^{-1}(P_{J'})\cap M_K}^{M_K} \delta_w \too \pInd_{w^{-1}(P_{J'})\cap M_K}^{M_K}
\bigl(\delta_w\otimes_k n_{w*}^{-1} \pInd_{P_{J\cap w(K)}}^{M_{J'\cap w(K)}}\bigr).
\]
For any $w\in D_{I,K}$, we consider the complex $C^\bullet(w) \coloneqq
C^\bullet_{I(w), I(w)\cap w(K)}(I,I)$ attached to the functor
\begin{align*}
\VVV^{I,I(w)}_{I,I(w)\cap w(K)} \colon \PPP\bigl(I(w)\setminus I\bigr) &\too \Rep_k(M_{I(w)\cap w(K)}),\\
J &\longmapsto \pInd_{P_{(I\cup J)\cap w(K)}}^{M_{I(w)\cap w(K)}},
\end{align*}
cf.~\S\ref{sss:coefficient_system}. Our analysis shows that
\[
E_1^{\bullet,s} = \bigoplus_{\substack{w\in D_{I,K}\\ [\field:\QQ_p] d_w = -s}}
\pInd_{P_{w^{-1}(I(w))\cap K}}^{M_K} \bigl(\delta_w\otimes_k n_{w*}^{-1}
C^\bullet(w)\bigr).
\]
From Proposition~\ref{prop:coefficient_system} it follows that $C^\bullet(w)$ is
a resolution of
\[
\H^0\bigl(C^\bullet(w)\bigr) = \begin{cases}
\Sp_{P_{I\cap w(K)}}^{M_{I(w)\cap w(K)}}, & \text{if $I(w)\setminus I \subseteq
w(K)$,}\\ 0, & \text{otherwise.}
\end{cases}
\]
We conclude that $E_2^{\bullet,\bullet}$ is supported on the $0$-th column, \ie,
the spectral sequence collapses on the second page. Since
\[
E_2^{0,-j} = \bigoplus_{\substack{w\in D_{I,K},\\ [\field:\QQ_p]d_w = j,\\
I(w)\setminus I\subseteq w(K)}}
\pInd_{P_{w^{-1}(I(w))\cap K}}^{M_K} \bigl(\delta_w\otimes_k n_{w*}^{-1}
\Sp_{P_{I\cap w(K)}}^{M_{I(w)\cap w(K)}}\bigr),
\]
the assertion follows.
\end{proof} 

\begin{cor}\label{cor:Generalized_Steinberg} 
For every $I\subseteq \Delta_G$ and $j\in\ZZ$, one has
\[
\L^{-j}\bigl(U_{\emptyset}, \Sp_{P_I}^G\bigr) =
\bigoplus_{\substack{w\in D_{I,\emptyset},\\ [\field:\QQ_p]d_w = j,\\
I=I(w)}} \delta_w.
\]
(The condition $I=I(w)$ indicates $w\in D_{I,\emptyset} \setminus
\bigcup_{J\supsetneq I} D_{J,\emptyset}$.) \hfill\ensuremath{\square}
\end{cor} 

\appendix
\section{}
\subsection{General abstract nonsense} 
\label{ss:abstract}
The purpose of this appendix is to formulate some formal statements about
adjunctions in monoidal categories. These will be applied in
Lemmas~\ref{lem:LN_Inf_commute}, \ref{lem:LN_cInd_commute},
and~\ref{lem:omega_LN_cInd}.

For the relevant notions concerning monoidal categories we refer to
\cite[\href{https://kerodon.net/tag/00BL}{Section 00BL}]{kerodon}.

\subsubsection{}\label{sss:mates} 
The primary tool for the constructions in this section is the \enquote*{mate correspondence} of \cite[Proposition~2.1]{Kelly-Street.1974}. We quickly recall here the formulation. For two functors $F,G\colon \cat C \to \cat D$ we denote by $\Nat(F,G)$ the class of natural transformations from $F$ to $G$.

\begin{prop}
Consider a (not necessarily commutative) diagram of functors
\begin{equation}\label{eq:mate-diagram}
\begin{tikzcd}[baseline = (current bounding box.center)]
\cat C \ar[d,"F"'] \ar[r, bend left=15, "L"] & \cat D \ar[d,"G"] \ar[l,bend left=15, "R"] \\
\cat C' \ar[r,bend left=15, "L'"] & \cat D' \ar[l,bend left=15, "R'"],
\end{tikzcd}
\end{equation}
where $L$ (resp. $L'$) is left adjoint to $R$ (resp. $R'$). Denote by $\eta\colon \id_{\cat C} \To RL$ and $\eta'\colon \id_{\cat C'} \To R'L'$ the units and by $\varepsilon\colon LR \To \id_{\cat D}$ and $\varepsilon'\colon L'R' \To \id_{\cat D'}$ the counits of the adjunctions. Then there is a natural bijection
\[
\begin{tikzcd}[row sep=0em]
\Nat(L'F, GL) \ar[r,leftrightarrow, "\cong"] & \Nat(FR, R'G) \\
\alpha \ar[r,mapsto] & (R'G\varepsilon) \circ (R'\alpha_R) \circ (\eta'_{FR}), \\
(\varepsilon'_{GL}) \circ (L'\beta_L) \circ (L'F\eta) & \beta \ar[l,mapsto] ,
\end{tikzcd}
\]
which is compatible with horizontal and vertical compositions of the diagrams of the form \eqref{eq:mate-diagram}.
\end{prop}

\subsubsection{} 
\label{sss:setup1}
Let $\cat A$, $\cat B$, $\cat C$, $\cat D$ be monoidal categories and consider a
diagram
\[
\begin{tikzcd}
\cat A \ar[r,"\ol g^*"] \ar[d,"\ol f^*"'] & \cat B \ar[d,"f^*"]\\
\cat C\ar[r,"g^*"'] \ar[ur,Rightarrow, shorten=4mm, "\alpha" description]
& \cat D
\end{tikzcd}
\]
of (strongly) monoidal functors, where $\alpha\colon g^*\ol f^*
\Raiso f^*\ol g^*$ is a monoidal natural isomorphism. We make two
assumptions:
\begin{enumerate}[label=(\textbf{A\arabic*}),series=abstract,
leftmargin=5em]
\item\label{Assumption1} We have the following adjunctions (of the underlying
ordinary categories)
\begin{align*}
f_{(1)} \dashv f^* \dashv f_* \dashv f^{(1)}, \qquad g^* \dashv g_*\\
\ol f_{(1)} \dashv \ol f^* \dashv \ol f_* \dashv \ol f^{(1)}, \qquad \ol g^*
\dashv \ol g_*,
\end{align*}
where the notation ``$F\dashv G$'' means ``$F$ is left adjoint to $G$''.
\end{enumerate}
As $f^*$ is monoidal, we have a natural isomorphism
\[
\mon_f \colon f^*(b') \otimes f^*(b) \raiso f^*(b'\otimes b).
\]
The right mate of $f^*\circ (\blank\otimes b) \To (\blank\otimes f^*(b))
\circ f^*$ yields a natural map
\[
\rpf_f\colon f_*(d) \otimes b \too f_*\bigl(d\otimes f^*(b)\bigr).
\]
Similarly, $\rpf_{\ol f}$ and $\mon_h$, for $h\in \{\ol f, g, \ol g\}$, are
defined. Define $\rpf_g\colon c\otimes g_*(d) \to
g_*(g^*(c)\otimes d)$ as the right mate of $g^*\circ (c\otimes \blank)
\To (g^*(c)\otimes \blank) \circ g^*$ and similarly for $\rpf_{\ol g}$.
The second assumption we make is:
\begin{enumerate}[resume*=abstract]
\item\label{Assumption2}  $\rpf_f$ and 
$\rpf_{\ol f}$ are isomorphisms. In other words: the adjunctions $f^*\dashv f_*$
and $\ol f^* \dashv \ol f_*$ satisfy (right) projection formulas.
\end{enumerate}
Finally, we consider the natural maps (which we do not require to be
isomorphisms)
\[
\renewcommand{\arraystretch}{1.5}
\begin{matrix}
\groth_f\colon & f^{(1)}(b') \otimes f^*(b) & \too & f^{(1)}(b'\otimes b)\\
\wirth_f\colon & f_*\bigl(f^{(1)}(b)\otimes d\bigr) & \too & b\otimes f_{(1)}(d),
\end{matrix}
\]
and similarly for $f$ replaced by $\ol f$. 
Here, the morphism $\groth_f$ arises as the right mate of the map
$\rpf_f^{-1}\colon f_*\circ (\blank \otimes f^*(b)) \To (\blank\otimes
b)\circ f_*$ and the map $\wirth_f$ arises as the left mate of the natural
transformation
$\groth_f\colon (f^{(1)}(b)\otimes \blank) \circ f^*
\To f^{(1)}\circ (b\otimes \blank)$.

Denoting $r(\,\cdot\,)$ and $l(\,\cdot\,)$ the passage to right and left
mates,\footnote{These operations are not well-defined, because there are
usually several ways to pass to right/left mates. However, the context provides
enough information to make the intended meaning unambiguous.}
respectively, we consider the following natural transformations:
\[
\renewcommand{\arraystretch}{1.5}
\begin{matrix*}[r]
\alpha\colon & g^*\ol f^* & \Raiso & f^* \ol g^*\\
r^2(\alpha^{-1})\colon & \ol f_* g_* & \Raiso & \ol g_* f_*\\
\beta\coloneqq r(r(\alpha^{-1})^{-1})\colon & g^*\ol f^{(1)} & \Too &
f^{(1)} \ol g^*\\
\gamma\coloneqq l(r(\alpha)^{-1})\colon & \ol f_{(1)} g_* & \Too & \ol
g_* f_{(1)}.
\end{matrix*}
\]
Note also that $l(\gamma) = l^2(r(\alpha)^{-1}) = l^2(r(\alpha))^{-1} =
l(\alpha)^{-1}$ and therefore $\gamma= r(l(\alpha)^{-1})$.

\begin{lem}\label{lem:abstract-1} 
For all $a',a\in \cat A$ and $d\in\cat D$, the following diagrams commute:
\begin{enumerate}[label=(\alph*),ref=\thelem(\alph*),itemsep=2ex]
\item\label{lem:abstract-a} 
\qquad
\begin{tikzcd}[baseline=(current bounding box.center)] 
g^*\ol f^{(1)} (a') \otimes g^* \ol f^*(a)
\ar[r,"\mon_g"] \ar[d,"\beta \otimes \alpha"']
&
g^*\bigl(\ol f^{(1)}(a') \otimes \ol f^*(a)\bigr)
\ar[r,"g^*\groth_{\ol f}"]
&
g^*\ol f^{(1)}(a'\otimes a)
\ar[d,"\beta"]
\\
f^{(1)}\ol g^*(a') \otimes f^*\ol g^*(a)
\ar[r,"\groth_f"']
&
f^{(1)}\bigl(\ol g^*(a')\otimes \ol g^*(a)\bigr)
\ar[r,"f^{(1)}\mon_{\ol g}"']
&
f^{(1)} \ol g^*(a'\otimes a);
\end{tikzcd} 

\item\label{lem:abstract-b}
\qquad
\begin{tikzcd}[baseline=(current bounding box.center)] 
\ol f_*g_* \bigl(g^*\ol f^{(1)}(a) \otimes d\bigr)
\ar[r,"r^2(\alpha^{-1})"]
&
\ol g_*f_*\bigl(g^*\ol f^{(1)}(a)\otimes d\bigr)
\ar[r,"\ol g_*f_*(\beta\otimes\id)"]
&[1em]
\ol g_*f_*\bigl(f^{(1)}\ol g^*(a) \otimes d\bigr)
\ar[d,"\ol g_*\wirth_f"]
\\
\ol f_*\bigl(\ol f^{(1)}(a)\otimes g_*(d)\bigr)
\ar[u,"\ol f_*\rpf_g"] \ar[d,"\wirth_{\ol f}"']
& &
\ol g_*\bigl(\ol g^*(a) \otimes f_{(1)}(d)\big)
\\
a \otimes \ol f_{(1)}g_*(d)
\ar[rr,"\id\otimes \gamma"']
& &
a \otimes \ol g_* f_{(1)}(d).
\ar[u,"\rpf_{\ol g}"']
\end{tikzcd} 
\end{enumerate}
\end{lem} 
\begin{hideproof} 
To prove \hyperref[lem:abstract-a]{(a)}, we start with the observation that,
since $\alpha$ is a monoidal isomorphism, there is a commutative diagram
\begin{hideeq}\label{eq:abstract-1}
\begin{tikzcd}[proof,baseline = (current bounding box.center)]
f^*(\ol g^*(a') \otimes \ol g^*(a))
\ar[d,"\mon_f^{-1}"',"\cong"] \ar[r,"f^*\mon_{\ol g}"]
&[2em]
f^*\ol g^*(a'\otimes a)
\ar[r,"\alpha^{-1}","\cong"']
&
g^*\ol f^*(a'\otimes a)
\ar[d,"g^*\mon_{\ol f}^{-1}","\cong"']
\\
f^*\ol g^*(a') \otimes f^*\ol g^*(a)
\ar[r,"\alpha^{-1}\otimes \alpha^{-1}"',"\cong"]
&
g^*\ol f^*(a') \otimes g^*\ol f^*(a)
\ar[r,"\mon_g"']
&
g^*(\ol f^*(a')\otimes \ol f(a)).
\end{tikzcd}
\end{hideeq}
We now compute the right mate at $\mon_f^{-1}$:
\[
\begin{tikzcd}[proof]
\Nat\bigl( f^*(\blank\otimes \ol g^*(a)), (\blank\otimes f^*\ol g^*(a)) f^*\bigr)
\ar[r,leftrightarrow]
\ar[d,"(\ol g^*)^*"']
&
\Nat\bigr (\blank\otimes \ol g^*(a)) f_*, f_*(\blank\otimes f^*\ol g^*(a))\bigr)
\ar[d,"(g^*)^*"]
\\
\Nat\bigl(f^*(\blank\otimes \ol g^*(a))\ol g^*, (\blank\otimes f^*\ol g^*(a))
f^*\ol g^*\bigr)
\ar[d,"(\alpha^{-1})_*"']
&
\Nat\bigl((\blank\otimes \ol g^*(a))f_*g^*, f_*(\blank\otimes f^*\ol g^*(a))
g^*\bigr)
\ar[d,"r(\alpha^{-1})^*"]
\\
\Nat\bigl(f^*(\blank\otimes \ol g^*(a))\ol g^*, (\blank\otimes f^*\ol g^*(a))
g^*\ol f^*\bigr)
\ar[r,leftrightarrow]
\ar[d,"(\alpha^{-1})_*"']
&
\Nat\bigl((\blank\otimes \ol g^*(a)) \ol g^*\ol f_*, f_*(\blank\otimes f^*\ol
g^*(a)) g^*\bigr)
\ar[d,"(\alpha^{-1})_*"]
\\
\Nat\bigl(f^*(\blank\otimes \ol g^*(a)) \ol g^*, (\blank\otimes g^*\ol f^*(a))
g^*\ol f^*\bigr)
\ar[r,leftrightarrow]
\ar[d,"(\mon_g)_*"']
&
\Nat\bigl((\blank\otimes \ol g^*(a)) \ol g^*\ol f_*, f_*(\blank\otimes g^*\ol
f^*(a)) g^*\bigr)
\ar[d,"(\mon_g)_*"]
\\
\Nat\bigl(f^*(\blank\otimes \ol g^*(a)) \ol g^*, g^*(\blank\otimes \ol
f^*(a))\ol f^*\bigr)
\ar[r,leftrightarrow]
&
\Nat\bigl((\blank\otimes \ol g^*(a))\ol g^*\ol f_*, f_*g^*(\blank\otimes \ol
f^*(a))\bigr)
\\
\Nat\bigl(f^*\ol g^*(\blank\otimes a), g^*(\blank\otimes \ol f^*(a))\ol
f^*\bigr)
\ar[r,leftrightarrow]
\ar[u,"(\mon_{\ol g})^*"]
&
\Nat\bigl(\ol g^*(\blank\otimes a)\ol f_*, f_*g^*(\blank\otimes \ol
f^*(a))\bigr)
\ar[u,"(\mon_{\ol g})^*"']
\\
\Nat\bigl(g^*\ol f^*(\blank\otimes a), g^*(\blank\otimes \ol f^*(a))\ol
f^*\bigr)
\ar[u,"(\alpha^{-1})^*"]
&
\Nat\bigl(\ol g^*(\blank\otimes a)\ol f_*, \ol g^*\ol f_*(\blank\otimes \ol
f^*(a))\bigr)
\ar[u,"r(\alpha^{-1})_*"']
\\
\Nat\bigl(\ol f^*(\blank\otimes a), (\blank\otimes \ol f^*(a))\ol f^*\bigr)
\ar[r,leftrightarrow]
\ar[u,"(g^*)_*"]
&
\Nat\bigl((\blank\otimes a)\ol f_*, \ol f_*(\blank\otimes \ol f^*(a))\bigr)
\ar[u,"(\ol g^*)_*"'].
\end{tikzcd}
\]
The commutativity of \eqref{eq:abstract-1} means that $\mon_f^{-1}$ in the upper
left and $\mon_{\ol f}^{-1}$ in the lower left map to the same transformation in
the fourth line from the bottom. By the mate correspondence, this means that
$\rpf_f = r(\mon_f^{-1})$ in the upper right and $\rpf_{\ol f} = r(\mon_{\ol
f}^{-1})$ in the lower right map to the same transformation in the fourth line
from the bottom. The resulting commutative diagram is
\[
\begin{tikzcd}[proof]
\ol g^*\ol f_*(c)\otimes \ol g^*(a)
\ar[r,"r(\alpha^{-1})\otimes \id"]
\ar[d,"\mon_{\ol g}"']
&
f_*g^*(c)\otimes \ol g^*(a)
\ar[r,"\rpf_f"]
&
f_*(g^*(c)\otimes f^*\ol g^*(a))
\ar[r,"f_*(\id\otimes \alpha^{-1})"]
&
f_*(g^*(c)\otimes g^*\ol f^*(a))
\ar[d,"f_*\mon_g"]
\\
\ol g^*(\ol f_*(c)\otimes a)
\ar[r,"\ol g^*\rpf_{\ol f}"']
&
\ol g^*\ol f_*(c\otimes \ol f^*(a))
\ar[rr,"r(\alpha^{-1})"']
&&
f_*g^*(c\otimes \ol f^*(a)).
\end{tikzcd}
\]
Since the maps $\rpf_f$, $\rpf_{\ol f}$, $\alpha^{-1}$, and $r(\alpha^{-1})$ are
invertible, we obtain the commutative diagram
\begin{hideeq}\label{eq:abstract-1.2}
\begin{tikzcd}[proof,baseline = (current bounding box.center)]
f_*(g^*(c)\otimes g^*\ol f^*(a))
\ar[r,"f_*(\id\otimes \alpha)"]
\ar[d,"f_*\mon_g"']
&
f_*(g^*(c)\otimes f^*\ol g^*(a))
\ar[r,"\rpf_f^{-1}"]
&
f_*g^*(c) \otimes \ol g^*(a)
\ar[r,"r(\alpha^{-1})^{-1}\otimes\id"]
&
\ol g^*\ol f_*(c) \otimes \ol g^*(a)
\ar[d,"\mon_{\ol g}"]
\\
f_*g^*(c\otimes \ol f^*(a))
\ar[r,"r(\alpha^{-1})^{-1}"']
&
\ol g^*\ol f_*(c\otimes \ol f^*(a))
\ar[rr,"\ol g^* \rpf_{\ol f}^{-1}"']
&&
\ol g^*(\ol f_*(c)\otimes a),
\end{tikzcd}
\end{hideeq}
which, again, we compute the right mate of:
\[
\begin{tikzcd}[proof]
\Nat\bigl(f_*(\blank\otimes f^*\ol g^*(a)), (\blank\otimes \ol g^*(a))f_*\bigr)
\ar[r,leftrightarrow]
\ar[d,"(g^*)^*"']
&
\Nat\bigl((\blank\otimes f^*\ol g^*(a)) f^{(1)}, f^{(1)}(\blank\otimes \ol
g^*(a)) \bigr)
\ar[d,"(\ol g^*)^*"]
\\
\Nat\bigl(f_*(\blank\otimes f^*\ol g^*(a)) g^*, (\blank\otimes \ol g^*(a))
f_*g^*\bigr)
\ar[d,"(r(\alpha^{-1})^{-1})_*"']
&
\Nat\bigl((\blank\otimes f^*\ol g^*(a))f^{(1)}\ol g^*, f^{(1)}(\blank\otimes \ol
g^*(a)) \ol g^*\bigr)
\ar[d,"(r(r(\alpha^{-1})^{-1}))^*"]
\\
\Nat\bigl(f_*(\blank\otimes f^*\ol g^*(a)) g^*, (\blank\otimes \ol g^*(a)) \ol
g^* \ol f_*\bigr)
\ar[r,leftrightarrow]
\ar[d,"\alpha^*"']
&
\Nat\bigl((\blank\otimes f^*\ol g^*(a)) g^*\ol f^{(1)}, f^{(1)}(\blank\otimes
\ol g^*(a)) \ol g^*\bigr)
\ar[d,"\alpha^*"]
\\
\Nat\bigl(f_*(\blank\otimes g^*\ol f^*(a))g^*, (\blank\otimes \ol g^*(a)) \ol g^*\ol
f_*\bigr)
\ar[r,leftrightarrow]
\ar[d,"(\mon_{\ol g})_*"']
&
\Nat\bigl((\blank\otimes g^*\ol f^*(a))g^*\ol f^{(1)}, f^{(1)}(\blank\otimes \ol
g^*(a))\ol g^*\bigr)
\ar[d,"(\mon_{\ol g})_*"]
\\
\Nat\bigl(f_*(\blank\otimes g^*\ol f^*(a))g^*, \ol g^*(\blank\otimes a)\ol
f_*\bigr)
\ar[r,leftrightarrow]
&
\Nat\bigl((\blank\otimes g^*\ol f^*(a))g^*\ol f^{(1)}, f^{(1)}\ol
g^*(\blank\otimes a)\bigr)
\\
\Nat\bigl(f_*g^*(\blank\otimes \ol f^*(a)), \ol g^*(\blank\otimes a)\ol
f_*\bigr)
\ar[r,leftrightarrow]
\ar[u,"(\mon_g)^*"]
&
\Nat\bigl(g^*(\blank\otimes \ol f^*(a)) \ol f^{(1)}, f^{(1)} \ol
g^*(\blank\otimes a)\bigr)
\ar[u,"(\mon_g)^*"']
\\
\Nat\bigl(\ol g^*\ol f_*(\blank\otimes \ol f^*(a)), \ol g^*(\blank\otimes a)\ol
f_*\bigr)
\ar[u,"(r(\alpha^{-1})^{-1})^*"]
&
\Nat\bigl(g^*(\blank\otimes \ol f^*(a))\ol f^{(1)}, g^*\ol
f^{(1)}(\blank\otimes a)\bigr)
\ar[u,"(r(r(\alpha^{-1})^{-1}))_*"']
\\
\Nat\bigl(\ol f_*(\blank\otimes \ol f^*(a)), (\blank\otimes a)\ol f_*\bigr)
\ar[r,leftrightarrow]
\ar[u,"(\ol g^*)_*"]
&
\Nat\bigl((\blank\otimes \ol f^*(a))\ol f^{(1)}, \ol f^{(1)}(\blank\otimes
a)\bigr)
\ar[u,"(g^*)_*"'].
\end{tikzcd}
\]
The commutativity of \eqref{eq:abstract-1.2} means that $\rpf_f^{-1}$ in the upper
left and $\rpf_{\ol f}^{-1}$ in the lower left map to the same transformation in
the fourth line from the bottom. By the mate correspondence, this means that
$\groth_f = r(\rpf_f^{-1})$ in the upper right and $\groth_{\ol f} = r(\rpf_{\ol
f}^{-1})$ in the lower right map to the same transformation in the fourth line
from the bottom. Unraveling the definitons, we deduce that the diagram in the
assertion commutes.
\bigskip

We now prove \hyperref[lem:abstract-b]{(b)}. The following diagram
\begin{hideeq}\label{eq:lpf-1}
\begin{tikzcd}[proof,baseline=(current bounding box.center)]
&[-3em]
g^*\ol f^*(a')\otimes g^*\ol f^*(a)
\ar[r,"\alpha\otimes\alpha"]
&
f^*\ol g^*(a') \otimes f^*\ol g^*(a)
\ar[dr,"\mon_f"]
&[-3em]
\\
g^*(\ol f^*(a')\otimes \ol f^*(a))
\ar[dr,"g^*\mon_{\ol f}"']
\ar[ur,"\mon_g^{-1}","\cong"']
&&&
f^*(\ol g^*(a')\otimes \ol g^*(a))
\\
&
g^*\ol f^*(a'\otimes a)
\ar[r,"\alpha"']
&
f^*\ol g^*(a'\otimes a)
\ar[ur,"f^*\mon_{\ol g}^{-1}"',"\cong"]
\end{tikzcd}
\end{hideeq}
commutes, since $\alpha$ is a monoidal isomorphism. Let us compute the right mate:
\[
\begin{tikzcd}[proof]
\Nat\bigl(g^*(\ol f^*(a')\otimes \blank), (g^*\ol f^*(a')\otimes \blank)
g^*\bigr)
\ar[r,leftrightarrow]
\ar[d,"(\ol f^*)^*"']
&
\Nat\bigl((\ol f^*(a')\otimes \blank)g_*, g_* (g^*\ol f^*(a')\otimes
\blank)\bigr)
\ar[d,"(f^*)^*"]
\\
\Nat\bigl(g^*(\ol f^*(a')\otimes \blank)\ol f^*, (g^*\ol f^*(a')\otimes
\blank)g^*\ol f^*\bigr)
\ar[d,"\alpha_*"']
&
\Nat\bigl((\ol f^*(a')\otimes \blank)g_*f^*, g_*(g^*\ol f^*(a')\otimes \blank)
f^*\bigr)
\ar[d,"(r(\alpha))^*"]
\\
\Nat\bigl(g^*(\ol f^*(a')\otimes \blank)\ol f^*, (g^*\ol f^*(a')\otimes \blank)
f^*\ol g^*\bigr)
\ar[r,leftrightarrow]
\ar[d,"\alpha_*"']
&
\Nat\bigl((\ol f^*(a')\otimes \blank)\ol f^*\ol g_*, g_*(g^*\ol f^*(a')\otimes
\blank)f^*\bigr)
\ar[d,"\alpha_*"]
\\
\Nat\bigl(g^*(\ol f^*(a')\otimes \blank)\ol f^*, (f^*\ol g^*(a')\otimes \blank)
f^*\ol g^*\bigr)
\ar[r,leftrightarrow]
\ar[d,"(\mon_f)_*"']
&
\Nat\bigl((\ol f^*(a')\otimes \blank)\ol f^*\ol g_*, g_*(f^*\ol g^*(a')\otimes
\blank)f^*\bigr)
\ar[d,"(\mon_f)_*"]
\\
\Nat\bigl(g^*(\ol f^*(a')\otimes \blank)\ol f^*, f^*(\ol g^*(a')\otimes
\blank)\ol g^*\bigr)
\ar[r,leftrightarrow]
&
\Nat\bigl((\ol f^*(a')\otimes \blank)\ol f^*\ol g_*, g_*f^*(\ol g^*(a')\otimes
\blank)\bigr)
\\
\Nat\bigl(g^*\ol f^*(a'\otimes \blank), f^*(\ol g^*(a')\otimes \blank)\ol
g^*\bigr) 
\ar[u,"(\mon_{\ol f})^*"]
\ar[r,leftrightarrow]
&
\Nat\bigl(\ol f^*(a'\otimes \blank)\ol g_*, g_*f^*(\ol g^*(a')\otimes
\blank)\bigr)
\ar[u,"(\mon_{\ol f})^*"']
\\
\Nat\bigl(f^*\ol g^*(a'\otimes \blank), f^*(\ol g^*(a')\otimes \blank) \ol
g^*\bigr)
\ar[u,"\alpha^*"]
&
\Nat\bigl(\ol f^*(a'\otimes \blank)\ol g_*, \ol f^*\ol g_*(\ol g^*(a')\otimes
\blank)\bigr)
\ar[u,"(r(\alpha))_*"']
\\
\Nat\bigl(\ol g^*(a'\otimes \blank), (\ol g^*(a')\otimes \blank)\ol g^*\bigr)
\ar[r,leftrightarrow]
\ar[u,"(f^*)_*"]
&
\Nat\bigl((a'\otimes \blank)\ol g_*, \ol g_*(\ol g^*(a')\otimes \blank)\bigr)
\ar[u,"(\ol f^*)_*"'].
\end{tikzcd}
\]

The commutativity of \eqref{eq:lpf-1} means that $\mon_g^{-1}$ in the upper left
and $\mon_{\ol g}^{-1}$ in the lower left map to the same transformation in
the fourth line from the bottom. By the mate correspondence, this means that
$\rpf_g = r(\mon_g^{-1})$ in the upper right and $\rpf_{\ol g} = r(\mon_{\ol
g})$ in the lower right map to the same transformation in the fourth line from
the bottom. The resulting diagram is
\begin{hideeq}\label{eq:lpf-2}
\begin{tikzcd}[proof,baseline=(current bounding box.center)]
\ol f^*(a')\otimes \ol f^*\ol g_*(b)
\ar[d,"\mon_{\ol f}"']
\ar[r,"\id\otimes r(\alpha)"]
&
\ol f^*(a')\otimes g_* f^*(b)
\ar[r,"\rpf_g"]
&
g_*(g^*\ol f^*(a')\otimes f^*(b))
\ar[r,"g_*(\alpha\otimes\id)"]
&
g_*(f^*\ol g^*(a')\otimes f^*(b))
\ar[d,"g_*\mon_f"]
\\
\ol f^*(a'\otimes \ol g_*(b))
\ar[rr,"\ol f^*\rpf_{\ol g}"']
&&
\ol f^*\ol g_*(\ol g^*(a')\otimes b)
\ar[r,"r(\alpha)"']
&
g_*f^*(\ol g^*(a')\otimes b).
\end{tikzcd}
\end{hideeq}
Replacing $\mon_{\ol f}$, $g_*\mon_f$, and $g_*(\alpha\otimes\id)$ by
their inverses, we obtain the diagram
\begin{hideeq}\label{eq:lpf-3}
\begin{tikzcd}[proof,baseline=(current bounding box.center)]
\ol f^*(a'\otimes \ol g_*(b))
\ar[r,"\mon_{\ol f}^{-1}","\cong"']
\ar[d,"\ol f^*\rpf_{\ol g}"']
&
\ol f^*(a')\otimes \ol f^*\ol g_*(b)
\ar[rr,"\id\otimes r(\alpha)"]
&[1em] &[2em]
\ol f^*(a')\otimes g_*f^*(b)
\ar[d,"\rpf_g"]
\\
\ol f^*\ol g_*(\ol g^*(a')\otimes b)
\ar[r,"r(\alpha)"']
&
g_*f^*(\ol g^*(a')\otimes b)
\ar[r,"g_*\mon_f^{-1}"'{yshift=-3pt},"\cong"]
&
g_*(f^*\ol g^*(a')\otimes f^*(b))
\ar[r,"g_*(\alpha^{-1}\otimes\id)"'{yshift=-3pt},"\cong"]
&
g_*(g^*\ol f^*(a')\otimes f^*(b)).
\end{tikzcd}
\end{hideeq}

We now pass to the right mate:
\[
\begin{tikzcd}[proof]
\Nat\bigl(\ol f^*(\blank\otimes \ol g_*(b)), (\blank\otimes \ol f^*\ol g_*(b))
\ol f^*\bigr)
\ar[r,leftrightarrow]
\ar[d,"(r(\alpha))_*"']
&
\Nat\bigl((\blank\otimes\ol g_*(b))\ol f_*, \ol f_*(\blank\otimes \ol f^*\ol
g_*(b))\bigr)
\ar[d,"(r(\alpha))_*"]
\\
\Nat\bigl(\ol f^*(\blank\otimes \ol g_*(b)), (\blank\otimes g_*f^*(b))\ol
f^*\bigr)
\ar[r,leftrightarrow]
\ar[d,"(\rpf_g)_*"']
&
\Nat\bigl((\blank\otimes \ol g_*(b))\ol f_*, \ol f_*(\blank\otimes
g_*f^*(b))\bigr)
\ar[d,"(\rpf_g)_*"]
\\
\Nat\bigl(\ol f^*(\blank\otimes\ol g_*(b)), g_*(\blank\otimes f^*(b))g^*\ol
f^*\bigr)
\ar[r,leftrightarrow]
&
\Nat\bigl((\blank\otimes\ol g_*(b))\ol f_*, \ol f_*g_*(\blank\otimes
f^*(b))g^*\bigr)
\\
\Nat\bigl(\ol f^*\ol g_*(\blank\otimes b) \ol g^*, g_*(\blank\otimes
f^*(b))g^*\ol f^*\bigr)
\ar[r,leftrightarrow]
\ar[u,"(\rpf_{\ol g})^*"]
&
\Nat\bigl(\ol g_*(\blank\otimes b)\ol g^*\ol f_*, \ol f_* g_*(\blank\otimes
f^*(b))g^*\bigr)
\ar[u,"(\rpf_{\ol g})^*"']
\\
\Nat\bigl(g_*f^*(\blank\otimes b)\ol g^*, g_*(\blank\otimes f^*(b)) g^*\ol
f^*\bigr)
\ar[u,"(r(\alpha))^*"]
&
\Nat\bigl(\ol g_*(\blank\otimes b)\ol g^*\ol f_*, \ol g_*f_*(\blank\otimes
f^*(b)) g^*\bigr)
\ar[u,"(r^2(\alpha))_*"']
\\
\Nat\bigl(f^*(\blank\otimes b)\ol g^*, (\blank\otimes f^*(b))g^*\ol f^*\bigr)
\ar[r,leftrightarrow]
\ar[u,"(g_*)_*"]
&
\Nat\bigl((\blank\otimes b)\ol g^*\ol f_*, f_*(\blank\otimes f^*(b)) g^*\bigr)
\ar[u,"(\ol g_*)_*"']
\\
\Nat\bigl(f^*(\blank\otimes b)\ol g^*, (\blank\otimes f^*(b))f^*\ol g^*\bigr)
\ar[u,"(\alpha^{-1})_*"]
&
\Nat\bigl((\blank\otimes b)f_*g^*, f_*(\blank\otimes f^*(b)) g^*\bigr)
\ar[u,"(r(\alpha^{-1}))^*"']
\\
\Nat\bigl(f^*(\blank\otimes b), (\blank\otimes f^*(b))f^*\bigr)
\ar[r,leftrightarrow]
\ar[u,"(\ol g^*)^*"]
&
\Nat\bigl((\blank\otimes b)f_*, f_*(\blank\otimes f^*(b))\bigr)
\ar[u,"(g_*)^*"']
\end{tikzcd}
\]
The commutativity of \eqref{eq:lpf-3} means that $\mon_{\ol f}^{-1}$ in the
upper left and $\mon_f^{-1}$ in the lower left map to the same transformation in
the third line from the top. By the mate correspondence, this means that
$\rpf_{\ol f} = r(\mon_{\ol f}^{-1})$ in the upper right and $\rpf_f =
r(\mon_f^{-1})$ in the lower right map to the same transformation in the third
line from the top. The resulting diagram is
\[
\begin{tikzcd}[proof]
\ol f_*(c)\otimes \ol g_*(b)
\ar[r,"\rpf_{\ol f}"]
\ar[d,"\rpf_{\ol g}"']
&
\ol f_*(c\otimes \ol f^*\ol g_*(b))
\ar[rr,"\ol f_*(\id\otimes r(\alpha))"]
&&
\ol f_*(c\otimes g_*f^*(b))
\ar[d,"\ol f_*(\rpf_g)"]
\\
\ol g_*(\ol g^*\ol f_*(c)\otimes b)
\ar[r,"\ol g_*(r(\alpha^{-1})\otimes \id)"'{yshift=-3pt}]
&
\ol g_*(f_*g^*(c)\otimes b)
\ar[r,"\ol g_*(\rpf_f)"'{yshift=-3pt}]
&
\ol g_*f_*(g^*(c)\otimes f^*(b))
\ar[r,"r^2(\alpha)"']
&
\ol f_*g_*(g^*(c)\otimes f^*(b)).
\end{tikzcd}
\]
Inverting the maps $\rpf_{\ol f}, \ol g_*(\rpf_f), r(\alpha^{-1})$, and
$r^2(\alpha)$, we obtain the following diagram:
\begin{hideeq}\label{eq:lpf-4}
\begin{tikzcd}[proof,baseline=(current bounding box.center)]
\ol f_*(c\otimes \ol f^*\ol g_*(b))
\ar[r,"\rpf_{\ol f}^{-1}","\cong"']
\ar[d,"\ol f_*(\id\otimes r(\alpha))"']
&
\ol f_*(c)\otimes \ol g_*(b)
\ar[rr,"\rpf_{\ol g}"]
&&
\ol g_*(\ol g^*\ol f_*(c)\otimes b)
\\
\ol f_*(c\otimes g_*f^*(b))
\ar[r,"\ol f_*(\rpf_g)"'{yshift=-3pt}]
&
\ol f_*g_*(g^*(c)\otimes f^*(b))
\ar[r,"r^2(\alpha^{-1})"'{yshift=-3pt},"\cong"]
&
\ol g_*f_*(g^*(c)\otimes f^*(b))
\ar[r,"\ol g_*(\rpf_f^{-1})"'{yshift=-3pt},"\cong"]
&
\ol g_*(f_*g^*(c)\otimes b)
\ar[u,"\ol g_*(r(\alpha^{-1})^{-1}\otimes \id)"',"\cong"].
\end{tikzcd}
\end{hideeq}

We now compute the right mate of this diagram.

\[
\begin{tikzcd}[proof]
\Nat\bigl(\ol f_*(-\otimes \ol f^*\ol g_*(b)), (-\otimes \ol g_*(b))\ol f_*\bigr)
\ar[r,leftrightarrow]
\ar[d,"(\rpf_{\ol g})_*"']
&
\Nat\bigl((-\otimes \ol f^*\ol g_*(b))\ol f^{(1)}, \ol f^{(1)} (-\otimes \ol
g_*(b))\bigr)
\ar[d,"(\rpf_{\ol g})_*"]
\\
\Nat\bigl(\ol f_*(-\otimes \ol f^*\ol g_*(b)), \ol g_*(-\otimes b)\ol g^* \ol
f_*\bigr)
\ar[r,leftrightarrow]
&
\Nat\bigl((-\otimes \ol f^*\ol g_*(b))\ol f^{(1)}, \ol f^{(1)} \ol g_*(-\otimes b)
\ol g^*\bigr)
\\
\Nat\bigl(\ol f_*(-\otimes g_*f^*(b)), \ol g_*(-\otimes b) \ol g^*\ol f_*\bigr)
\ar[r,leftrightarrow]
\ar[u,"(r(\alpha))^*"]
&
\Nat\bigl((-\otimes g_*f^*(b))\ol f^{(1)}, \ol f^{(1)} \ol g_*(-\otimes b)\ol
g^*\bigr)
\ar[u,"(r(\alpha))^*"']
\\
\Nat\bigl(\ol f_*g_*(-\otimes f^*(b))g^*, \ol g_*(-\otimes b)\ol g^*\ol f_*\bigr)
\ar[r,leftrightarrow]
\ar[u,"(\rpf_g)^*"]
&
\Nat\bigl(g_*(-\otimes f^*(b))\ol f^{(1)}, \ol f^{(1)} \ol g_*(-\otimes b) \ol
g^*\bigr)
\ar[u,"(\rpf_g)^*"']
\\
\Nat\bigl(\ol g_*f_*(-\otimes f^*(b))g^*, \ol g_*(-\otimes b)\ol g^*\ol f_*\bigr)
\ar[u,"(r^2(\alpha^{-1}))^*"]
&
\Nat\bigl(g_*(-\otimes f^*(b)) g^*\ol f^{(1)}, g_*f^{(1)}(-\otimes b)\ol g^*\bigr)
\ar[u,"(r^3(\alpha^{-1}))_*"']
\\
\Nat\bigl(f_*(-\otimes f^*(b))g^*, (-\otimes b)\ol g^*\ol f_*\bigr)
\ar[r,leftrightarrow]
\ar[u,"(\ol g_*)_*"]
&
\Nat\bigl((-\otimes f^*(b))g^*\ol f^{(1)}, f^{(1)}(-\otimes b)\ol g^*\bigr)
\ar[u,"(g_*)_*"']
\\
\Nat\bigl(f_*(-\otimes f^*(b))g^*, (-\otimes b)f_*g^*\bigr)
\ar[u,"(r(\alpha^{-1})^{-1})_*"]
&
\Nat\bigl((-\otimes f^*(b))f^{(1)}\ol g^*, f^{(1)}(-\otimes b)\ol g^*\bigr)
\ar[u,"(r(r(\alpha^{-1})^{-1}))^*"']
\\
\Nat\bigl(f_*(-\otimes f^*(b)), (-\otimes b)f_*\bigr)
\ar[r,leftrightarrow]
\ar[u,"(g^*)^*"]
&
\Nat\bigl((-\otimes f^*(b)) f^{(1)}, f^{(1)}(-\otimes b)\bigr)
\ar[u,"(\ol g^*)^*"'].
\end{tikzcd}
\]

The commutativity of \eqref{eq:lpf-4} means that $\rpf_{\ol f}^{-1}$ in the
upper left and $\rpf_f^{-1}$ in the lower left map to the same transformation in
the second line from the top. By the mate correspondence, this means that
$\groth_{\ol f} = r(\rpf_f^{-1})$ in the upper right and $\groth_f =
r(\rpf_f^{-1})$ in the lower right map to the same transformation in the second
line from the top. The resulting diagram is
\begin{hideeq}\label{eq:lpf-5}
\begin{tikzcd}[proof,baseline=(current bounding box.center)]
\ol f^{(1)}(a)\otimes g_*f^*(b)
\ar[d,"\rpf_g"']
&
\ol f^{(1)}(a)\otimes \ol f^*\ol g_*(b)
\ar[r,"\groth_{\ol f}"]
\ar[l,"\id\otimes r(\alpha)"',"\cong"]
&
\ol f^{(1)}(a\otimes \ol g_*(b))
\ar[r,"\ol f^{(1)}\rpf_{\ol g}"]
&
\ol f^{(1)}\ol g_*(\ol g^*(a)\otimes b)
\\
g_*(g^*\ol f^{(1)}(a)\otimes f^*(b))
\ar[rr,"g_*(r(r(\alpha^{-1})^{-1})\otimes \id)"']
&&
g_*(f^{(1)}\ol g^*(a)\otimes f^*(b))
\ar[r,"g_*\groth_f"']
&
g_*f^{(1)}(\ol g^*(a)\otimes b)
\ar[u,"r^3(\alpha^{-1})"'].
\end{tikzcd}
\end{hideeq}

After taking the inverse of $\id\otimes r(\alpha)$, we compute the left mate:
\[
\begin{tikzcd}[proof]
\Nat\bigl(\ol f_*(\ol f^{(1)}(a)\otimes \blank), (a\otimes \blank) \ol
f_{(1)}\bigr)
\ar[r,leftrightarrow]
\ar[d,"(g_*)^*"']
&
\Nat\bigl((\ol f^{(1)}(a)\otimes \blank)\ol f^*, \ol f^{(1)}(a\otimes
\blank)\bigr) 
\ar[d,"(\ol g_*)^*"]
\\
\Nat\bigl(\ol f_*(\ol f^{(1)}(a)\otimes \blank) g_*, (a\otimes \blank) \ol
f_{(1)}g_*\bigr)
\ar[d,"(l(r(\alpha)^{-1}))_*"']
&
\Nat\bigl((\ol f^{(1)}(a)\otimes \blank)\ol f^*\ol g_*, \ol f^{(1)}(a\otimes
\blank) \ol g_*\bigr)
\ar[d,"(r(\alpha)^{-1})^*"]
\\
\Nat\bigl(\ol f_*(\ol f^{(1)}(a)\otimes \blank)g_*, (a\otimes \blank)\ol g_*
f_{(1)}\bigr)
\ar[r,leftrightarrow]
\ar[d,"(\rpf_{\ol g})_*"']
&
\Nat\bigl((\ol f^{(1)}(a)\otimes \blank)g_*f^*, \ol f^{(1)}(a\otimes \blank)\ol
g_*\bigr)\ar[d,"(\rpf_{\ol g})_*"]
\\
\Nat\bigl(\ol f_*(\ol f^{(1)}(a)\otimes \blank)g_*, \ol g_*(\ol g^*(a)\otimes
\blank) f_{(1)}\bigr)
\ar[r,leftrightarrow]
&
\Nat\bigl((\ol f^{(1)}(a)\otimes \blank)g_*f^*, \ol f^{(1)}\ol g_*(\ol
g^*(a)\otimes \blank)\bigr)
\\
\Nat\bigl(\ol f_* g_*(g^*\ol f^{(1)}(a)\otimes \blank), \ol g_*(\ol
g^*(a)\otimes \blank) f_{(1)}\bigr)
\ar[r,leftrightarrow]
\ar[u,"(\rpf_g)^*"]
&
\Nat\bigl(g_*(g^*\ol f^{(1)}(a)\otimes \blank)f^*, \ol f^{(1)}\ol g_*(\ol
g^*(a)\otimes \blank)\bigr)
\ar[u,"(\rpf_g)^*"']
\\
\Nat\bigl(\ol g_*f_*(g^*\ol f^{(1)}(a)\otimes \blank), \ol g_*(\ol
g^*(a)\otimes \blank)f_{(1)}\bigr)
\ar[u,"(r^2(\alpha^{-1}))^*"]
&
\Nat\bigl(g_*(g^*\ol f^{(1)}(a)\otimes \blank)f^*, g_*f^{(1)}(\ol g^*(a)\otimes
\blank)\bigr)
\ar[u,"(r^3(\alpha^{-1}))_*"']
\\
\Nat\bigl(f_*(g^*\ol f^{(1)}(a)\otimes \blank), (\ol g^*(a)\otimes
\blank)f_{(1)}\bigr)
\ar[u,"(\ol g_*)_*"]
\ar[r,leftrightarrow]
&
\Nat\bigl((g^*\ol f^{(1)}(a)\otimes \blank)f^*, f^{(1)}(\ol g^*(a)\otimes
\blank)\bigr)
\ar[u,"(g_*)_*"'].
\\
\Nat\bigl(f_*(f^{(1)}\ol g^*(a)\otimes \blank), (\ol g^*(a)\otimes
\blank)f_{(1)}\bigr)
\ar[r,leftrightarrow]
\ar[u,"(r(r(\alpha^{-1})^{-1}))^*"]
&
\Nat\bigl((f^{(1)}\ol g^*(a)\otimes \blank)f^*, f^{(1)}(\ol g^*(a)\otimes
\blank)\bigr)
\ar[u,"(r(r(\alpha^{-1})^{-1}))^*"']
\end{tikzcd}
\]

The commutativity of \eqref{eq:lpf-5} means that $\groth_{\ol f}$ in the upper
right and $\groth_f$ in the lower right map to the same transformation in the
fourth line from the top. By the mate correspondence, this means that
$\lpf_{\ol f} = l(\groth_{\ol f})$ in the upper left and $\lpf_f = l(\groth_f)$
in the lower left map to the same transformation in the fourth line from the
top. Unraveling the definitions, we deduce that the diagram in
\hyperref[lem:abstract-b]{(b)} commutes.
\end{hideproof} 

\subsubsection{} 
\label{sss:setup2}
Consider a diagram
\[
\begin{tikzcd}[column sep=tiny]
\cat A 
\ar[dr,"f^*"] \ar[dd,bend right,"h^*"'] \ar[rrr,"\alpha^*"] 
& &[5em] &
\ol{\cat A}
\ar[dl,"\ol{f}^*"'] \ar[dd,bend left,"\ol{h}^*"]
\\
& 
\cat B
\ar[dl,shift right,"g^*"'] \ar[r,"\beta^*" description]
\ar[urr,Rightarrow,shorten=3em,xshift=-1em,"\sigma" description]
& 
\ol{\cat B}
\ar[dr,shift left,"\ol{g}^*"]
\\
\cat C 
\ar[ur,dashed,shift right,"g_!"'] \ar[rrr,"\gamma^*"']
\ar[urr,Rightarrow,shorten=3em,xshift=1.4em,"\tau" description]
& & & 
\ol{\cat C},
\ar[ul,dashed,shift left,"\ol{g}_!"]
\end{tikzcd}
\]
where the solid diagram commutes and consists of (strongly) monoidal functors
between monoidal categories, and where 
$\sigma\colon \beta^*f^*\Raiso \ol{f}^*\alpha^*$ and $\tau\colon \gamma^*g^*
\Raiso \ol{g}^*\beta^*$ are monoidal natural isomorphisms. Identify $h^* =
g^*f^*$ and $\ol{h}^* = \ol{g}^* \ol{f}^*$ and put
$\rho\coloneqq \ol{g}^*\sigma \circ \tau f^* \colon \gamma^*h^* \Raiso \ol{h}^*
\alpha^*$. Let further
\[
\phi\colon \ol{g}_!\gamma^* \Raiso \beta^* g_!
\]
be a natural isomorphism. We make the following assumptions:
\begin{enumerate}[resume*=abstract] 
\item\label{Assumption3} The functors $f^*, \ol{f}^*, h^*, \ol{h}^*$ admit left
adjoints:
\[
f_{(1)} \dashv f^*,\qquad \ol{f}_{(1)}\dashv \ol{f}^*, \qquad
h_{(1)} \dashv h^*,\qquad \ol{h}_{(1)}\dashv \ol{h}^*.
\]

\item\label{Assumption4} There exist natural isomorphisms
\[
\renewcommand{\arraystretch}{1.5}
\begin{matrix}
\pf_g \colon & g_!\bigl(c \otimes g^*(b)\bigr) & \raiso & g_!(c)\otimes b
\\
\pf_{\ol{g}} \colon & \ol{g}_!\bigl(\bar c\otimes \ol{g}^*(\bar b)\bigr) &
\raiso & \ol{g}_!(\bar{c})\otimes \bar{b}
\end{matrix}
\]
such that the following diagram commutes:
\[
\begin{tikzcd}
\ol{g}_!\bigl(\gamma^*(c)\otimes \ol{g}^*\beta^*(b)\bigr)
\ar[d,"\ol{g}_!(\id\otimes\tau^{-1})"']
\ar[r,"\pf_{\ol{g}}"] 
&
\ol{g}_!\gamma^*(c) \otimes \beta^*(b)
\ar[d,"\phi\otimes\id"]
\\
\ol{g}_!\bigl(\gamma^*(c)\otimes \gamma^*g^*(b)\bigr)
\ar[d,"\ol{g}_!\mon_\gamma"']
&
\beta^*g_!(c)\otimes \beta^*(b)
\ar[dd,"\mon_\beta"]
\\
\ol{g}_!\gamma^*\bigl(c\otimes g^*(b)\bigr)
\ar[d,"\phi"']
&
\\
\beta^*g_!\bigl(c\otimes g^*(b)\bigr)
\ar[r,"\pf_g"']
&
\beta^*\bigl(g_!(c)\otimes b\bigr).
\end{tikzcd}
\]
\end{enumerate} 

By~\ref{Assumption3}, the left mate of $\mon_f\colon (\blank\otimes f^*(a))\circ
f^* \Raiso f^*\circ (\blank\otimes a)$ yields a natural map
\[
\lpf_f\colon f_{(1)}\bigl(b\otimes f^*(a)\bigr) \too f_{(1)}(b)\otimes a,
\]
and similarly for $\lpf_{\ol{f}}$. Passing to the left mate of the composite
\[
f_{(1)}g_!\circ (c\otimes \blank) \circ g^*f^* \xrightarrow{f_{(1)}\pf_g}
f_{(1)}\circ \bigl(g_!(c)\otimes \blank\bigr)\circ f^* \xrightarrow{\lpf_f}
f_{(1)}g_!(c)\otimes \blank,
\]
and using $h^* = g^*f^*$, we obtain a natural map
\[
\ltm_{f,g}\colon f_{(1)}g_!(c'\otimes c) \too f_{(1)}g_!(c') \otimes
h_{(1)}(c).
\]
Similarly, the map $\ltm_{\ol{f},\ol{g}}\colon \ol{f}_{(1)}\ol{g}_! (\bar{c}'
\otimes \bar{c}) \to \ol{f}_{(1)}\ol{g}_!(\bar{c}') \otimes
\ol{h}_{(1)}(\bar{c})$ is defined.

\begin{lem}\label{lem:abstract-2} 
For all $c',c\in \cat C$ the following diagram commutes:
\[\begin{tikzcd}[baseline=(current bounding box.center)] 
\ol{f}_{(1)}\ol{g}_!\bigl(\gamma^*(c')\otimes \gamma^*(c)\bigr)
\ar[r,"\ltm_{\ol{f},\ol{g}}"]
\ar[d,"\ol{f}_{(1)}\ol{g}_!\mon_\gamma"']
&
\ol{f}_{(1)}\ol{g}_!\gamma^*(c') \otimes \ol{h}_{(1)}\gamma^*(c)
\ar[d,"\ol{f}_{(1)}\phi\otimes\id"]
\\
\ol{f}_{(1)}\ol{g}_!\gamma^*(c'\otimes c)
\ar[d,"\ol{f}_{(1)}\phi"']
&
\ol{f}_{(1)}\beta^*g_!(c') \otimes \ol{h}_{(1)}\gamma^*(c)
\ar[d,"l(\sigma)\otimes l(\rho)"]
\\
\ol{f}_{(1)}\beta^*g_!(c'\otimes c)
\ar[d,"l(\sigma)"']
&
\alpha^*f_{(1)}g_!(c')\otimes \alpha^*h_{(1)}(c)
\ar[d,"\mon_\alpha"]
\\
\alpha^*f_{(1)}g_!(c'\otimes c)
\ar[r,"\alpha^*\ltm_{f,g}"']
&
\alpha^*\bigl(f_{(1)}g_!(c')\otimes h_{(1)}(c)\bigr).
\end{tikzcd}\] 
\end{lem} 
\begin{hideproof} 
Since $\sigma$ is a monoidal isomorphism, we have a commutative diagram
\[
\begin{tikzcd}[proof]
\beta^*f^*(a')\otimes \beta^*f^*(a)
\ar[r,"\mon_\beta"] \ar[d,"\sigma\otimes\sigma"']
&
\beta^*\bigl(f^*(a')\otimes f^*(a)\bigr)
\ar[r,"\beta^*\mon_f"]
&
\beta^*f^*(a'\otimes a)
\ar[d,"\sigma"]
\\
\ol f^*\alpha^*(a')\otimes \ol f^*\alpha^*(a)
\ar[r,"\mon_{\ol f}"']
&
\ol f^*\bigl(\alpha^*(a')\otimes \alpha^*(a)\bigr)
\ar[r,"\ol f^*\mon_\alpha"']
&
\ol f^*\alpha^*(a'\otimes a).
\end{tikzcd}
\]
After passing to the left mates at $\mon_{\ol{f}}\colon (\blank\otimes
\ol{f}^*\alpha^*(a)\bigr) \ol{f}^*\alpha^* \To \ol{f}^*\bigl(\blank\otimes
\alpha^*(a)\bigr) \alpha^*$, we obtain a commutative diagram
\[
\begin{tikzcd}[proof]
\ol f_{(1)}\bigl(\beta^*g_!(c)\otimes \ol f^*\alpha^*(a)\bigr)
\ar[r,"\lpf_{\ol f}"]
&
\ol f_{(1)}\beta^*g_!(c)\otimes \alpha^*(a)
\ar[d,"l(\sigma)\otimes\id"]
\\
\ol f_{(1)}\bigl(\beta^*g_!(c)\otimes \beta^*f^*(a)\bigr)
\ar[u,"\ol f_{(1)}(\id\otimes\sigma)"]
\ar[d,"\mon_\beta"']
&
\alpha^*f_{(1)}g_!(c)\otimes \alpha^*(a)
\ar[dd,"\mon_\alpha"]
\\
\ol f_{(1)}\beta^*\bigl(g_!(c)\otimes f^*(a)\bigr)
\ar[d,"l(\sigma)"']
\\
\alpha^*f_{(1)}\bigl(g_!(c)\otimes f^*(a)\bigr)
\ar[r,"\alpha^*\lpf_f"']
&
\alpha^*\bigl(f_{(1)}g_!(c)\otimes a\bigr).
\end{tikzcd}
\]
Using \ref{Assumption4}, one now easily checks that the diagram
\[
\begin{tikzcd}[proof]
\ol f_{(1)}\ol g_!\bigl(\gamma^*(c)\otimes \ol g^*\ol f^*\alpha^*(a)\bigr)
\ar[r,"\ol{f}_{(1)}\pf_{\ol g}"]
&
\ol f_{(1)}\bigl(\ol g_!\gamma^*(c)\otimes \ol f^*\alpha^*(a)\bigr)
\ar[r,"\lpf_{\ol f}"]
&
\ol f_{(1)}\ol g_!\gamma^*(c)\otimes\alpha^*(a)
\ar[d,"l(\sigma)\ol f_{(1)}\phi\otimes\id"]
\\
\ol f_{(1)}\ol g_!\bigl(\gamma^*(c)\otimes \gamma^*g^*f^*(a)\bigr)
\ar[u,"\ol f_{(1)}\ol g_!(\id\otimes\rho)"]
\ar[d,"\ol{f}_{(1)}\ol{g}_!\mon_\gamma"']
& &
\alpha^*f_{(1)}g_!(c)\otimes\alpha^*(a)
\ar[dd,"\mon_{\alpha}"]
\\
\ol f_{(1)}\ol g_!\gamma^*(c\otimes g^*f^*(a))
\ar[d,"l(\sigma)\ol f_{(1)}\phi"']
\\
\alpha^*f_{(1)}g_!(c\otimes g^*f^*(a))
\ar[r,"\alpha^*f_{(1)}\pf_g"']
&
\alpha^*f_{(1)}\bigl(g_!(c)\otimes f^*(a)\bigr)
\ar[r,"\alpha^*\lpf_f"']
&
\alpha^*\bigl(f_{(1)}g_!(c)\otimes a\bigr).
\end{tikzcd}
\]
commutes. Now, passing to the left mates at 
\[
\lpf_{\ol{f}}\circ \ol{f}_{(1)}\pf_{\ol{g}}\colon
\ol{f}_{(1)}\ol{g}_!\bigl(\gamma^*(c)\otimes\blank\bigr)
\ol{g}^*\ol{f}^*\alpha^* \Too \bigl(\ol{f}_{(1)}\ol{g}_!\gamma^*(c)\otimes
\blank\bigr) \alpha^*
\]
yields the commutativity of the diagram in the lemma.
\end{hideproof} 

\subsubsection{} 
\label{sss:setup3}
We stay in the context of \S\ref{sss:setup2} but restrict to the subdiagram
\[
\begin{tikzcd}
\cat A \ar[r,"f^*"] \ar[dr,bend right,"h^*"'] 
& 
\cat B \ar[r,"\beta^*"] \ar[d,shift right,"g^*"'] 
& 
\ol{\cat B} \ar[d,shift left,"\ol{g}^*"]
\\
&
\cat C \ar[u,shift right,dashed,"g_!"'] \ar[r,"\gamma^*"']
& 
\ol{\cat C} \ar[u,shift left,dashed,"\ol{g}_!"].
\end{tikzcd}
\]
Besides \ref{Assumption3} and \ref{Assumption4} we additionally assume:
\begin{enumerate}[resume*=abstract]
\item\label{Assumption5} The functors $\beta^*$ and $\gamma^*$ admit left
adjoints:
\[
\beta_{(1)}\dashv \beta^*,\qquad\qquad \gamma_{(1)}\dashv \gamma^*.
\]
\end{enumerate}
We denote $\varepsilon_\beta\colon \beta_{(1)}\beta^*\To \id_{\cat B}$ 
the counit. 
Observe that $f_{(1)}\beta_{(1)}$ is a left adjoint of $\beta^*f^*$.
Passing to the left mate of the composite
$f_{(1)}\beta_{(1)}\ol{g}_!\circ (\bar{c}'\otimes\blank) \circ
\ol{g}^*\beta^*f^*
\xRightarrow{\pf_{\ol g}} f_{(1)}\beta_{(1)}\circ(\ol{g}_!(\bar{c}') \otimes
\blank)\circ \beta^*f^* \xRightarrow{\lpf_\beta} f_{(1)}\circ
(\beta_{(1)}\ol{g}_!(\bar{c}')\otimes \blank)\circ f^*
\xRightarrow{\lpf_f} f_{(1)}\beta_{(1)}\ol{g}_!(\bar{c}')\otimes \blank$, and
using $\ol{g}^*\beta^*f^* = \gamma^*h^*$ we obtain a natural map
\[
\ltm_{f\beta,\ol{g}}\colon f_{(1)}\beta_{(1)}\ol{g}_!(\bar{c}'\otimes
\bar{c}) \too f_{(1)}\beta_{(1)}\ol{g}_!(\bar{c}') \otimes
h_{(1)}\gamma_{(1)}(\bar{c}).
\]
Let $\lpf_\gamma\colon \gamma_{(1)}(\gamma^*(c)\otimes \bar c)
\to c\otimes \gamma_{(1)}(\bar c)$ be the left mate of $\mon_{\gamma}\colon
(\gamma^*(c)\otimes \blank)\circ \gamma^* \Raiso \gamma^*\circ (c\otimes
\blank)$.

\begin{lem}\label{lem:abstract-3} 
For all $c\in \cat C$ and $\bar{c}\in \ol{\cat C}$ the following diagram
commutes:
\[\begin{tikzcd} 
f_{(1)}\beta_{(1)}\ol{g}_!\bigl(\gamma^*(c)\otimes \bar{c}\bigr)
\ar[r,"\ltm_{f\beta,\ol{g}}"]
\ar[d,"f_{(1)}l(\phi)"']
&
f_{(1)}\beta_{(1)}\ol{g}_!\gamma^*(c)\otimes h_{(1)}\gamma_{(1)}(\bar{c})
\ar[d,"f_{(1)}\beta_{(1)}\phi\otimes\id"]
\\
f_{(1)}g_!\gamma_{(1)}\bigl(\gamma^*(c)\otimes\bar{c}\bigr)
\ar[d,"f_{(1)}g_!\lpf_\gamma"']
&
f_{(1)}\beta_{(1)}\beta^*g_!(c)\otimes h_{(1)}\gamma_{(1)}(\bar{c})
\ar[d,"f_{(1)}\varepsilon_\beta\otimes\id"]
\\
f_{(1)}g_!\bigl(c\otimes\gamma_{(1)}(\bar{c})\bigr)
\ar[r,"\ltm_{f,g}"']
&
f_{(1)}g_!(c)\otimes h_{(1)}\gamma_{(1)}(\bar{c}).
\end{tikzcd}\] 
\end{lem} 
\begin{hideproof} 
It suffices to show that the following diagram commutes, because then passing to
the left mates at $\lpf_f \circ f_{(1)}\lpf_\beta\circ
f_{(1)}\beta_{(1)}\pf_{\ol{g}}\colon
f_{(1)}\beta_{(1)}\ol{g}_!(\gamma*(c)\otimes\blank)
\ol{g}^*\beta^*f^* \To (f_{(1)}\beta_{(1)}\ol{g}_!\gamma^*(c)\otimes \blank)$
proves the assertion.
\[
\begin{tikzcd}[proof,sep=large]
f_{(1)}\beta_{(1)}\ol{g}_!\bigl(\gamma^*(c)\otimes \ol{g}^*\beta^*f^*(a)\bigr)
\ar[r,"f_{(1)}l(\phi)(\id\otimes\tau^{-1})"{yshift=5pt}]
\ar[d,"f_{(1)}\beta_{(1)}\pf_{\ol{g}}"']
&
f_{(1)}g_!\gamma_{(1)}\bigl(\gamma^*(c)\otimes\gamma^*g^*f^*(a)\bigr)
\ar[r,"f_{(1)}g_!\gamma_{(1)}\mon_\gamma"{yshift=5pt}]
&
f_{(1)}g_!\gamma_{(1)}\gamma^*\bigl(c\otimes g^*f^*(a)\bigr)
\ar[d,"f_{(1)}g_!\varepsilon_\gamma"]
\\
f_{(1)}\beta_{(1)}\bigl(\ol{g}_!\gamma^*(c)\otimes \beta^*f^*(a)\bigr)
\ar[d,"f_{(1)}\lpf_\beta"']
& &
f_{(1)}g_!\bigl(c\otimes g^*f^*(a)\bigr)
\ar[d,"f_{(1)}\pf_g"]
\\
f_{(1)}\bigl(\beta_{(1)}\ol{g}_!\gamma^*(c)\otimes f^*(a)\bigr)
\ar[d,"\lpf_f"']
\ar[r,"f_{(1)}(\beta_{(1)}\phi\otimes\id)"]
&
f_{(1)}\bigl(\beta_{(1)}\beta^*g_!(c)\otimes f^*(a)\bigr)
\ar[d,"\lpf_f"']
\ar[r,"f_{(1)}(\varepsilon_\beta\otimes\id)"]
&
f_{(1)}\bigl(g_!(c)\otimes f^*(a)\bigr)
\ar[d,"\lpf_f"]
\\
f_{(1)}\beta_{(1)}\ol{g}_!\gamma^*(c)\otimes a
\ar[r,"f_{(1)}\beta_{(1)}\phi\otimes\id"']
&
f_{(1)}\beta_{(1)}\beta^*g_!(c)\otimes a
\ar[r,"f_{(1)}\varepsilon_\beta\otimes\id"']
&
f_{(1)}g_!(c)\otimes a
\end{tikzcd}
\]
The small squares commute by the naturality of $\lpf_f$. For the big square, we
consider
\[
\begin{tikzcd}[proof]
\beta_{(1)}\ol{g}_!\bigl(\gamma^*(c)\otimes \ol{g}^*\beta^*(b)\bigr)
\ar[dd,"\beta_{(1)}\ol{g}_!(\id\otimes\tau^{-1})"']
\ar[rr,"\beta_{(1)}\pf_{\ol{g}}"]
\ar[rrdd,phantom,"(A)"{font=\tiny}]
& &
\beta_{(1)}\bigl(\ol{g}_!\gamma^*(c)\otimes\beta^*(b)\bigr)
\ar[d,"\beta_{(1)}(\phi\otimes\id)"']
\ar[r,"\lpf_\beta"]
&
\beta_{(1)}\ol{g}_!\gamma^*(c)\otimes b
\ar[d,"\beta_{(1)}\phi\otimes\id"]
\\
& &
\beta_{(1)}\bigl(\beta^*g_!(c)\otimes \beta^*(b)\bigr)
\ar[r,"\lpf_\beta"]
\ar[d,"\beta_{(1)}\mon_\beta"']
\ar[ddr,phantom,"(B)"{yshift=1.5em,font=\tiny}]
&
\beta_{(1)}\beta^*g_!(c)\otimes b
\ar[dd,"\varepsilon_\beta\otimes\id"]
\\
\beta_{(1)}\ol{g}_!\bigl(\gamma^*(c)\otimes \gamma^*g^*(b)\bigr)
\ar[d,"l(\phi)"']
\ar[r,"\beta_{(1)}\ol{g}_!\mon_\gamma"]
&
\beta_{(1)}\ol{g}_!\gamma^*\bigl(c\otimes g^*(b)\bigr)
\ar[d,"l(\phi)"]
&
\beta_{(1)}\beta^*\bigl(g_!(c)\otimes b\bigr)
\ar[dr,"\varepsilon_\beta"']
\\
g_!\gamma_{(1)}\bigl(\gamma^*(c)\otimes \gamma^*g^*(b)\bigr)
\ar[r,"g_!\gamma_{(1)}\mon_\gamma"'{yshift=-5pt}]
&
g_!\gamma_{(1)}\gamma^*\bigl(c\otimes g^*(b)\bigr)
\ar[r,"g_!\varepsilon_\gamma"']
&
g_!\bigl(c\otimes g^*(b)\bigr)
\ar[r,"\pf_g"']
&
g_!(c)\otimes b,
\end{tikzcd}
\]
where we have put $b \coloneqq f^*(a)$.
The commutativity of the small squares on the left and top right are clear from
the naturality of $l(\phi)$ and $\lpf_\beta$, respectively. For the diagram
$(A)$ we consider
\[
\begin{tikzcd}[proof]
\beta_{(1)}\ol{g}_!\bigl(\gamma^*(c)\otimes \ol{g}^*\beta^*(b)\bigr)
\ar[d,"\beta_{(1)}\ol{g}_!(\id\otimes\tau^{-1})"']
\ar[r,"\beta_{(1)}\pf_{\ol{g}}"]
&
\beta_{(1)}\bigl(\ol{g}_!\gamma^*(c)\otimes\beta^*(b)\bigr)
\ar[r,"\beta_{(1)}(\phi\otimes\id)"]
&
\beta_{(1)}\bigl(\beta^*g_!(c)\otimes\beta^*(b)\bigr)
\ar[dd,"\beta_{(1)}\mon_\beta"]
\\
\beta_{(1)}\ol{g}_!\bigl(\gamma^*(c)\otimes\gamma^*g^*(b)\bigr)
\ar[d,"\beta_{(1)}\ol{g}_!\mon_\gamma"']
\\
\beta_{(1)}\ol{g}_!\gamma^*\bigl(c\otimes g^*(b)\bigr)
\ar[d,"l(\phi)"']
\ar[r,"\beta_{(1)}\phi"]
\ar[dr,phantom,"(B')"{font=\tiny}]
&
\beta_{(1)}\beta^*g_!\bigl(c\otimes g^*(b)\bigr)
\ar[d,"\varepsilon_\beta"']
\ar[r,"\beta_{(1)}\beta^*\pf_g"]
&
\beta_{(1)}\beta^*\bigl(g_!(c)\otimes b\bigr)
\ar[d,"\varepsilon_\beta"]
\\
g_!\gamma_{(1)}\gamma^*\bigl(c\otimes g^*(b)\bigr)
\ar[r,"g_!\varepsilon_\gamma"']
&
g_!\bigl(c\otimes g^*(b)\bigr)
\ar[r,"\pf_g"']
&
g_!(c)\otimes b.
\end{tikzcd}
\]
Here, the lower right square commutes by naturality of $\varepsilon_\beta$, and
the top square commutes by \ref{Assumption4}. The diagrams $(B)$ and $(B')$
commute by the following general fact: Consider the left diagram of functors
\[
\begin{tikzcd}
\cat X
\ar[r,shift left,"L"] \ar[d,"F"']
\ar[dr,Rightarrow,shorten=1em,"\chi" description]
&
\cat Y
\ar[l,shift left,"R"] \ar[d,"G"]
&[3em]
L'FR \ar[r,"L'\chi"] \ar[d,"l(\chi)R"'] & L'R'G \ar[d,"\varepsilon'G"]
\\
\cat X'
\ar[r,shift left,"L'"]
&
\cat Y'
\ar[l,shift left,"R'"],
&
GLR \ar[r,"G\varepsilon"'] & G
\end{tikzcd}
\]
where $\chi\colon FR\To R'G$ is a natural transformation with left mate
$l(\chi)\colon L'F\To GL$. Denote by $\varepsilon\colon LR\To \id_{\cat Y}$ and $\varepsilon'\colon L'R'\To \id_{\cat Y'}$ the counits of the respective adjunctions. Then the right diagram above commutes, since the diagram
\[
\begin{tikzcd}[proof,sep=large]
L'FR
\ar[d,"L'F\eta R" description] \ar[dr,equals,bend left] 
\ar[ddd, start anchor=south west, end anchor=north west,bend right,"l(\chi)R"']
\\
L'FRLR 
\ar[r,"L'FR\varepsilon"] \ar[d,"L'\chi LR" description]
& 
L'FR
\ar[d,"L'\chi"]
\\
L'R'GLR
\ar[d,"\varepsilon'GLR" description] \ar[r,"L'R'G\varepsilon"] 
& 
L'R'G
\ar[d,"\varepsilon'G"]
\\
GLR 
\ar[r,"G\varepsilon"']
& 
G
\end{tikzcd}
\]
is commutative, where $\eta\colon \id_{\cat X}\To RL$ denotes the unit of the
adjunction $L\dashv R$.

\end{hideproof} 

\subsubsection{} 
\label{sss:setup4}
Consider a diagram
\[
\begin{tikzcd}
& & \cat A 
\ar[dl,shift right,"f^*"'] \ar[dd,shift left,bend left,"h^*"]
\\
\cat D
\ar[urr,bend left,"\alpha^*"] \ar[r,"\beta^*"] \ar[drr,bend right,"\gamma^*"']
&
\cat B
\ar[ur,shift right,dashed,"f_!"'] \ar[dr,shift right,"g^*"']
\\
& &
\cat C
\ar[ul,shift right,dashed, "g_!"'] \ar[uu,shift left, dashed, bend right,
"h_!"]
\end{tikzcd}
\]
where the solid diagram commutes and consists of (strongly) symmetric monoidal
functors between symmetric monoidal categories, and where the dashed diagram
commutes. The commutativity of the right triangles is witnessed by natural
isomorphisms
\[
\lambda\colon h_! \Raiso f_!g_!
\qquad\text{and}\qquad 
\mu\colon h^* \Raiso g^*f^*,
\]
where $\mu$ is monoidal. We make the following assumption:
\begin{enumerate}[resume*=abstract]
\item\label{Assumption6} The functors $\alpha^*$, $\beta^*$, and $\gamma^*$
admit left adjoints:
\[
\alpha_{(1)}\dashv \alpha^*,\qquad \beta_{(1)}\dashv \beta^*,\qquad 
\gamma_{(1)}\dashv \gamma.
\]
\end{enumerate}
Assume moreover that there are natural isomorphisms
\[
\begin{matrix}
\pf_f\colon & f_!\bigl(b\otimes f^*(a)\bigr) &\raiso& f_!(b)\otimes a,\\
\pf_g\colon & g_!\bigl(c\otimes g^*(b)\bigr) &\raiso& g_!(c)\otimes b,\\
\pf_h\colon & h_!\bigl(c\otimes h^*(a)\bigr) &\raiso& h_!(c)\otimes a,
\end{matrix}
\]
which satisfy the following conditions:
\begin{enumerate}[resume*=abstract]
\item\label{Assumption7} For all $a\in\cat A$, $c\in \cat C$ the diagram
\[
\begin{tikzcd}
h_!\bigl(c\otimes h^*(a)\bigr)
\ar[rr,"\pf_h"] \ar[d,"\lambda"']
& &
h_!(c)\otimes a
\ar[dd,"\lambda\otimes\id"]
\\
f_!g_!\bigl(c\otimes h^*(a)\bigr)
\ar[d,"f_!g_!(\id\otimes\mu)"']
\\
f_!g_!\bigl(c\otimes g^*f^*(a)\bigr)
\ar[r,"f_!\pf_g"']
&
f_!\bigl(g_!(c)\otimes f^*(a)\bigr)
\ar[r,"\pf_f"']
&
f_!g_!(c)\otimes a
\end{tikzcd}
\]
is commutative.

\item\label{Assumption8} For all $a',a\in \cat A$, $b\in \cat B$ the diagram
\[
\begin{tikzcd}
f_!\bigl((b\otimes f^*(a'))\otimes f^*(a)\bigr)
\ar[d,"f_!\assoc"',"\cong"] \ar[r,"\pf_f"]
&[2em]
f_!\bigl(b\otimes f^*(a')\bigr)\otimes a
\ar[r,"\pf_f\otimes\id"]
&
\bigl(f_!(b)\otimes a'\bigr)\otimes a
\ar[d,"\assoc","\cong"']
\\
f_!\bigl(b\otimes (f^*(a')\otimes f^*(a))\bigr)
\ar[r,"f_!(\id\otimes\mon_f)"']
&
f_!\bigl(b\otimes f^*(a'\otimes a)\bigr)
\ar[r,"\pf_f"']
&
f_!(b)\otimes (a'\otimes a)
\end{tikzcd}
\]
is commutative, and similarly with $f$ replaced by $g$. Here, $\assoc$ denotes
the associativity constraint in the respective monoidal categories.
\end{enumerate}

Finally, let $\wt{\pf}_g$ be the unique natural isomorphism making the
diagram
\[
\begin{tikzcd}
g_!\bigl(g^*(b)\otimes c\bigr)
\ar[r,"\wt{\pf}_g"] \ar[d,"\sym"',"\cong"]
&
b\otimes g_!(c)
\ar[d,"\sym","\cong"']
\\
g_!\bigl(c\otimes g^*(b)\bigr)
\ar[r,"\pf_g"']
&
g_!(c)\otimes b
\end{tikzcd}
\]
commutative, where $\sym$ denotes the symmetry constraints in the respective
symmetric monoidal categories.

\begin{lem}\label{lem:abstract-4} 
For all $b\in \cat B$ and $c',c\in \cat C$ the following diagram commutes:
\[
\begin{tikzcd}[font=\small,labels={font=\tiny},column sep=3em]
\alpha_{(1)}h_!\bigl((g^*(b)\otimes c')\otimes c\bigr)
\ar[d,"\alpha_{(1)}\lambda"'] 
\ar[r,"\ltm_{\alpha,h}"]
&
\alpha_{(1)}h_!\bigl(g^*(b)\otimes c'\bigr)\otimes \gamma_{(1)}(c)
\ar[r,"\alpha_{(1)}\lambda \otimes\id"]
&
\alpha_{(1)}f_!g_!\bigl(g^*(b)\otimes c'\bigr)\otimes \gamma_{(1)}(c)
\ar[d,"\alpha_{(1)}f_!\wt{\pf}_g\otimes\id"]
\\
\alpha_{(1)}f_!g_!\bigl((g^*(b)\otimes c')\otimes c\bigr)
\ar[d,"\alpha_{(1)}f_!g_!\assoc"']
& &
\alpha_{(1)}f_!\bigl(b\otimes g_!(c')\bigr)\otimes \gamma_{(1)}(c)
\ar[d,"\ltm_{\alpha,f}\otimes\id"]
\\
\alpha_{(1)}f_!g_!\bigl(g^*(b)\otimes (c'\otimes c)\bigr)
\ar[d,"\alpha_{(1)}f_!\wt{\pf}_g"']
& &
\bigl(\alpha_{(1)}f_!(b)\otimes \beta_{(1)}g_!(c')\bigr)\otimes \gamma_{(1)}(c)
\ar[d,"\assoc"]
\\
\alpha_{(1)}f_!\bigl(b\otimes g_!(c'\otimes c)\bigr)
\ar[r,"\ltm_{\alpha,f}"']
&
\alpha_{(1)}f_!(b)\otimes \beta_{(1)}g_!(c'\otimes c)
\ar[r,"\id\otimes\ltm_{\beta,g}"']
&
\alpha_{(1)}f_!(b)\otimes \bigl(\beta_{(1)}g_!(c')\otimes \gamma_{(1)}(c)\bigr),
\end{tikzcd}
\]
where $\ltm_{\alpha,h}$, $\ltm_{\alpha,f}$, and $\ltm_{\beta,g}$ are defined as
in \S\ref{sss:setup2}.
\end{lem} 
\begin{hideproof} 
It suffices to show that the diagram
\[ 
\hspace{-6em}
\begin{tikzcd}[proof]
\alpha_{(1)}h_!\bigl((g^*(b)\otimes c') \otimes h^*\alpha^*(d)\bigr)
\ar[d,"\alpha_{(1)}\lambda(\id\otimes\mu)"']
\ar[rr,"\alpha_{(1)}\pf_h"]
& &
\alpha_{(1)}\bigl(h_!(g^*(b)\otimes c')\otimes \alpha^*(d)\bigr)
\ar[r,"\lpf_\alpha"]
\ar[d,"\alpha_{(1)}(\lambda\otimes\id)"]
&
\alpha_{(1)}h_!\bigl(g^*(b)\otimes c')\otimes d
\ar[d,"\alpha_{(1)}\lambda\otimes \id"]
\\
\alpha_{(1)}f_!g_!\bigl((g^*(b)\otimes c')\otimes g^*f^*\alpha^*(d)\bigr)
\ar[d,"\alpha_{(1)}f_!g_!\assoc"']
\ar[r,"\alpha_{(1)}f_!\pf_g"]
\ar[ddr,phantom,"\ref{lwirth-trans-I}"]
&
\alpha_{(1)}f_!\bigl(g_!(g^*(b)\otimes c')\otimes f^*\alpha^*(d)\bigr)
\ar[d,"\alpha_{(1)}f_!(\wt{\pf}_g\otimes\id)"]
\ar[r,"\alpha_{(1)}\pf_f"]
&
\alpha_{(1)}\bigl(f_!g_!(g^*(b)\otimes c')\otimes \alpha^*(d)\bigr)
\ar[r,"\lpf_\alpha"]
\ar[dd,"\alpha_{(1)}(f_!\wt{\pf}_g\otimes\id)"]
&
\alpha_{(1)}f_!g_!\bigl(g^*(b)\otimes c'\bigr)\otimes d
\ar[dd,"\alpha_{(1)}f_!\wt{\pf}_g\otimes \id"]
\\
\alpha_{(1)}f_!g_!\bigl(g^*(b)\otimes (c'\otimes g^*\beta^*(d))\bigr)
\ar[d,"\alpha_{(1)}f_!\wt{\pf}_g"']
&
\alpha_{(1)}f_!\bigl((b\otimes g_!(c'))\otimes \beta^*(d)\bigr)
\ar[d,"\alpha_{(1)}f_!\assoc"']
\ar[dr,"\alpha_{(1)}\pf_f"]
\\
\alpha_{(1)}f_!\bigl(b\otimes g_!(c'\otimes g^*\beta^*(d))\bigr)
\ar[d,"\ltm_{\alpha,f}"']
\ar[r,"\alpha_{(1)}f_!(\id\otimes \pf_g)"]
&
\alpha_{(1)}f_!\bigl(b\otimes (g_!(c')\otimes \beta^*(d))\bigr)
\ar[d,"\ltm_{\alpha,f}"']
&
\alpha_{(1)}\bigl(f_!(b\otimes g_!(c'))\otimes \alpha^*(d)\bigr)
\ar[r,"\lpf_\alpha"]
\ar[d,phantom, "\ref{lwirth-trans-II}"]
&
\alpha_{(1)}f_!\bigl(b\otimes g_!(c')\bigr)\otimes d
\ar[d,"\ltm_{\alpha,f}\otimes\id"]
\\
\alpha_{(1)}f_!(b)\otimes \beta_{(1)}g_!\bigl(c'\otimes g^*\beta^*(d)\bigr)
\ar[r,"\id\otimes\beta_{(1)}\pf_g"']
&
\alpha_{(1)}f_!(b)\otimes \beta_{(1)}\bigl(g_!(c')\otimes \beta^*(d)\bigr)
\ar[r,"\id\otimes\lpf_\beta"']
&
\alpha_{(1)}f_!(b)\otimes \bigl(\beta_{(1)}g_!(c')\otimes d\bigr)
&
\bigl(\alpha_{(1)}f_!(b)\otimes \beta_{(1)}g_!(c')\bigr)\otimes d
\ar[l,"\assoc"]
\end{tikzcd}
\] 
commutes, because then passing to the left mates at $\lpf_\alpha
\circ\alpha_{(1)}\pf_h \colon \alpha_{(1)}h_!\bigl((g^*(b)\otimes c')\otimes
\blank) h^*\alpha^* \To \alpha_{(1)}h_!\bigl(g^*(b)\otimes c'\bigr)\otimes
\blank$ shows that the diagram in the lemma commutes. The upper left rectangle
commutes by \ref{Assumption7}, and the
commutativity of all squares except \ref{lwirth-trans-I} and
\ref{lwirth-trans-II} is clear. 
\begin{enumerate}[label=Ad (\Roman*), ref=(\Roman*)]
\item\label{lwirth-trans-I} We need to show that the outer diagram
\[ 
\begin{tikzcd}[proof]
g_!\bigl((g^*(b)\otimes c')\otimes g^*(b')\bigr)
\ar[d,"g_!(s\otimes\id)"']
\ar[r,"\pf_g"]
\ar[ddddd,start anchor=south west, end anchor=north west,bend
right,"g_!(\assoc)"']
&
g_!\bigl(g^*(b)\otimes c'\bigr)\otimes b'
\ar[d,"g_!s\otimes \id"]
\ar[r,"\wt{\pf}_g\otimes\id"]
&
\bigl(b\otimes g_!(c')\bigr)\otimes b'
\ar[d,"s\otimes \id"]
\ar[ddddd,start anchor=south east, end anchor=north east, bend left, "\assoc"]
\\
g_!\bigl((c'\otimes g^*(b)\bigr)\otimes g^*(b')\bigr)
\ar[d,"g_!\assoc"']
\ar[r,"\pf_g"]
&
g_!\bigl(c'\otimes g^*(b)\bigr)\otimes b'
\ar[r,"\pf_g\otimes\id"]
\ar[d,phantom,"(*)"{font=\tiny}]
&
\bigl(g_!(c')\otimes b\bigr)\otimes b'
\ar[d,"\assoc"]
\\
g_!\bigl(c'\otimes (g^*(b)\otimes g^*(b'))\bigr)
\ar[d,"g_!(\id\otimes s)"']
\ar[r,"g_!(\id\otimes\mon_g)"]
&
g_!\bigl(c'\otimes g^*(b\otimes b')\bigr)
\ar[d,"g_!(\id\otimes g^*s)"]
\ar[r,"\pf_g"]
&
g_!(c')\otimes (b\otimes b')
\ar[d,"\id\otimes s"]
\\
g_!\bigl(c'\otimes (g^*(b')\otimes g^*(b))\bigr)
\ar[r,"g_!(\id\otimes\mon_g)"]
&
g_!\bigl(c'\otimes g^*(b'\otimes b)\bigr)
\ar[r,"\pf_g"']
\ar[d,phantom,"(**)"{font=\tiny}]
&
g_!(c')\otimes (b'\otimes b)
\\
g_!\bigl((c'\otimes g^*(b'))\otimes g^*(b)\bigr)
\ar[u,"g_!\assoc"]
\ar[r,"\pf_g"]
&
g_!\bigl(c'\otimes g^*(b')\bigr)\otimes b
\ar[r,"\pf_g\otimes\id"]
&
\bigl(g_!(c')\otimes b'\bigr)\otimes b
\ar[u,"\assoc"']
\\
g_!\bigl(g^*(b)\otimes (c'\otimes g^*(b'))\bigr)
\ar[u,"g_!s"]
\ar[r,"\pf_g"']
&
b\otimes g_!\bigl(c'\otimes g^*(b')\bigr)
\ar[u,"s"]
\ar[r,"\id\otimes\pf_g"']
&
b\otimes \bigl(g_!(c')\otimes b'\bigr)
\ar[u,"s"']
\end{tikzcd}
\] 
commutes, where we have put $b'\coloneqq f^*\alpha^*(d)$. It is clear that all
the small squares commute; for example, the left square in
the middle commutes, because $g$ is a symmetric monoidal functor. The diagrams
marked $(*)$ and
$(**)$ commute by \ref{Assumption8}. The diagrams on the far left and far right
commute by the hexagon identity.

\item\label{lwirth-trans-II} After passing to the right mate at
$\id\otimes\lpf_\beta$ or $\ltm_{\alpha,f}\otimes \id$, and using $\beta^* =
f^*\alpha^*$, it suffices to show that the following diagram commutes:
\[
\begin{tikzcd}[proof]
\alpha_{(1)}f_!\bigl(b\otimes (f^*\alpha^*(d')\otimes f^*\alpha^*(d))\bigr)
\ar[r,"\alpha_{(1)}f_!(\id\otimes \mon_f)"]
&
\alpha_{(1)}f_!\bigl(b\otimes f^*(\alpha^*(d')\otimes \alpha^*(d))\bigr)
\ar[r,"\alpha_{(1)}f_!(\id\otimes f^*\mon_\alpha)"]
\ar[d,"\alpha_{(1)}\pf_f"]
&
\alpha_{(1)}f_!\bigl(b\otimes f^*\alpha^*(d'\otimes d)\bigr)
\ar[d,"\alpha_{(1)}\pf_f"]
\\
\alpha_{(1)}f_!\bigl((b\otimes f^*\alpha^*(d'))\otimes f^*\alpha^*(d)\bigr)
\ar[u,"\alpha_{(1)}f_!\assoc"]
\ar[d,"\alpha_{(1)}\pf_f"']
&
\alpha_{(1)}\bigl(f_!(b)\otimes (\alpha^*(d')\otimes \alpha^*(d))\bigr)
\ar[r,"\alpha_{(1)}(\id\otimes\mon_\alpha)"]
&
\alpha_{(1)}\bigl(f_!(b)\otimes \alpha^*(d'\otimes d)\bigr)
\ar[d,"\lpf_\alpha"]
\\
\alpha_{(1)}\bigl(f_!(b\otimes f^*\alpha^*(d'))\otimes \alpha^*(d)\bigr)
\ar[d,"\lpf_\alpha"']
\ar[r,"\alpha_{(1)}(\pf_f\otimes\id)"]
&
\alpha_{(1)}\bigl((f_!(b)\otimes \alpha^*(d'))\otimes \alpha^*(d)\bigr)
\ar[u,"\alpha_{(1)}\assoc"]
\ar[d,"\lpf_\alpha"']
&
\alpha_{(1)}f_!(b)\otimes (d'\otimes d)
\\
\alpha_{(1)}f_!\bigl(b\otimes f^*\alpha^*(d')\bigr)\otimes d
\ar[r,"\alpha_{(1)}\pf_f\otimes\id"']
&
\alpha_{(1)}\bigl(f_!(b)\otimes \alpha^*(d')\bigr) \otimes d
\ar[r,"\lpf_\alpha\otimes\id"']
&
\bigl(\alpha_{(1)}f_!(b)\otimes d'\bigr)\otimes d
\ar[u,"\assoc"']
\end{tikzcd}
\]
The lower left square commutes by naturality of $\lpf_\alpha$, and the upper
right square by the naturality of $\alpha_{(1)}\pf_f$. The upper left rectangle
commutes by \ref{Assumption8}. For the lower right rectangle, we pass to the
right mate at, say, $\lpf_\alpha\colon \alpha_{(1)}(-\otimes\alpha^*(d'\otimes
d)) \Rightarrow (-\otimes (d'\otimes d))\alpha_{(1)}$, and then it suffices to
show that the diagram
\[
\begin{tikzcd}[proof]
\bigl(\alpha^*(d'')\otimes \alpha^*(d')\bigr)\otimes \alpha^*(d)
\ar[r,"\mon_\alpha\otimes\id"]
\ar[d,"\assoc"']
&
\alpha^*(d''\otimes d')\otimes \alpha^*(d)
\ar[r,"\mon_\alpha"]
&
\alpha^*\bigl((d''\otimes d')\otimes d\bigr)
\ar[d,"\alpha^*\assoc"]
\\
\alpha^*(d'')\otimes \bigl(\alpha^*(d')\otimes \alpha^*(d)\bigr)
\ar[r,"\id\otimes\mon_\alpha"']
&
\alpha^*(d'')\otimes \alpha^*(d'\otimes d)
\ar[r,"\mon_\alpha"']
&
\alpha^*\bigl(d''\otimes (d'\otimes d)\bigr).
\end{tikzcd}
\]
is commutative. But this follows from the fact that $\alpha$ is a monoidal
functor. Thus, we deduce that \ref{lwirth-trans-II} commutes.
\end{enumerate}
\end{hideproof} 

\bibliographystyle{alphaurl}
\bibliography{../references}{}
\end{document}